\documentclass[11pt,a4paper,reqno]{amsart}
\usepackage[T1]{fontenc}
\usepackage[utf8]{inputenc}
\usepackage{lmodern}
\usepackage[a4paper,margin=28mm]{geometry}
\usepackage{amsmath,amssymb,amsthm,mathtools,mathrsfs}
\usepackage{microtype,booktabs,longtable,listings}
\usepackage[unicode,hidelinks]{hyperref}
\usepackage{bookmark}
\newtheorem{theorem}{Theorem}[section]
\newtheorem{lemma}[theorem]{Lemma}
\newtheorem{proposition}[theorem]{Proposition}
\newtheorem{corollary}[theorem]{Corollary}
\theoremstyle{remark}

\numberwithin{equation}{section}
\allowdisplaybreaks[2]
\newcommand{\qb}[2]{\begin{bmatrix}#1\\#2\end{bmatrix}_q}
\newcommand{\N}{\mathfrak N}
\newcommand{\E}{\mathcal E_9}
\newcommand{\Ec}{\mathcal E_3}
\newcommand{\br}[1]{\langle #1\rangle}
\DeclareMathOperator{\ord}{ord}

\DeclareMathOperator{\diag}{diag}
\usepackage{xcolor}

\newcommand{\Nzero}{\mathbb N_0}

\newcommand{\Jbar}{\overline{J}}

\newcommand{\Ai}{\operatorname{Ai}}
\newcommand{\thp}{\vartheta^{+}}

\title[Eighteen Kanade--Russell related identities]{Proofs of eighteen conjectural identities related to the modulo nine Kanade--Russell identities}
\author{Ernest X. W. Xia}
\address{School of Mathematical Sciences, Suzhou University of Science and
Technology, Suzhou 215009, Jiangsu, P. R. China}
\email{ernestxwxia@163.com}
\subjclass[2020]{Primary 11P84; Secondary 33D15, 11F03, 05A17}
\keywords{Kanade--Russell identities, reflections, Schur polynomials,
Nahm sums, divided differences, $q$-Airy functions}
\date{September 19, 2026}
\hypersetup{pdftitle={Proofs of eighteen conjectural identities related to the
		modulo nine Kanade--Russell identities},
pdfauthor={Ernest X. W. Xia}}

\begin{document}

\begin{abstract}
	Uncu and Zudilin studied reflections of finite forms of the
	modulo nine Kanade--Russell identities under $q\mapsto q^{-1}$,
	leading to ten conjectural sum--product evaluations, including
	two proposed by Warnaar.
	Hickerson conjectured four further identities in his study of
	contiguous relations for the associated parameterized double
	sums. By constructing finite forms of Hickerson's sums and
	studying their reflections, Konenkov proposed two additional
	product identities, whose conjugates give two more.
	In this paper, we prove all eighteen conjectural identities by combining
	Nahm-sum dual evaluations established by the author with parameterized
	divided-difference identities. For the symmetric reflections, Schur-type
	decompositions and finite rational identities reduce the required
	evaluations to nonsingular two-by-two linear systems. The asymmetric and
	cyclotomic cases are treated through $q$-Airy--theta connection formulas,
	whereas the Hickerson identities are derived from low-order divided
	differences.
\end{abstract}

\maketitle

\section{Introduction}\label{sec:intro}

\allowdisplaybreaks

Two of the most important results in the theory of $q$-series are the classical 
Rogers--Ramanujan identities which state that
\begin{align}
 G(q):=\sum_{n\ge0}\frac{q^{n^2}}{(q;q)_n}
 &=\frac1{(q,q^4;q^5)_\infty},\label{eq:intro-RR1}\\
H(q):= \sum_{n\ge0}\frac{q^{n^2+n}}{(q;q)_n}
 &=\frac1{(q^2,q^3;q^5)_\infty}.\label{eq:intro-RR2}
\end{align}
Whenever analytic convergence is invoked, we assume that
$q\in\mathbb{C}$ with $|q|<1$. In formal power-series arguments,
$q$ is regarded as an indeterminate. We use the standard
$q$-series notation:
\[
(a;q)_0:=1,\qquad
(a;q)_n:=\prod_{k=0}^{n-1}(1-aq^k),\qquad
(a;q)_\infty:=\prod_{k=0}^{\infty}(1-aq^k),
\]
and
\begin{equation*}
	(a_1,\ldots,a_m;q)_n:=\prod_{j=1}^m(a_j;q)_n,
	\qquad n\in\Nzero\cup\{\infty\},
\end{equation*}
 where  $\mathbb C$ is the set 
 of complex numbers, $\mathbb C^*:=\mathbb C\setminus\{0\}$
 is  the set of nonzero complex numbers and   $\Nzero$ is the set of nonnegative integers.
In addition,
the product notation used below is
\begin{align*}
	J_m&:=(q^m;q^m)_\infty,\quad
	J_{a,m}:=(q^a,q^{m-a},q^m;q^m)_\infty,
	\quad \Jbar_{a,m}:=(-q^a,-q^{m-a},q^m;q^m)_\infty.
\end{align*}
Here $m$ is a positive integer and $a $ is an integer.

The two identities were first proved by Rogers~\cite{Rogers} in 1894 and were
independently rediscovered by Ramanujan \cite{RamanujanCollected}  sometime 
 before 1913.
Ramanujan communicated them, without proof, to Hardy in a letter from
India. 

 These identities connect apparently different mathematical objects.
After multiplication by suitable powers of $q$, their products belong
to the theory of modular functions, allowing transformations and cusp
behavior to be studied by complex analysis. In representation theory,
Lepowsky and Wilson~\cite{LW} explained them through affine Lie algebra
modules and vertex operators: character formulas and suitable bases
provide two descriptions of the same graded dimensions. In statistical
mechanics, Baxter's hard-hexagon model~\cite{Baxter} gives a setting in
which related series and products occur in exact thermodynamic
calculations. These connections have led to
numerous generalizations, commonly referred to as
\emph{Rogers--Ramanujan type identities}. A classical milestone is
Slater's list of $130$ such identities~\cite{Slater}.
The passage from the classical identities to multisum identities is
illustrated by the analytic Andrews--Gordon identities~\cite{AndrewsAG,gordon}.
For $k\ge2$ and
$1\le i\le k$,  
\[
 \sum_{n_1,\ldots,n_{k-1}\ge0}
 \frac{q^{(n_1^+)^2+\cdots+(n_{k-1}^+)^2+n_i^++\cdots+n_{k-1}^+}}
 {(q;q)_{n_1}\cdots(q;q)_{n_{k-1}}}
 =\frac{(q^i,q^{2k+1-i},q^{2k+1};q^{2k+1})_\infty}{(q;q)_\infty},
\]
where   $n_j^+:=n_j+\cdots+n_{k-1}$. 
The linear sum in the exponent is empty when $i=k$. At $k=2$, this
recovers \eqref{eq:intro-RR1} and \eqref{eq:intro-RR2}.
More general Rogers--Ramanujan type multisums involve a quadratic form
in several indices and factorials with different bases. Subsequent
developments have considerably broadened the subject: in particular,
many Rogers--Ramanujan type identities now involve double, triple, or
more general multiple $q$-hypergeometric sums whose values are given
by infinite products, see, for example \cite{cao-rosengren-wang,milas-wang,Rosengren,ShiWang,sills,wang-wang-rank-three,wangwang,wang-rank-two,wang-rank-three,wang-duality,zagier}.

The  importance of the Rogers--Ramanujan identities  is visible already in their partition interpretation
given by Schur~\cite{schur}.
The first equates partitions with consecutive parts differing by at
least two to partitions into parts congruent to $1$ or $4$ modulo $5$.
For the second, the difference condition is supplemented by a smallest
part of at least two, and the permitted residues are $2$ and $3$.
Thus a local restriction on a partition can have the same generating
function as a restriction on the residue classes of its parts.
Motivated by such partition identities arising from Rogers--Ramanujan identities,   Kanade and Russell~\cite{KRfinder,KRstairs} initiated the study of the 
modulo nine identities  through computational searches
using their Maple package \texttt{IdentityFinder}. 
Their original paper proposed four modulo nine partition
conjectures, and a fifth companion was subsequently recorded
in Russell's doctoral thesis.
Kur\c{s}ung\"oz~\cite{Kursungoz} obtained Andrews--Gordon-type
double-sum representations for the first four conjectures,
while Uncu and Zudilin~\cite{UZ} supplied a double-sum
representation for the fifth.  Tsuchioka~\cite{Tsuchioka} gave a vertex
operator reformulation of the modulo nine conjecture. 
Recently, Xia~\cite{Xia} proved all five modulo nine
Kanade--Russell sum--product identities, together with
related generalized Nahm-sum dual identities.
For further results on  Rogers--Ramanujan type identities related to Kanade--Russell multisums and generalized Nahm sums, see, for example, \cite{BJM,ChernLi,mizuno-1,Mizuno,Rosengren,Wang12}.

The use of the reciprocal substitution $q\mapsto q^{-1}$
to obtain related identities has a substantial history.
Finite forms of the Rogers--Ramanujan identities, originating
in Schur's work and developed further by Andrews \cite{AndrewsFinite,AndrewsHexagon}, provide
polynomial identities to which this substitution can be
applied.
After multiplication by a suitable power of $q$, one may
take limits to obtain different infinite-series identities.
Andrews \cite{AndrewsFinite,AndrewsHexagon} developed this approach in connection with Baxter's
hard-hexagon model, revealing duality relations between
families of Rogers--Ramanujan type identities.
A related development was given by Mortenson
\cite{MortensonDual}, who studied the dual nature of partial
theta functions and Appell--Lerch sums.

Uncu and Zudilin~\cite{UZ} constructed finite versions of the five
modulo nine Kanade--Russell sums and investigated their reflected
limits. According as the truncation parameter is congruent to \(0\), \(1\), or
\(2\pmod 3\), the five finite identities yield fifteen reflected limits.   They proved five relations among these limits and proposed
ten remaining sum-product evaluations, including 
  two conjectured by Warnaar and recorded by Uncu and Zudilin~\cite{UZ}. The six symmetric
candidates have modulus-$45$ product combinations; the four asymmetric
candidates have individual product expressions. These are
$F_1,\ldots,F_{10}$ below. 
Konenkov  extended this program to four identities
conjectured by Hickerson, which  
  were recorded  in \cite[Section 3]{Konenkov}. He obtained finite forms and reflected
limits, proved four linear relations between the latter, and proposed
two cyclotomic product evaluations. Applying the coefficient-field
conjugation gives two additional evaluations.
  Konenkov's four proved linear relations are
not the same statements as these four product candidates.

Before stating the main results, we fix the notation and enumerate the
eighteen sum sides.  For \(n,r\in\mathbb Z\), the \(q\)-binomial
coefficient (or Gaussian polynomial) is defined by
\begin{equation*}
	\begin{bmatrix}
		n\\ r
	\end{bmatrix}_{q}
	:=
	\begin{cases}
		\displaystyle
		\frac{(q;q)_n}{(q;q)_r(q;q)_{n-r}},
		& 0\leq r\leq n,\\[8pt]
		0,
		& \text{otherwise}.
	\end{cases}
\end{equation*}

Define
\begin{equation}\label{eq:Q}
	\mathcal Q_\pm(r,s):=r^2\pm3rs+3s^2
	=\left(r\pm\frac32s\right)^2+\frac34s^2.
\end{equation} 
and 
\begin{align}
\mathcal S(a,b)&:=\sum_{r,s\ge0}
 \frac{q^{\mathcal Q_+(r,s)+ar+bs}}{(q;q)_r(q^3;q^3)_s},\notag\\
\N(a,b)&:=\sum_{r,s\ge0}
 \frac{q^{\mathcal Q_-(r,s)+ar+bs}}{(q;q)_r(q^3;q^3)_s},\label{Ndef}\\
\mathcal T_{a,b}^{(c)}&:=\sum_{r,s\ge0}
 \frac{q^{\mathcal Q_-(r,s)+ar+bs}}{(q^3;q^3)_s}
 \qb{3s-r+c}{r},\label{Tdef}
\end{align}
where  $a,b,c$ are integer parameters.

We denote the ten reflected sums of Uncu and Zudilin \cite{UZ} by
$F_1,\ldots,F_{10}$:
\begin{align}
F_1&:=\sum_{r,s\ge0}\frac{q^{\mathcal Q_-(r,s)-1}}{(q^3;q^3)_s}
                    \qb{3s-r-1}{r},\notag\\
F_2&:=\sum_{r,s\ge0}\frac{q^{\mathcal Q_-(r,s)}}{(q^3;q^3)_s}
                    \qb{3s-r}{r},\notag\\
F_3&:=\sum_{r,s\ge0}\frac{q^{\mathcal Q_-(r,s)-r+3s}}{(q^3;q^3)_s}
                    \qb{3s-r}{r},\notag\\
F_4&:=\sum_{r,s\ge0}\frac{q^{\mathcal Q_-(r,s)-r+3s}}{(q^3;q^3)_s}
                    \qb{3s-r+1}{r},\notag\\
F_5&:=1+\sum_{r,s\ge0}\frac{q^{\mathcal Q_-(r,s)+r}}{(q^3;q^3)_s}
                    \qb{3s-r-2}{r},\notag\\
F_6&:=\sum_{r,s\ge0}\frac{q^{\mathcal Q_-(r,s)+r-2}}{(q^3;q^3)_s}
                    \qb{3s-r-1}{r},\notag\\
F_7&:=\sum_{r,s\ge0}\frac{q^{\mathcal Q_-(r,s)+s}}{(q^3;q^3)_s}
\qb{3s-r}{r},\label{eq:F7}\\
F_8&:=\sum_{r,s\ge0}\frac{q^{\mathcal Q_-(r,s)+s-2}}{(q^3;q^3)_s}
\qb{3s-r-1}{r},\label{eq:F8}\\
F_9&:=(1+q)\sum_{r,s\ge0}\frac{q^{\mathcal Q_-(r,s)+2s}}{(q^3;q^3)_s}
\qb{3s-r}{r}+\sum_{r,s\ge0}\frac{q^{\mathcal Q_-(r,s)-r+5s+2}}{(q^3;q^3)_s}
\qb{3s-r+1}{r},\label{eq:F9}\\
F_{10}&:=(1+q)\sum_{r,s\ge0}\frac{q^{\mathcal Q_-(r,s)+2s-2}}{(q^3;q^3)_s}
\qb{3s-r-1}{r} +\sum_{r,s\ge0}\frac{q^{\mathcal Q_-(r,s)-r+5s}}{(q^3;q^3)_s}
\qb{3s-r}{r}.\label{eq:F10}
\end{align}
The Gaussian range and \eqref{eq:Q} imply coefficientwise convergence.
The nonzero summands have nonnegative exponents, including in the
formulas with an apparent negative shift.

The product evaluations of $F_7,F_8$ were conjectured by Warnaar
in terms of reflected limits and written as double-sum identities by
Uncu and Zudilin~\cite[(21)--(22)]{UZ}. The other eight evaluations
were conjectured by Uncu and Zudilin~\cite[(25)--(26), (28)--(33)]{UZ}.
For compact product notation, put
\begin{align*}
\br{a,b,c,d}&:=\frac{J_{45}}{J_3
(q^a,q^{45-a},q^b,q^{45-b},q^c,q^{45-c},q^d,q^{45-d};q^{45})_\infty}
\nonumber\\
&=\frac{J_{45}^5}{J_3J_{a,45}J_{b,45}J_{c,45}J_{d,45}}. 
\end{align*}

\begin{theorem}\label{thm:main}
The following six identities hold:
\begin{align}
F_1&=\br{2,8,11,20}+q^3\br{2,14,20,22}
                       -q^8\br{17,19,20,22},\label{eq:R1}\\
F_2&=\br{1,7,11,20}+q^6\br{11,13,14,20}
                       -q^6\br{8,14,19,20},\label{eq:R2}\\
F_3&=\br{2,5,14,22}-q^2\br{5,7,16,17}
                       -q^5\br{5,17,19,22},\label{eq:R3}\\
F_4&=\br{2,5,16,19}-q^2\br{5,8,14,19}
                       +q^2\br{5,11,13,14},\label{eq:R4}\\
F_5&=\br{4,7,10,13}-q^4\br{7,10,16,17}
                       -q^7\br{10,17,19,22},\label{eq:R5}\\
F_6&=\br{2,10,16,19}+q\br{4,10,16,22}
                       -q^2\br{8,10,14,19}.\label{eq:R6}
\end{align}
\end{theorem}

\begin{theorem}\label{thm:asym}
The four asymmetric reflected sums satisfy
\begin{align}
F_7&=\frac1{(q^2;q^3)_\infty(q^3;q^9)_\infty
 (q^9,q^{36};q^{45})_\infty},\label{eq:R7}\\
F_8&=\frac1{(q^2;q^3)_\infty(q^3;q^9)_\infty
 (q^{18},q^{27};q^{45})_\infty},\label{eq:R8}\\
F_9&=\frac1{(q;q^3)_\infty(q^6;q^9)_\infty
 (q^9,q^{36};q^{45})_\infty},\label{eq:R9}\\
F_{10}&=\frac1{(q;q^3)_\infty(q^6;q^9)_\infty
 (q^{18},q^{27};q^{45})_\infty}.\label{eq:R10}
\end{align}
\end{theorem}

Throughout this paper,  $\omega$ is   a primitive cube root and  so $\omega^2+\omega+1=0$.
We always use $\zeta\in\{\omega,\omega^2\}$.  
The remaining eight evaluations are the Hickerson and cyclotomic
families recorded by Konenkov~\cite{Konenkov}.
Set
\begin{align}
\mathcal H_{1,\zeta}&:=\mathcal S(1,1)-\zeta q\mathcal S(2,4),&
\mathcal H_{2,\zeta}&:=\mathcal S(0,-1)+\zeta^2\mathcal S(0,2),\label{h-51}\\
\mathcal K_{1,\zeta}&:=\mathcal T_{1,-1}^{(-1)}-\zeta\mathcal T_{0,2}^{(0)},&
\mathcal K_{2,\zeta}&:=\mathcal T_{1,-2}^{(-1)}+\zeta^2q\mathcal T_{-2,4}^{(2)}.\label{h-52}
\end{align}

\begin{theorem}[Hickerson's four identities]\label{Hthm}
For  each $\zeta\in\{\omega,\omega^2\}$,
\begin{align}
 \mathcal H_{1,\zeta}
 &=\frac{(q^6;q^9)_\infty(\zeta q,\zeta^2q^3;q^3)_\infty}
 {(q^2;q^3)_\infty},\label{H1}\\
 \mathcal H_{2,\zeta}
 &=-\zeta\frac{(q^3;q^9)_\infty(\zeta^2q^2,\zeta q^3;q^3)_\infty}
 {(q;q^3)_\infty}.\label{H2}
\end{align}
\end{theorem}

\begin{theorem}[Konenkov's four reflected identities]\label{Kthm}
For each $\zeta\in\{\omega,\omega^2\}$,
\begin{align}
\mathcal K_{1,\zeta}
 &=-\zeta\,
 \frac{(q^{15};q^{45})_\infty(\zeta q^3;q^3)_\infty (\zeta^2q,\zeta^2q^2,\zeta^2q^4,\zeta^2q^8,\zeta^2q^{10};q^{15})_\infty}
 {(q^5,q^{11},q^{14};q^{15})_\infty (q^3,q^{12},q^{18},q^{27};q^{45})_\infty}
  ,\label{K1}\\
\mathcal K_{2,\zeta}
 &=-\zeta\,
 \frac{(q^{30};q^{45})_\infty(\zeta^2q^3;q^3)_\infty (\zeta q^5,\zeta q^7,\zeta q^{11},\zeta q^{13},\zeta q^{14};q^{15})_\infty}
 {(q,q^4,q^{10};q^{15})_\infty (q^{18},q^{27},q^{33},q^{42};q^{45})_\infty}.\label{K2}
\end{align}
\end{theorem}

The paper is organized as follows. 
Section 2 collects the parameter recurrences, Schur identities, and
notation for divided differences used in several proofs. Sections 3--6
prove Theorems~\ref{thm:main}, \ref{thm:asym}, \ref{Hthm}, and
\ref{Kthm}, respectively. In each section, the auxiliary results are
established before the proof of the main theorem. The Airy identities
proved in Section 4 are also used in Section 6. The final section
summarizes the scope and remaining directions.

\section{Preliminaries}\label{sec:prelim}

For a Laurent series $f(q)$, let $\ord_q f$ be its least exponent, with
$\ord_q0=+\infty$. We write $f=O(q^L)$ when
$f\in q^L\mathbb C[[q]]$, componentwise for vectors. Here and throughout,    $\mathbb C[[q]]$ 
denotes the ring of formal power series in \(q\) with complex coefficients. 
 In addition, 
 we write $\mathbb{C}((q))$ for the set of formal Laurent
series in $q$ with complex coefficients:
\[
\mathbb{C}((q))
=
\left\{
\sum_{n=N}^{\infty} a_n q^n :
N\in\mathbb{Z},\ a_n\in\mathbb{C}
\right\}.
\]
Thus each element contains only finitely many terms with
negative powers of $q$, and no analytic convergence is assumed. 
 Positive quadratic
growth in \eqref{eq:Q} and uniform bounds for Gaussian coefficients
justify the coefficientwise limits below. The same estimates imply
local uniform convergence on compact subsets of $0<|q|<1$.

Moreover, for a formal power series $f(z)=\sum_{k\ge0}a_kz^k$ and $m\in \mathbb{N}_0$,
we write
\[
[z^m]f(z):=a_m
\]
for the coefficient of $z^m$ in $f(z)$.

Throughout this paper, define 
\begin{align*}
  \Ec&:=\frac{J_3}{J_1},\qquad \qquad \E:=\frac{J_9}{J_3},\qquad \qquad 
		\Delta:=G(q)G(q^9)+q^2H(q)H(q^9),\\
		X_1&:=(q,q^8;q^9)_\infty,\quad\qquad X_2:=(q^2,q^7;q^9)_\infty,\quad\qquad 
		X_3:=(q^4,q^5;q^9)_\infty,\\
		P_1&:=\frac{\E}{X_1},\  P_2:=\frac{\E}{X_2},\ 
		P_3:=\frac{\E}{X_3},\ 
		P_4:=\frac1{(q^2;q^3)_\infty(q^3;q^9)_\infty},\ 
		P_5:=\frac1{(q;q^3)_\infty(q^6;q^9)_\infty}.
\end{align*}

Since $\omega$ is a primitive cube root of unity, we have
\begin{equation}\label{cyclo}
 (1-x)(1-\omega x)(1-\omega^2x)=1-x^3.
\end{equation}

\begin{lemma} \label{lem:parameter-recurrences}
For fixed $b$, the two parameter families satisfy
\begin{align}
&\mathcal S(a+1,b)
=-q^{a+1-b}\mathcal S(a-2,b)
 +(q^{a+1-b}-q)\mathcal S(a-1,b)
 +(q^{2a-b}+q+1)\mathcal S(a,b),\label{Sfrec}\\
&q\N(a,b)=(1+q+q^{2a+b+2})\N(a+1,b)
 -(1-q^{a+2})\N(a+2,b)-q^{a+2}\N(a+3,b).\label{Nfrec}
\end{align}
\end{lemma}

\begin{proof}
	We use the contiguous relations
	\begin{align}
		\mathcal S(a,b)-\mathcal S(a+1,b)
		&=q^{a+1}\mathcal S(a+2,b+3),\label{Scont1}\\
		\mathcal S(a,b)-\mathcal S(a,b+3)
		&=q^{b+3}\mathcal S(a+3,b+6),\label{Scont2}\\
		\N(a,b)-\N(a+1,b)
		&=q^{a+1}\N(a+2,b-3),\label{cont1}\\
		\N(a,b)-\N(a,b+3)
		&=q^{b+3}\N(a-3,b+6),\label{cont2}
	\end{align}
	which are proved in \cite[(2.1), (2.2), (7.7), and (7.8)]{Xia};
	see also \cite{ChernLi,Konenkov}.

 We first derive \eqref{Sfrec}. Replacing \(a\) by \(a-2\) in
 \eqref{Scont1} gives
 \begin{equation}
 	\mathcal S(a,b+3)
 	=q^{1-a}\bigl\{\mathcal S(a-2,b)-\mathcal S(a-1,b)\bigr\}.
 	\label{eq:S-b-shift}
 \end{equation}
Replacing $a$  by \(a+1\) and \(a+2\) in \eqref{eq:S-b-shift}, respectively,
 gives
 \begin{align}
 	\mathcal S(a+1,b+3)
 	&=q^{-a}\bigl\{\mathcal S(a-1,b)-\mathcal S(a,b)\bigr\},
 	\label{eq:S-a1}\\
 	\mathcal S(a+2,b+3)
 	&=q^{-a-1}\bigl\{\mathcal S(a,b)-\mathcal S(a+1,b)\bigr\}.
 	\label{eq:S-a2}
 \end{align}
 Applying  \eqref{Scont1} at \((a+1,b+3)\), we have
 \begin{align}\label{h-91}
  \mathcal S(a+3,b+6)=q^{-a-2}\left(\mathcal S(a+1,b+3)-\mathcal S(a+2,b+3)\right)
    .
 \end{align}
 Substituting \eqref{eq:S-a1} and \eqref{eq:S-a2} into  \eqref{h-91} 
 yields
 \begin{equation}
 	\mathcal S(a+3,b+6)
 	=
 	q^{-2a-3}\bigl\{
 	q\mathcal S(a-1,b)
 	-(q+1)\mathcal S(a,b)
 	+\mathcal S(a+1,b)
 	\bigr\}.
 	\label{eq:S-double-shift}
 \end{equation}

 Substituting \eqref{eq:S-b-shift} and
 \eqref{eq:S-double-shift} into \eqref{Scont2}, we obtain
 \begin{align*}
 	&\mathcal S(a,b)
 	-q^{1-a}\bigl\{\mathcal S(a-2,b)-\mathcal S(a-1,b)\bigr\}\\
 	&\qquad
 	=q^{b-2a}\bigl\{
 	q\mathcal S(a-1,b)
 	-(q+1)\mathcal S(a,b)
 	+\mathcal S(a+1,b)
 	\bigr\}.
 \end{align*}
 Solving for \(\mathcal S(a+1,b)\) gives
  \eqref{Sfrec}.

  We next derive \eqref{Nfrec}. Applying \eqref{cont1} at
  \((a+1,b)\) gives
  \begin{equation}
  	\N(a+3,b-3)
  	=q^{-a-2}\bigl\{\N(a+1,b)-\N(a+2,b)\bigr\}.
  	\label{eq:N-first-shift}
  \end{equation}
  Applying \eqref{cont2} at \((a+3,b-3)\), we obtain
  \[
  \N(a+3,b-3)-\N(a+3,b)
  =q^b\N(a,b+3).
  \]
  Substitution of \eqref{eq:N-first-shift} gives
  \begin{equation}
  	\N(a,b+3)
  	=q^{-b}\left\{
  	q^{-a-2}\bigl(\N(a+1,b)-\N(a+2,b)\bigr)
  	-\N(a+3,b)
  	\right\}.
  	\label{Nup}
  \end{equation}

  Replacing \(a\) by \(a-2\) and \(a-1\) in \eqref{Nup}, respectively,
  gives
  \begin{align}
  	\N(a-2,b+3)
  &=q^{-b}\left\{
  	q^{-a}\bigl(\N(a-1,b)-\N(a,b)\bigr)
  	-\N(a+1,b)
  	\right\},
  	\label{eq:N-a-minus-2}
\\
  	\N(a-1,b+3)
  	&=q^{-b}\left\{
  	q^{-a-1}\bigl(\N(a,b)-\N(a+1,b)\bigr)
  	-\N(a+2,b)
  	\right\}.
  	\label{eq:N-a-minus-1}
  \end{align}
  Subtracting \eqref{eq:N-a-minus-1} from
  \eqref{eq:N-a-minus-2} and simplifying, we obtain
  \begin{align}
  	&q^{a+b+1}\left(\N(a-2,b+3)-\N(a-1,b+3)\right)\notag\\
   = & 
  	q\N(a-1,b)-(1+q)\N(a,b)
  	+(1-q^{a+1})\N(a+1,b)
  	+q^{a+1}\N(a+2,b)
  	. \label{eq:N-subtraction}
  \end{align}
  On the other hand, applying \eqref{cont1} at
  \((a-2,b+3)\) gives
  \[
  \N(a-2,b+3)-\N(a-1,b+3)
  =q^{a-1}\N(a,b).
  \]
  Consequently, \eqref{eq:N-subtraction} becomes
  \begin{align*}
  	q\N(a-1,b)
  	&=\bigl(1+q+q^{2a+b}\bigr)\N(a,b) 
  	-\bigl(1-q^{a+1}\bigr)\N(a+1,b)
  	-q^{a+1}\N(a+2,b).
  \end{align*}
  Finally, replacing \(a\) by \(a+1\) in the above identity, we obtain
  \eqref{Nfrec}.
  \end{proof}

\begin{lemma} \label{lem:schur}
For $m\ge0$, define
\begin{align}
\mathsf A_m&:=\sum_{r\ge0}q^{r^2-mr}\qb{m-r}{r}, \label{2-16}\\
\mathsf B_m&:=\sum_{r\ge0}q^{r^2-mr}\qb{m-r-1}{r}. \label{2-17}
\end{align}
Then
\begin{equation}\label{eq:schur}
\sum_{r\ge0}\frac{q^{r^2-mr}}{(q;q)_r}
=G(q)\mathsf A_m+H(q)\mathsf B_m.
\end{equation}
The recurrence extensions are $\mathsf A_{-1}:=0$, $\mathsf B_{-1}:=1$, and 
both sequences $\mathsf A_m$ and $\mathsf B_m$ satisfy 
\begin{equation}\label{eq:hrec}
\mathsf w_m=\mathsf w_{m-1}+q^{1-m}\mathsf w_{m-2}\qquad(m\ge1).
\end{equation}
\end{lemma}
\begin{proof}
The Gaussian recurrence
\[
\qb Nr=q^r\qb{N-1}{r}+\qb{N-1}{r-1}\qquad(N\ge1)
\]
and a shift $r\mapsto r+1$ prove~\eqref{eq:hrec} for $m\ge2$.
At the initial indices,
$(\mathsf A_0,\mathsf A_1)=(1,1)$ and
$(\mathsf B_0,\mathsf B_1)=(0,1)$.
For the single sum $\kappa_m$ on the left of~\eqref{eq:schur},
\begin{align*}
\kappa_m-\kappa_{m-1}&=\sum_{r\geq 0} \frac{q^{r^2-mr}(1-q^r)}{(q;q)_r}
=\sum_{r\ge1}\frac{q^{r^2-mr}}{(q;q)_{r-1}}
\nonumber\\
&=\sum_{r\ge0 }\frac{q^{(r+1)^2-m(r+1)}}{(q;q)_{r }}
=q^{1-m}\kappa_{m-2}.
\end{align*}
Here $\kappa_0=G(q)$, and the same difference calculation gives
$\kappa_1=G(q)+H(q)$.
The initial values and recurrence prove~\eqref{eq:schur}.
\end{proof}

For prescribed initial values $\mathsf w_{-1},\mathsf w_0$, let
$(\mathsf w_m)_{m\ge-1}$ be
the unique sequence defined by  \eqref{eq:hrec}, and define
\[
\mathscr S_{\mathsf w}(x):=\sum_{j\ge0}
\frac{q^{3j^2}\mathsf w_{3j}}{(q^3;q^3)_j}x^j,
\qquad
\mathscr T_{\mathsf w}(x):=\sum_{j\ge0}
\frac{q^{3j^2}\mathsf w_{3j-1}}{(q^3;q^3)_j}x^j.
\]

\begin{lemma} \label{lem:schur-transform}
	 One has  
\begin{align}
\mathscr S_{\mathsf w}(x)-\mathscr S_{\mathsf w}(q^3x)
&=x(q^3\mathscr S_{\mathsf w}(q^6x)
 +(q+q^2)\mathscr S_{\mathsf w}(q^3x)
 +q^3\mathscr T_{\mathsf w}(q^3x)+q\mathscr T_{\mathsf w}(x)),\label{eq:ST1}\\
\mathscr T_{\mathsf w}(x)-\mathscr T_{\mathsf w}(q^3x)
&=x(q^3\mathscr S_{\mathsf w}(q^6x)
 +q^2\mathscr S_{\mathsf w}(q^3x)
 +q^3\mathscr T_{\mathsf w}(q^3x)).\label{eq:ST2}
\end{align}
Moreover,
\[
\mathscr S_{\mathsf w}(0)=\mathsf w_0,\qquad
\mathscr T_{\mathsf w}(0)=\mathsf w_{-1},
\]
and these two constants together with \eqref{eq:ST1} and \eqref{eq:ST2}
determine the transforms uniquely as formal power series in $x$.
In particular,
\[
(\mathscr S_{\mathsf A}(0),\mathscr T_{\mathsf A}(0))=(1,0),\qquad
(\mathscr S_{\mathsf B}(0),\mathscr T_{\mathsf B}(0))=(0,1).
\]
\end{lemma}
\begin{proof}
	Successive applications of \eqref{eq:hrec} give, for $j\ge0$,
	\begin{align*}
		\mathsf w_{3j+2}
		&=(1+q^{-3j-1})\mathsf w_{3j}
		+q^{-3j}\mathsf w_{3j-1},\\
		\mathsf w_{3j+3}
		&=(1+q^{-3j-1}+q^{-3j-2})\mathsf w_{3j}
		+(q^{-3j}+q^{-6j-2})\mathsf w_{3j-1}.
	\end{align*}
	Using $(q^3;q^3)_j=(1-q^{3j})(q^3;q^3)_{j-1}$
	and shifting the summation index, we obtain
	\begin{align*}
		 \mathscr S_{\mathsf w}(x)-\mathscr S_{\mathsf w}(q^3x)&=\sum_{j\ge1}
		\frac{q^{3j^2}\mathsf w_{3j}x^j}{(q^3;q^3)_{j-1}}
		\nonumber\\
		&=x\sum_{j\ge0}
		\frac{q^{3(j+1)^2}\mathsf w_{3j+3}x^j}{(q^3;q^3)_j}\\
		& =x\sum_{j\ge0}\frac{q^{3j^2}x^j}{(q^3;q^3)_j}
	 \left(
		(q^{6j+3}+(q+q^2)q^{3j})\mathsf w_{3j}
		+(q^{3j+3}+q)\mathsf w_{3j-1}
	 \right)\\
		& =x\left( 
		q^3\mathscr S_{\mathsf w}(q^6x)
		+(q+q^2)\mathscr S_{\mathsf w}(q^3x)
		+q^3\mathscr T_{\mathsf w}(q^3x)
		+q\mathscr T_{\mathsf w}(x)
		\right).
	\end{align*}
	This proves \eqref{eq:ST1}. Similarly,
	\[
	\mathscr T_{\mathsf w}(x)-\mathscr T_{\mathsf w}(q^3x)
	=x\sum_{j\ge0}
	\frac{q^{3(j+1)^2}\mathsf w_{3j+2}x^j}{(q^3;q^3)_j},
	\]
	and substitution of the first recurrence gives \eqref{eq:ST2}.
	
	The constant terms follow directly from the definitions.
	For $n\ge1$, comparison of coefficients of $x^n$ in
	\eqref{eq:ST1}--\eqref{eq:ST2} determines both coefficients
	of degree $n$ from those of degree $n-1$, since
	$1-q^{3n}$ is invertible in $\mathbb C((q))$.
	Thus the prescribed constant terms determine the transforms
	uniquely. The final assertion follows from the initial
	values of $\mathsf A_m$ and $\mathsf B_m$.
\end{proof}

\begin{lemma}[{\cite[Lemma 2.4]{Xia}}]\label{lem:unique}
	Let $\mathfrak D_n$ be matrices with entries in $\mathbb C[[q]]$.
	Suppose $\boldsymbol f_n=\mathfrak D_n\boldsymbol f_{n+1}$ and
	$\boldsymbol g_n=\mathfrak D_n\boldsymbol g_{n+1}$, and
	$\ord_q(\boldsymbol f_n-\boldsymbol g_n)\to+\infty$.
	Then $\boldsymbol f_n=\boldsymbol g_n$ for every $n$.
	The statement also holds in a fixed Laurent shift of the vector space.
\end{lemma}

We also require the  definition of the divided difference in this paper. For pairwise distinct nodes $z_0,\ldots,z_n$, define the divided difference
\begin{equation}\label{eq:divided-difference}
	f[z_0,\ldots,z_n]:=\sum_{j=0}^n
	\frac{f(z_j)}{\prod_{k\ne j}(z_j-z_k)},
\end{equation}
where $f(z)$ is a function in $z$.
 
In particular, 
\begin{align*}
	f[z_0]&=f(z_0),\\
	f[z_0,z_1]&=\frac{f(z_0)-f(z_1)}{z_0-z_1},
	\\
	f[z_0,z_1,z_2]&=
	\frac{f(z_0)}{(z_0-z_1)(z_0-z_2)}+
	\frac{f(z_1)}{(z_1-z_0)(z_1-z_2)}+
	\frac{f(z_2)}{(z_2-z_0)(z_2-z_1)}.
\end{align*}
We always evaluate the full product at a node before taking a divided
difference. A divided difference of an infinite product is not a
product of divided differences of its factors.
For general background on divided differences, see de Boor~\cite{deBoor}.

\section{The six symmetric reflection identities}\label{sec:symmetric}

The aim of this section is to 
  present a proof of  Theorem~\ref{thm:main}.
The common construction is given first; the three pairs are then
resolved in the order $(F_1,F_2)$, $(F_3,F_4)$, and $(F_5,F_6)$.

We first prove several lemmas.

\begin{lemma} \label{lem:common-vector}
Define
\begin{equation*}
\mathbf U_n:=
\begin{pmatrix}\mathscr S_{\mathsf A}(q^{3n})\\\mathscr S_{\mathsf A}(q^{3n+3})+\mathscr T_{\mathsf A}(q^{3n})\\\mathscr T_{\mathsf A}(q^{3n})\end{pmatrix},
\qquad
\mathbf V_n:=
\begin{pmatrix}\mathscr S_{\mathsf B}(q^{3n})\\\mathscr S_{\mathsf B}(q^{3n+3})+\mathscr T_{\mathsf B}(q^{3n})\\\mathscr T_{\mathsf B}(q^{3n})\end{pmatrix}.
\end{equation*}
Then
\begin{equation}\label{eq:Drec}
\mathbf U_n=\mathbf D(q^{3n})\mathbf U_{n+1},\qquad
\mathbf V_n=\mathbf D(q^{3n})\mathbf V_{n+1},
\end{equation}
where 
$\mathbf D(x)$ is defined by 
\begin{equation}\label{eq:D}
	\mathbf D(x):=
	\begin{pmatrix}
		1+(q+q^2)x+q^3x^2&q^3x(1+qx)&qx\\
		1+q^2x&q^3x&1\\
		q^2x&q^3x&1
	\end{pmatrix}.
\end{equation}
The boundary and initial vectors are
\begin{equation}\label{eq:UVboundary}
\mathbf U_n=(1,1,0)^{\mathsf T}+O(q^{3n}),\qquad
\mathbf V_n=(0,1,1)^{\mathsf T}+O(q^{3n}),
\end{equation}
and
\begin{equation}\label{eq:UVzero}
\mathbf U_0=(F_2,F_4,q^2F_6)^{\mathsf T},\qquad
\mathbf V_0=(qF_1,F_3,F_5)^{\mathsf T}.
\end{equation}
\end{lemma}

\begin{proof}
	We first establish the matrix recurrence. Fix
	$\mathsf w=\mathsf A$ or $\mathsf B$, and put $x:=q^{3n}$. Write
	\[
	\begin{aligned}
		\mathsf U&:=\mathscr S_{\mathsf w}(x),&
		\mathsf V&:=\mathscr S_{\mathsf w}(q^3x)
		+\mathscr T_{\mathsf w}(x),&
		\mathsf W&:=\mathscr T_{\mathsf w}(x),\\
		\mathsf U'&:=\mathscr S_{\mathsf w}(q^3x),&
		\mathsf V'&:=\mathscr S_{\mathsf w}(q^6x)
		+\mathscr T_{\mathsf w}(q^3x),&
		\mathsf W'&:=\mathscr T_{\mathsf w}(q^3x).
	\end{aligned}
	\]
	Here the primes indicate the components at index $n+1$, not
	differentiation. By \eqref{eq:ST2},
	\begin{align}\label{3-5}
	\mathsf W=q^2x\mathsf U'+q^3x\mathsf V'+\mathsf W'.
	\end{align}
	Since $\mathsf V=\mathsf U'+\mathsf W$, this also gives
	\begin{align}\label{3-6}
	\mathsf V=(1+q^2x)\mathsf U'+q^3x\mathsf V'+\mathsf W'.
	\end{align}

	Next, \eqref{eq:ST1} yields
	\[
	\begin{aligned}
		\mathsf U-\mathsf U'
		&=(q+q^2)x\mathsf U'+q^3x\mathsf V'+qx\mathsf W.
	\end{aligned}
	\]
	Substituting \eqref{3-5}
	 into the above identity and simplifying, we obtain
	\begin{align}
		\mathsf U 
		&=\bigl(1+(q+q^2)x+q^3x^2\bigr)\mathsf U'
		+q^3x(1+qx)\mathsf V'
		+qx\mathsf W'.\label{3-7}
	\end{align}
Combining \eqref{3-5}--\eqref{3-7} yields 
	\[
	\begin{pmatrix}
		\mathsf U\\ \mathsf V\\ \mathsf W
	\end{pmatrix}
	=
	\mathbf D(x)
	\begin{pmatrix}
		\mathsf U'\\ \mathsf V'\\ \mathsf W'
	\end{pmatrix}.
	\]
	Taking $\mathsf w=\mathsf A$ and $\mathsf w=\mathsf B$ proves
	\eqref{eq:Drec} with the matrix in \eqref{eq:D}.

	We next verify the boundary values. By definition,
	\[
	\begin{aligned}
		\mathscr S_{\mathsf w}(q^{3n})
		&=\mathsf w_0+
		\sum_{j\ge1}
		\frac{q^{3j^2+3nj}\mathsf w_{3j}}{(q^3;q^3)_j},\\
		\mathscr T_{\mathsf w}(q^{3n})
		&=\mathsf w_{-1}+
		\sum_{j\ge1}
		\frac{q^{3j^2+3nj}\mathsf w_{3j-1}}{(q^3;q^3)_j}.
	\end{aligned}
	\]
	For $\mathsf w=\mathsf A,\mathsf B$, substitution of
	 the finite sums defining $\mathsf A_m$ and $\mathsf B_m$
	 in \eqref{2-16} and \eqref{2-17}   shows that the exponents in the first remainder
	have the form
	\[
	3j^2+3nj+r^2-3jr
	=
	3nj+\left(r-\frac{3j}{2}\right)^2+\frac{3j^2}{4}.
	\]
	The corresponding exponents in the second remainder are
	\[
	3j^2+3nj+r^2-(3j-1)r
	=
	3nj+\left(r-\frac{3j}{2}\right)^2
	+\frac{3j^2}{4}+r.
	\]
	The Gaussian coefficients and $(q^3;q^3)_j^{-1}$ contain only
	nonnegative powers of $q$. Hence, for $n\ge0$, both remainders
	are $O(q^{3n})$ coefficientwise. It follows that
	\begin{align}\label{3-8}
	\begin{pmatrix}
		\mathscr S_{\mathsf w}(q^{3n})\\
		\mathscr S_{\mathsf w}(q^{3n+3})
		+\mathscr T_{\mathsf w}(q^{3n})\\
		\mathscr T_{\mathsf w}(q^{3n})
	\end{pmatrix}
	=
	\begin{pmatrix}
		\mathsf w_0\\
		\mathsf w_0+\mathsf w_{-1}\\
		\mathsf w_{-1}
	\end{pmatrix}
	+O(q^{3n}).
	\end{align}
	Using \eqref{3-8} and 
	\[
	(\mathsf A_{-1},\mathsf A_0)=(0,1),
	\qquad
	(\mathsf B_{-1},\mathsf B_0)=(1,0),
	\]
	we obtain \eqref{eq:UVboundary}.

	It remains to prove \eqref{eq:UVzero}.
	Substitution of the definitions of $\mathsf A_{3s}$ and
	$\mathsf B_{3s}$ gives
	\begin{align}
		\mathscr S_{\mathsf A}(1)
		&=\sum_{s\ge0}
		\frac{q^{3s^2}\mathsf A_{3s}}{(q^3;q^3)_s} =\sum_{s\ge0}
		\frac{q^{3s^2} }{(q^3;q^3)_s}\sum_{r\geq 0} q^{r^2-3sr } \qb{3s-r}{r} 
		\nonumber\\
		&=\sum_{r,s\ge0}
		\frac{q^{\mathcal Q_-(r,s)}}{(q^3;q^3)_s}
		\qb{3s-r}{r}
		=F_2,\label{h-1}\\
		\mathscr S_{\mathsf B}(1)
		&=\sum_{s\ge0}
		\frac{q^{3s^2}\mathsf B_{3s}}{(q^3;q^3)_s} =\sum_{s\ge0}
		\frac{q^{3s^2} }{(q^3;q^3)_s}\sum_{r\geq 0} q^{r^2-3sr } \qb{3s-1-r}{r} 
		\nonumber\\
		&=\sum_{r,s\ge0}
		\frac{q^{\mathcal Q_-(r,s)}}{(q^3;q^3)_s}
		\qb{3s-r-1}{r}
		=qF_1. \label{h-2}
	\end{align}

	For the second components, applying  \eqref{eq:hrec} with
	$m=3s+1$ gives 
\[	 
	\mathsf w_{3s+1}
	=\mathsf w_{3s}+q^{-3s}\mathsf w_{3s-1}.
\]
	Therefore,
	\begin{align}
		\mathscr S_{\mathsf w}(q^3)+\mathscr T_{\mathsf w}(1)
		&=\sum_{s\ge0}
		\frac{q^{3s^2}
			\bigl(q^{3s}\mathsf w_{3s}+\mathsf w_{3s-1}\bigr)}
		{(q^3;q^3)_s}
		=\sum_{s\ge0}
		\frac{q^{3s^2+3s}\mathsf w_{3s+1}}{(q^3;q^3)_s}.  \label{h-7}
	\end{align}
	Expanding $\mathsf A_{3s+1}$ and $\mathsf B_{3s+1}$ now yields
	\begin{align}
		\mathscr S_{\mathsf A}(q^3)+\mathscr T_{\mathsf A}(1)
		&=\sum_{s\ge0}
		\frac{q^{3s^2+3s}\mathsf A_{3s+1}}{(q^3;q^3)_s}=\sum_{s\ge0}
		\frac{q^{3s^2+3s}}{(q^3;q^3)_s} \sum_{r\geq 0}
		 q^{r^2-(3s+1)r }\qb{3s+1-r}{r} \nonumber\\
		 &=\sum_{r,s\ge0}
		\frac{q^{\mathcal Q_-(r,s)-r+3s}}{(q^3;q^3)_s}
		\qb{3s+1-r}{r}
		=F_4, \label{h-3}\\
		\mathscr S_{\mathsf B}(q^3)+\mathscr T_{\mathsf B}(1)
		&=\sum_{s\ge0}
		\frac{q^{3s^2+3s}\mathsf B_{3s+1}}{(q^3;q^3)_s}=\sum_{s\ge0}
		\frac{q^{3s^2+3s}}{(q^3;q^3)_s} \sum_{r\geq 0}
		q^{r^2-(3s+1)r }\qb{3s-r}{r} \nonumber\\
		&=\sum_{r,s\ge0}
		\frac{q^{\mathcal Q_-(r,s)-r+3s}}{(q^3;q^3)_s}
		\qb{3s-r}{r}
		=F_3. \label{h-4}
	\end{align}

	Finally,  
	the finite sums defining $\mathsf A_m$ and $\mathsf B_m$
	in \eqref{2-16} and \eqref{2-17}   apply only when $m\ge0$, so we separate the
	$s=0$ terms in the third components:
	\begin{align}
		\mathscr T_{\mathsf A}(1)
		&=\mathsf A_{-1}
		+\sum_{s\ge1}
		\frac{q^{3s^2}\mathsf A_{3s-1}}{(q^3;q^3)_s}= 
		 \sum_{s\ge1}
		\frac{q^{3s^2}}{(q^3;q^3)_s}\sum_{r\geq 0} q^{r^2-(3s-1)r}\qb{3s-1-r}{r}
		\nonumber \\
		&=\sum_{\substack{r\ge0\\s\ge1}}
		\frac{q^{\mathcal Q_-(r,s)+r}}{(q^3;q^3)_s}
		\qb{3s-r-1}{r}
		=q^2F_6, \label{h-5}\\
		\mathscr T_{\mathsf B}(1)
		&=\mathsf B_{-1}
		+\sum_{s\ge1}
		\frac{q^{3s^2}\mathsf B_{3s-1}}{(q^3;q^3)_s}=1
		+\sum_{s\ge1}
		\frac{q^{3s^2} }{(q^3;q^3)_s}\sum_{r\geq 0} q^{r^2-(3s-1)r}\qb{3s-2-r}{r}
		\nonumber \\
		&=1+\sum_{\substack{r\ge0\\s\ge1}}
		\frac{q^{\mathcal Q_-(r,s)+r}}{(q^3;q^3)_s}
		\qb{3s-r-2}{r}
		=F_5.\label{h-6}
	\end{align}
	Combining \eqref{h-1}, \eqref{h-2} and  \eqref{h-3}--\eqref{h-6}  proves
  \eqref{eq:UVzero}.
\end{proof}

\begin{proposition}\label{prop:schur}
The six reflected sums satisfy
\begin{align}
qH(q)F_1+G(q)F_2&=\N(0,0),\label{eq:sumSchur-1}\\
H(q)F_3+G(q)F_4&=\N(-1,3), \label{eq:sumSchur-2}\\
H(q)F_5+q^2G(q)F_6&=\N(1,0).\label{eq:sumSchur-3}
\end{align}
\end{proposition}

\begin{proof}
 By \eqref{Ndef} and \eqref{eq:schur},
	\begin{align}
	\mathfrak N(a,b)
	&=\sum_{s\ge0}\frac{q^{3s^2+bs}}{(q^3;q^3)_s}
	\sum_{r\ge0}\frac{q^{r^2-(3s-a)r}}{(q;q)_r}
	\label{h-9}\\
	&=\sum_{s\ge0}\frac{q^{3s^2+bs}}{(q^3;q^3)_s}
 \left(G(q)\mathsf A_{3s-a}+H(q)\mathsf B_{3s-a} \right),
	\qquad(a\le0).\label{h-8}
	\end{align}
Taking  $(a,b)=(0,0)$ in \eqref{h-8}  gives 
	\begin{align}
		\mathfrak N(0,0)
		&=\sum_{s\ge0}\frac{q^{3s^2}}{(q^3;q^3)_s}
		\bigl(G(q)\mathsf A_{3s}+H(q)\mathsf B_{3s}\bigr) \nonumber\\
		&=G(q)\mathscr S_{\mathsf A}(1)
		+H(q)\mathscr S_{\mathsf B}(1).\label{vv-2}
	\end{align}
	By  \eqref{h-1} and \eqref{h-2},
	\begin{align}
	\mathscr S_{\mathsf A}(1)=F_2,
	\qquad
	\mathscr S_{\mathsf B}(1)=qF_1. \label{v-3}
\end{align}
	Identity \eqref{eq:sumSchur-1} follows from \eqref{vv-2} and \eqref{v-3}.

Setting  $(a,b)=(-1,3)$ in \eqref{h-8} yields 
	\begin{align}
		\mathfrak N(-1,3)
		&=\sum_{s\ge0}
		\frac{q^{3s^2+3s}}{(q^3;q^3)_s}
		\bigl(G(q)\mathsf A_{3s+1}
		+H(q)\mathsf B_{3s+1}\bigr) \nonumber\\
		&=G(q)\sum_{s\ge0}
		\frac{q^{3s^2+3s} \mathsf A_{3s+1} }{(q^3;q^3)_s} 
		+H(q) \sum_{s\ge0}
		\frac{q^{3s^2+3s} \mathsf B_{3s+1} }{(q^3;q^3)_s} \nonumber\\
			&=G(q)\bigl(
		\mathscr S_{\mathsf A}(q^3)+\mathscr T_{\mathsf A}(1)
		\bigr)
		+H(q)\bigl(
		\mathscr S_{\mathsf B}(q^3)+\mathscr T_{\mathsf B}(1)
		\bigr) \qquad ({\rm by}\ \eqref{h-7}) \nonumber \\
		&=G(q)F_4+H(q)F_3,
	\end{align}
	where the last equality again follows from \eqref{h-3} and \eqref{h-4}.
	This proves \eqref{eq:sumSchur-2}.

	Taking  $(a,b)=(1,0)$ in \eqref{h-9} gives 
	\begin{align}\label{h-10}
		\mathfrak N(1,0)
		&=\sum_{r\ge0}\frac{q^{r^2+r}}{(q;q)_r}
		+\sum_{s\ge1}\frac{q^{3s^2}}{(q^3;q^3)_s}
		\sum_{r\ge0}\frac{q^{r^2-(3s-1)r}}{(q;q)_r}\nonumber\\
		&=H(q)
		+ \sum_{s\ge1}
		\frac{q^{3s^2} }{(q^3;q^3)_s}(G(q) \mathsf A_{3s-1}  +H(q) \mathsf B_{3s-1}) .
	\end{align}
Here we have used \eqref{eq:intro-RR2} and \eqref{eq:schur}. 
	Since $\mathsf A_{-1}=0$ and $\mathsf B_{-1}=1$, we have
	\begin{align}
		\sum_{s\ge1}
		\frac{q^{3s^2}\mathsf A_{3s-1}}{(q^3;q^3)_s}
		&=\mathscr T_{\mathsf A}(1),\qquad 
		\sum_{s\ge1}
		\frac{q^{3s^2}\mathsf B_{3s-1}}{(q^3;q^3)_s}
		 =\mathscr T_{\mathsf B}(1)-1.\label{h-11}
	\end{align}
Combining \eqref{h-5}, \eqref{h-6}, \eqref{h-10}
 and \eqref{h-11} yields 
 \[
	\begin{aligned}
		\mathfrak N(1,0)
		&=H(q)+G(q)\mathscr T_{\mathsf A}(1)
		+H(q)\bigl(\mathscr T_{\mathsf B}(1)-1\bigr)\\
		&=G(q)\mathscr T_{\mathsf A}(1)
		+H(q)\mathscr T_{\mathsf B}(1)\\
		&=q^2G(q)F_6+H(q)F_5,
	\end{aligned}
	\]
	which proves \eqref{eq:sumSchur-3}.
 \end{proof}

 Let $\mathfrak q,x,t$ be independent variables, with $\mathfrak q\ne0$.
 For finite integer sets $I$, define
 $$\mathscr H_I(t):=\prod_{e\in I}(t-\mathfrak q^ex),\qquad   
  \mathscr J_I(t):=\prod_{e\in I}(t-\mathfrak q^ex^2).$$
 
 Moreover, 
 define
 \begin{align}
 	f(t)&:=\frac{(t-\mathfrak q^9x^3)(t-\mathfrak q^4x)(t-\mathfrak q^5x)}
 	{\mathscr J_{\{7,8,10,11\}}(t)}, \label{h-12}\\
 	\mathscr R(t)&:= 
 	\frac{t(t-1) \mathscr J_{\{-2,-1,1,2,4,5\}}(t)}
 	{x^3(t-x^3) \mathscr H_{\{-5,-4,-2,-1,1,2\}}(t)},\label{h-13}\\
 \mathscr C_1(t)&:=
 \frac{\mathfrak q^9x^3(t-\mathfrak q^9x^3)\mathscr H_{\{4,5,7,8,10\}}(t)}
 {t^2\mathscr J_{\{7,8,10,11,13\}}(t)}, \nonumber \\ 
 \mathscr C_2(t)&:=
 \frac{\Pi(t)\mathscr H_{\{4,5,7\}}(t)}{t^2\mathscr J_{\{7,8,10,11\}}(t)},
 \nonumber
  \end{align}
 where
 \[
 \Pi(t):=-\mathfrak q^8x^4(1+\mathfrak q^2)t^2
 +\mathfrak q^{18}x^5(1+x)(1+\mathfrak qx)t-\mathfrak q^{27}x^8(1+\mathfrak q^2).
 \]
 
 In addition,  the following rational functions are defined  in~\cite[Section 4]{Xia}:
  \begin{align*}
  	\mathscr D_1(t)&:=\mu_0(t)-\chi_1(t)
  	-\mathfrak q^3x(1+\mathfrak q^3x)\mu_1(t)-\mathfrak q^8x^2\chi_2(t),\\
  	\mathscr D_2(t)&:=\chi_0(t)-(1+\mathfrak qx)\mu_0(t)
  	-\mathfrak q^2x(1+\mathfrak qx)\chi_1(t)+\mathfrak q^7x^3\mu_1(t),
  \end{align*}
  where
  \begin{align*}
  		\mu_0(t)&:=1-\mathfrak q^9x^3/t^2,\\
  	\chi_1(t)&:=\frac{(1-\mathfrak q^3x)(1-\mathfrak q^2x)\mu_0(t)t}
  	{(t-\mathfrak q^2x)(t-\mathfrak q^7x^2)},\\
  	\mu_1(t)&:=(1-\mathfrak q^3x)(1-\mathfrak q^{18}x^3/t^2)f(t),\\
  	\chi_2(t)&:=\frac{(1-\mathfrak q^3x)(1-\mathfrak q^6x)(1-\mathfrak q^5x)
  		(1-\mathfrak q^{18}x^3/t^2)f(t)t}{(t-\mathfrak q^5x)(t-\mathfrak q^{13}x^2)},\\
  	\chi_0(t)&:=\frac{(1-\mathfrak q^{-1}x)(1-x^3/t^2)t\mathscr J_{\{2,4,5\}}(t)}
  	{(t-x^3)\mathscr H_{\{-1,1,2\}}(t)}.
  \end{align*}

\begin{lemma}\label{lem:rational-certificates}
For $i=1,2$,  
\begin{equation}\label{eq:ratcert}
\mathscr D_i(t)=\mathscr C_i(t)+\mathscr R(t)\mathscr C_i(\mathfrak q^9t) .
\end{equation}
\end{lemma}

\begin{proof}
Substitute the displayed rational functions into \eqref{eq:ratcert}
and multiply by their common denominator. In each case, expansion
of the numerator gives the zero polynomial in $t,x,\mathfrak q$.
Thus both identities hold as rational functions. 
\end{proof}

For a positive integer $d$ and a parameter $c$, define
\begin{equation}\label{eq:kernel-unified}
	\mathcal D_{d,c}(z):=\prod_{k\ge0}(1-zq^{dk}+cq^{2dk}),\qquad
	\mathcal P_{d,c}(z):=\mathcal D_{d,c}(z)^{-1}.
\end{equation}
Thus, if $z=\alpha+\beta$ and $\alpha\beta=c$, then
\[
\mathcal D_{d,c}(z)=(\alpha,\beta;q^d)_\infty, \qquad 
\mathcal P_{d,c}(z)=\frac{1}{(\alpha,\beta;q^d)_\infty}.\]
The second subscript is always the product $c=\alpha\beta$. This convention applies equally to
$d=3$ and $d=9$.

For $n\ge0$ and $0\le j\le n$, define
\begin{align}\label{v-4}
	\xi_{n,j}:=q^{4-6n+3j}+q^{5-3n-3j},
\end{align}
and, for $n\ge1$,
\begin{align}\label{v-5}
	\eta_{n,j}:=q^{8-6n+3j}+q^{10-3n-3j}.
\end{align}

Define
\begin{align}
	C_n^*
	&:=\mathcal E_9(-1)^n q^{-9\binom n2}(q^3;q^3)_n
	\mathcal D_{9,q^{9-9n}}[\xi_{n,0},\ldots,\xi_{n,n}],
	&& n\ge0, \label{eq:Cstar}\\
	A_n^*
	&:=\mathcal E_9(-1)^n q^{-9\binom n2}(q^3;q^3)_n
	(1-q^{1-3n})
	\mathcal D_{9,q^{18-9n}}[\eta_{n,0},\ldots,\eta_{n,n}],
	&& n\ge1, \label{eq:Astar}\\
	B_n^*
	&:=C_n^*+q^{3n+4}A_{n+1}^*,
	&& n\ge0.  \label{eq:Bstar}
\end{align}
The remaining initial value is defined separately by
$$
A_0^*:=\mathcal E_9X_1.
$$

Factoring the difference of two nodes shows that equality at distinct
indices would require $3(j+k)=3n+1$ for the $\xi $ nodes or
$3(j+k)=3n+2$ for the $\eta$ nodes, both of which are impossible.
Therefore, the divided differences in \eqref{eq:Cstar} and \eqref{eq:Astar} are
well defined.

\begin{lemma}\label{lem:starrec}
The vector $\mathbf P_n:=(A_n^*,B_n^*,C_n^*)^{\mathsf T}$ satisfies
\begin{equation}\label{eq:tilde}
\mathbf P_n=\widetilde{\mathbf M}(q^{3n})\mathbf P_{n+1},\qquad
\widetilde{\mathbf M}(x):=
\begin{pmatrix}
q^3(1+qx)(1+q^2x)&q^2(1+qx)&1\\
q^4x(1+q^2x)&q^3x&1\\
q^6x^2&q^3x&1
\end{pmatrix}.
\end{equation}
Its initial vector is $\E(X_1,X_2,X_3)^{\mathsf T}$.
\end{lemma}

\begin{proof}
We first establish a residue identity and then apply it to the
rational identities in Lemma~\ref{lem:rational-certificates}. Fix $0<q<1$ and put
\[
\Psi_*(t):=(q^9t,1/t;q^9)_\infty.
\]
This function is holomorphic on $\mathbb C^*$. For a rational function $R(t)\in\mathbb C(t)$, with
coefficients possibly depending on the fixed $q$, define
\begin{equation}\label{eq:residue-definition}
\mathscr L_*(R):=\frac12
\sum_{\substack{r\in\mathbb Z\\3\nmid r}}
\operatorname{Res}_{t=q^r}\bigl(\Psi_*(t)R(t)\bigr).
\end{equation}
Only poles of $R$ in $\mathbb C^*$ can contribute, so this sum has
only finitely many nonzero terms.

Shifting the product factors gives
\begin{equation}\label{eq:psi-residue-shift}
\Psi_*(q^9z)=-q^{-9}z^{-1}\Psi_*(z).
\end{equation}
Under $t=q^9z$, comparison of Laurent coefficients shows that
\[
\operatorname{Res}_{t=q^r}f(t)
=\operatorname{Res}_{z=q^{r-9}}\bigl(q^9f(q^9z)\bigr)
\]
for any function $f$ meromorphic near $q^r$.
Indeed, since
$
q^9z-q^r=q^9(z-q^{r-9}),
$
the coefficient of $(z-q^{r-9})^{-1}$ in the Laurent
expansion of $f(q^9z)$ is $q^{-9}$ times the coefficient
of $(t-q^r)^{-1}$ in the Laurent expansion of $f(t)$.
Multiplication by $q^9$ therefore makes these coefficients,
and hence the corresponding residues, equal.
Applying this observation and \eqref{eq:psi-residue-shift}, we get
\[
\begin{aligned}
 \operatorname{Res}_{t=q^r}
 \bigl(\Psi_*(t)q^{-9}tR(q^{-9}t)\bigr)&=\operatorname{Res}_{z=q^{r-9}}
 \bigl(q^9\Psi_*(q^9z)zR(z)\bigr)
 \\
&=-\operatorname{Res}_{z=q^{r-9}}\bigl(\Psi_*(z)R(z)\bigr).
\end{aligned}
\]
Since $3\nmid r$ if and only if $3\nmid(r-9)$, reindexing the
finite sum in \eqref{eq:residue-definition} gives
\[
\mathscr L_*(q^{-9}tR(q^{-9}t))=-\mathscr L_*(R).
\]
 Hence
\begin{equation}\label{eq:resShift}
\mathscr L_*\bigl(R(t)+q^{-9}tR(q^{-9}t)\bigr)=0.
\end{equation}

We now introduce two rational functions to which \eqref{eq:resShift} will
be applied.
 For an auxiliary nonzero parameter $\mathfrak q$, set
\begin{equation}\label{eq:star-rational-functions}
\begin{aligned}
K_n(t;\mathfrak q)
&:=\frac{(\mathfrak q^9/t;\mathfrak q^9)_n}
 {t^{n+1}(\mathfrak q^{3n+4}/t,\mathfrak q^{3n+5}/t;
 \mathfrak q^3)_{n+1}}, &&n\ge0,\\
I_n(t;\mathfrak q)
&:=\frac{(1-\mathfrak q^{3n-1})(1-\mathfrak q^{9n}/t^2)
 (\mathfrak q^9/t;\mathfrak q^9)_{n-1}}
 {t^{n+1}(\mathfrak q^{3n-1}/t,\mathfrak q^{3n+1}/t;
 \mathfrak q^3)_{n+1}}, &&n\ge1.
\end{aligned}
\end{equation}
All products in \eqref{eq:star-rational-functions} are finite, so we
may substitute $\mathfrak q=q^{-1}$.
For $x=\mathfrak q^{3n}$, cancellation of finite products gives
\begin{align}
 f(t)|_{x=\mathfrak q^{3n}}=\frac{K_{n+1}(t;\mathfrak q)}{K_n(t;\mathfrak q)},\qquad
 \mathscr R(t)|_{x=\mathfrak q^{3n}}=\frac{\mathfrak q^9tK_n(\mathfrak q^9t;\mathfrak q)}
 {K_n(t;\mathfrak q)},\label{h-20}
\end{align}
where $f(t)$ and $\mathscr R(t)$ are defined by \eqref{h-12} and \eqref{h-13}, respectively.  
Thus the two terms on the right of \eqref{eq:ratcert}, after
multiplication by $K_n(t;\mathfrak q)$, have exactly the form in \eqref{eq:resShift}
when $\mathfrak q=q^{-1}$.

Combining the finite product with $\Psi_*$ and using
\eqref{eq:kernel-unified}, we obtain, for $n\ge0$,
\begin{equation}\label{eq:star-product-combination}
\begin{aligned}
\Psi_*(t)(q^{-9}/t;q^{-9})_n
&=(q^9t,q^{-9n}/t;q^9)_\infty\\
&=\mathcal D_{9,q^{9-9n}}
 \left(q^9t+\frac{q^{-9n}}{t}\right).
\end{aligned}
\end{equation}
To evaluate the expression involving $K_n(t;\mathfrak q)$, define
\[
z_n(t):=t+\frac{q^{-9n-9}}t,\quad
a_{n,j}:=q^{-3n-4-3j},\quad b_{n,j}:=q^{-6n-5+3j}
\quad(0\le j\le n).
\]
By \eqref{v-4}, these points satisfy
\[
a_{n,j}b_{n,j}=q^{-9n-9},\quad
q^9(a_{n,j}+b_{n,j})=\xi_{n,j},\quad
z_n(t)-q^{-9}\xi_{n,j}
=\frac{(t-a_{n,j})(t-b_{n,j})}{t}.
\]
Pairing the factors in the denominator of
\eqref{eq:star-rational-functions} in this order therefore gives
\[
t^{n+1}(q^{-3n-4}/t,q^{-3n-5}/t;q^{-3})_{n+1}
=\prod_{j=0}^n\bigl(z_n(t)-q^{-9}\xi_{n,j}\bigr).
\]
Together with \eqref{eq:star-product-combination}, this yields
\begin{equation}\label{eq:star-K-integrand}
\Psi_*(t)\left(1-\frac{q^{-9n-9}}{t^2}\right)K_n(t;q^{-1})
=\frac{\mathcal D_{9,q^{9-9n}}(q^9z_n(t))z_n'(t)}
 {\displaystyle\prod_{j=0}^n(z_n(t)-q^{-9}\xi_{n,j})}.
\end{equation}

For clarity, the elementary residue formula used below is
\begin{equation}\label{eq:star-local-residue}
\operatorname{Res}_{t=t_0}
\frac{F(z(t))z'(t)}{\prod_{k=0}^n(z(t)-\zeta_k)}
=\frac{F(\zeta_j)}{\prod_{k\ne j}(\zeta_j-\zeta_k)},
\end{equation}
where the $\zeta_k$ are distinct, $F(z)$ is holomorphic near $\zeta_j$,
and $z$ is holomorphic near $t_0$ with $z(t_0)=\zeta_j$ and
$z'(t_0)\ne0$. It follows by multiplying the fraction by $t-t_0$
and taking the limit, since
\[
\lim_{t \to t_0}\frac{z(t)-\zeta_j}{t-t_0}= z'(t_0).
\]

The exponents of $a_{n,j}$ and $b_{n,j}$ are congruent to $2$ and
$1$ modulo $3$, respectively, and are distinct within each family.
Thus all $2n+2$ points are distinct and occur in
\eqref{eq:residue-definition}. Also
\[
z_n'(a_{n,j})=1-b_{n,j}/a_{n,j}\ne0,
\]
 and similarly at $b_{n,j}$.
 
 Applying \eqref{eq:star-local-residue} to
 \eqref{eq:star-K-integrand}, with
 \[
 F(u):=\mathcal D_{9,q^{9-9n}}(q^9u),
 \qquad
 \zeta_k:=q^{-9}\xi_{n,k},
 \qquad
 z(t):=z_n(t),
 \]
 gives, for either \(t_0=a_{n,j}\) or \(t_0=b_{n,j}\),
 \begin{align*}
  &\quad \operatorname{Res}_{t=t_0}
 	\left(
 	\Psi_*(t)
 	\left(1-\frac{q^{-9n-9}}{t^2}\right)
 	K_n(t;q^{-1})
 	\right) \nonumber\\ & =\operatorname{Res}_{t=t_0}
 	\left(
 \frac{\mathcal D_{9,q^{9-9n}}(q^9z_n(t))z_n'(t)}
 {\prod_{j=0}^n(z_n(t)-q^{-9}\xi_{n,j})}
 	\right)\\
 	&   =
 	\frac{
 		\mathcal D_{9,q^{9-9n}}
 		\bigl( \xi_{n,j}\bigr)
 	}{
 		 \prod_{\substack{0\le k\le n,\\k\ne j}}
 		\bigl(q^{-9}\xi_{n,j}-q^{-9}\xi_{n,k}\bigr)
 	}\nonumber\\
 	&=
 	q^{9n}
 	\frac{
 		\mathcal D_{9,q^{9-9n}}(\xi_{n,j})
 	}{	\prod_{\substack{0\le k\le n,\\k\ne j}}
 		(\xi_{n,j}-\xi_{n,k})
 	}.
 \end{align*} 
 Here we used
 \(z_n(a_{n,j})=z_n(b_{n,j})=q^{-9}\xi_{n,j}\)
 and the fact that \(z_n'\) is nonzero at both points.
 The factor \(q^{9n}\) arises by extracting \(q^{-9}\)
 from each of the \(n\) factors in the denominator.
 Thus the residues at \(a_{n,j}\) and \(b_{n,j}\) are equal.

The two equal contributions cancel the factor $1/2$
in \eqref{eq:residue-definition}; summing the divided-difference
formula gives
\begin{equation}\label{eq:star-K-residue}
\mathscr L_*\left(
 \left(1-\frac{q^{-9n-9}}{t^2}\right)K_n(t;q^{-1})\right)
=q^{9n}\mathcal D_{9,q^{9-9n}}[\xi_{n,0},\ldots,\xi_{n,n}].
\end{equation}

For $I_n$, let $n\ge1$ and define
\[
\widehat z_n(t):=t+\frac{q^{-9n}}t,\qquad
\widehat a_{n,j}=q^{1-3n-3j},\qquad
\widehat b_{n,j}=q^{-6n-1+3j}\quad(0\le j\le n).
\]
By \eqref{v-5},
\[
\widehat a_{n,j}\widehat b_{n,j}=q^{-9n},\qquad
q^9(\widehat a_{n,j}+\widehat b_{n,j})=\eta_{n,j}.
\]
Using \eqref{eq:star-product-combination} with $n$ replaced by
$n-1$, and pairing the denominator factors as before, we obtain
\[
\Psi_*(t)I_n(t;q^{-1})
=(1-q^{1-3n})
\frac{\mathcal D_{9,q^{18-9n}}(q^9\widehat z_n(t))
 \widehat z_n'(t)}
 {\displaystyle\prod_{j=0}^n
 (\widehat z_n(t)-q^{-9}\eta_{n,j})}.
\]
The two families of points again have distinct exponents congruent
to $1$ and $2$ modulo $3$. Formula \eqref{eq:star-local-residue}
therefore applies at both points corresponding to each $j$ and gives
\begin{equation}\label{eq:star-I-residue}
\mathscr L_*\bigl(I_n(t;q^{-1})\bigr)
=(1-q^{1-3n})q^{9n}
\mathcal D_{9,q^{18-9n}}[\eta_{n,0},\ldots,\eta_{n,n}].
\end{equation}

Since
\[
q^{9n}(q^{-3};q^{-3})_n
=q^{3n(n+1)}(-1)^nq^{-9\binom n2}(q^3;q^3)_n,
\]
equations \eqref{eq:star-K-residue}--\eqref{eq:star-I-residue}
and the definitions \eqref{eq:Cstar}--\eqref{eq:Astar} imply
\begin{equation}\label{eq:star-residue-normalization}
\begin{aligned}
\E(q^{-3};q^{-3})_n
\mathscr L_*\left(
 \left(1-\frac{q^{-9n-9}}{t^2}\right)K_n(t;q^{-1})\right)
&=q^{3n(n+1)}C_n^*, &&n\ge0,\\
\E(q^{-3};q^{-3})_n\mathscr L_*\bigl(I_n(t;q^{-1})\bigr)
&=q^{3n(n+1)}A_n^*, &&n\ge1.
\end{aligned}
\end{equation}

We now derive the recurrences explicitly. Throughout the following
calculation, set \(x=\mathfrak q^{3n}\) in the rational functions
of Lemma~\ref{lem:rational-certificates}.
By their definitions and \eqref{eq:star-rational-functions},
cancellation of the finite product factors gives
\begin{align}
	K_n(t;\mathfrak q)\mu_0(t)
	&=\left(1-\frac{\mathfrak q^{9n+9}}{t^2}\right)
	K_n(t;\mathfrak q), \label{h-21}\\
	K_n(t;\mathfrak q)\chi_0(t)
	&=I_n(t;\mathfrak q),\label{h-22}\\
	K_n(t;\mathfrak q)\chi_1(t)
	&=(1-\mathfrak q^{3n+3})I_{n+1}(t;\mathfrak q),\label{h-23}\\
	K_n(t;\mathfrak q)\mu_1(t)
	&=(1-\mathfrak q^{3n+3})
	\left(1-\frac{\mathfrak q^{9n+18}}{t^2}\right)
	K_{n+1}(t;\mathfrak q), \label{h-25}\\
	K_n(t;\mathfrak q)\chi_2(t)
	&=(1-\mathfrak q^{3n+3})(1-\mathfrak q^{3n+6})
	I_{n+2}(t;\mathfrak q). \label{h-26}
\end{align}
The identity involving \(\chi_0\) holds for \(n\ge1\);
the other four hold for \(n\ge0\).

Next, multiplying \eqref{eq:ratcert} by
\(K_n(t;\mathfrak q)\) and using \eqref{h-20}, we obtain
\begin{align*}
	K_n(t;\mathfrak q)\mathscr D_i(t)
	&=K_n(t;\mathfrak q)\mathscr C_i(t)+\mathfrak q^9t
	K_n(\mathfrak q^9t;\mathfrak q)
	\mathscr C_i(\mathfrak q^9t),
	\qquad i=1,2.
\end{align*}
From this point onward, set \(\mathfrak q=q^{-1}\) and
\(x=q^{-3n}\) in all the rational functions above.
Applying \eqref{eq:resShift} with
\(R(t)=K_n(t;q^{-1})\mathscr C_i(t)\) gives
\[
\mathscr L_*\bigl(K_n(t;q^{-1})\mathscr D_i(t)\bigr)=0,
\qquad i=1,2.
\] 

Define
\begin{align}
\widehat A_n:=q^{3n(n+1)}A_n^*,
\qquad
\widehat C_n:=q^{3n(n+1)}C_n^*. \label{h-17}
\end{align}
The five product identities \eqref{h-21}--\eqref{h-26}, together with
\eqref{eq:star-residue-normalization}, imply
\[
\begin{aligned}
	\mathcal E_9(q^{-3};q^{-3})_n
	\mathscr L_*\bigl(K_n(t;q^{-1})\mu_0(t)\bigr)
	&=\widehat C_n,\\
	\mathcal E_9(q^{-3};q^{-3})_n
	\mathscr L_*\bigl(K_n(t;q^{-1})\chi_0(t)\bigr)
	&=\widehat A_n,\\
	\mathcal E_9(q^{-3};q^{-3})_n
	\mathscr L_*\bigl(K_n(t;q^{-1})\chi_1(t)\bigr)
	&=\widehat A_{n+1},\\
	\mathcal E_9(q^{-3};q^{-3})_n
	\mathscr L_*\bigl(K_n(t;q^{-1})\mu_1(t)\bigr)
	&=\widehat C_{n+1},\\
	\mathcal E_9(q^{-3};q^{-3})_n
	\mathscr L_*\bigl(K_n(t;q^{-1})\chi_2(t)\bigr)
	&=\widehat A_{n+2}.
\end{aligned}
\]
For instance, the third equality follows from
\[
\begin{aligned}
	\mathcal E_9(q^{-3};q^{-3})_n
	\mathscr L_*\bigl(K_n(t;q^{-1})\chi_1(t)\bigr)
	&=
	\mathcal E_9(q^{-3};q^{-3})_n(1-q^{-3n-3})
	\mathscr L_*\bigl(I_{n+1}(t;q^{-1})\bigr)\\
	&=
	\mathcal E_9(q^{-3};q^{-3})_{n+1}
	\mathscr L_*\bigl(I_{n+1}(t;q^{-1})\bigr)
	=\widehat A_{n+1}.
\end{aligned}
\]

Finally, substitute these five evaluations into the vanishing
residue sums for \(K_n(t;q^{-1}) \mathscr D_1\) and \(K_n(t;q^{-1})\mathscr D_2\).
The definitions of \(\mathscr D_1\) and \(\mathscr D_2\)
in Lemma~\ref{lem:rational-certificates}, with
\(\mathfrak q=q^{-1}\) and \(x=q^{-3n}\), yield
\begin{equation}\label{h-15}
\begin{aligned}
	\widehat C_n-\widehat A_{n+1}
	-q^{-3n-3}(1+q^{-3n-3})\widehat C_{n+1}
	-q^{-6n-8}\widehat A_{n+2}
	&=0,  \\
	\widehat A_n-(1+q^{-3n-1})\widehat C_n
	-q^{-3n-2}(1+q^{-3n-1})\widehat A_{n+1}
	+q^{-9n-7}\widehat C_{n+1}
	&=0. 
\end{aligned}
\end{equation}
The first recurrence holds for \(n\ge0\), whereas the second
holds for \(n\ge1\), since its derivation uses the residue
formula for \(I_n\).

Define
\[
\widehat B_n:=q^{3n(n+1)}B_n^*.
\]
By \eqref{eq:Bstar} and the definitions of
$\widehat A_n$ and $\widehat C_n$, we have
\[
\widehat B_n
=\widehat C_n+q^{-3n-2}\widehat A_{n+1},
\]
and hence
\[
\widehat A_{n+2}
=q^{3n+5}\bigl(\widehat B_{n+1}-\widehat C_{n+1}\bigr).
\]
Substituting this expression into the first recurrence
in \eqref{h-15} gives
\[
\widehat C_n
=\widehat A_{n+1}
+q^{-3n-3}\widehat B_{n+1}
+q^{-6n-6}\widehat C_{n+1}.
\]
Consequently,
\begin{align*}
	\widehat B_n
	&=(1+q^{-3n-2})\widehat A_{n+1}
	+q^{-3n-3}\widehat B_{n+1}
	+q^{-6n-6}\widehat C_{n+1}.
\end{align*}
The second recurrence in \eqref{h-15} becomes
\[
\widehat A_n
=(1+q^{-3n-1})\widehat B_n
-q^{-9n-7}\widehat C_{n+1}.
\]
Substituting the expression for $\widehat B_n$, we obtain
\begin{align*}
	\widehat A_n
	&=(1+q^{-3n-1})(1+q^{-3n-2})\widehat A_{n+1}\\
	&\quad
	+q^{-3n-3}(1+q^{-3n-1})\widehat B_{n+1}
	+q^{-6n-6}\widehat C_{n+1}.
\end{align*}
Thus, for $n\ge1$, these three relations give
\[
q^{3n(n+1)}\mathbf P_n
=\mathbf M(q^{-3n};q^{-1})
q^{3(n+1)(n+2)}\mathbf P_{n+1},
\]
where
\[
\mathbf M(x;\mathfrak q):=
\begin{pmatrix}
(1+\mathfrak qx)(1+\mathfrak q^2x)&\mathfrak q^3x(1+\mathfrak qx)&\mathfrak q^6x^2\\
1+\mathfrak q^2x&\mathfrak q^3x&\mathfrak q^6x^2\\
1&\mathfrak q^3x&\mathfrak q^6x^2
\end{pmatrix}.
\]
The ratio of the two scalar factors is $q^{6n+6}$. Thus, with
$x=q^{3n}$, direct simplification gives
\[q^6x^2\mathbf M(x^{-1};q^{-1})=\widetilde{\mathbf M}(x),\]
which proves \eqref{eq:tilde} for $n\ge1$.

It remains to check $n=0$ and the initial vector. Evaluating
\eqref{eq:Cstar} and \eqref{eq:Astar} at the lowest indices gives
\begin{equation}\label{eq:starinitial}
C_0^*=\E X_3,\qquad
A_1^*=q^{-4}\E(X_2-X_3),\qquad
C_1^*=\E((1+q)X_2-qX_1).
\end{equation}
For $A_1^*$, the two product values are $X_2,X_3$ and the
node difference is $q^2(1-q^2)(1-q^3)$. For $C_1^*$, the two product values and their node difference are
$-q^{-2}(1-q^2)X_2$, $-q^{-1}(1-q)X_1$, and
$q^{-2}(1-q)(1-q^3)$, respectively. These evaluations give
\eqref{eq:starinitial} directly from the first divided difference.
By \eqref{eq:Bstar}, $B_0^*=\E X_2$; together with
$A_0^*=\E X_1$, this proves the stated initial vector.

Moreover, \eqref{eq:starinitial} gives
\begin{align}\label{h-31}
qA_0^*-(1+q)C_0^*-q^4(1+q)A_1^*+C_1^*=0.
\end{align}
Setting $n=0$ in \eqref{h-15} gives 
\[
\widehat C_0-\widehat A_{1}
-q^{-3}(1+q^{-3})\widehat C_{1}
-q^{-8}\widehat A_{2}
=0,
\]
which, together with  \eqref{h-17}, gives 
\begin{align}\label{h-32}
  C_0^*-q^6 A_{1}^*
- (1+q^3)C_{1}^*
-q^{10}A_{2}^*
=0.
\end{align}
Taking $n=0,1$ in \eqref{eq:Bstar} yields 
\begin{align} 
B_0^*&=C_0^*+q^4A_1^*, \label{h-33}\\
B_1^*&=C_1^*+q^7A_2^*. \label{h-34}
\end{align}
Solving the system \eqref{h-31}--\eqref{h-34}, we obtain
\[
\mathbf P_0=\widetilde{\mathbf M}(1)\mathbf P_{1}.
\]
Therefore,  \eqref{eq:tilde} is true when  $n=0$.

Finally, both sides are holomorphic on the punctured unit disk, and hence the identities
  extend from $0<q<1$ to $0<|q|<1$
by analyticity.
\end{proof}

\begin{lemma}\label{lem:scaling}
Define $\mathbf L(x):=\diag(x,q^{-1},q^{-1})$
 and 
 \begin{align}\label{h-38}
 	\mathbf Z_n:=\mathbf L(q^{3n})\mathbf P_n
 	=(q^{3n}A_n^*,q^{-1}B_n^*,q^{-1}C_n^*)^{\mathsf T},\qquad n\ge0.
 \end{align}
Then $\mathbf Z_n$ satisfies the recurrence~\eqref{eq:Drec}.

\end{lemma}

\begin{proof}
By direct multiplication,
\begin{equation}\label{eq:scaling}
	\mathbf L(x)\widetilde{\mathbf M}(x)=\mathbf D(x)\mathbf L(q^3x).
\end{equation}
  Combining it
with \eqref{eq:tilde} gives
\begin{align}
 \mathbf Z_n&=\mathbf L(q^{3n}) (\widetilde{\mathbf M}(q^{3n}) \mathbf P_{n+1})\nonumber\\
 &=(\mathbf L(q^{3n}) \widetilde{\mathbf M}(q^{3n})) \mathbf P_{n+1}\nonumber\\
 &= (\mathbf D(q^{3n})\mathbf L(q^{3n+3})) \mathbf P_{n+1}  \qquad ({\rm by}\ \eqref{eq:scaling})\nonumber\\
 &=  \mathbf D(q^{3n})(\mathbf L(q^{3n+3}) \mathbf P_{n+1})
 \nonumber\\
 &=  \mathbf D(q^{3n})  \mathbf Z_{n+1},
\end{align}
as required.
\end{proof}

\begin{lemma}\label{lem:boundary}
Let $\mathbf Z_n$ be defined by \eqref{h-38}. Then,  
\begin{equation}\label{eq:Zboundary}
\mathbf Z_n=
(-qH(q^9),\ q^{-1}G(q^9)-qH(q^9),\ q^{-1}G(q^9))^{\mathsf T}
+O(q^{3n-1}).
\end{equation}
\end{lemma}

\begin{proof}
Let $\mathsf h_k(z_0,\ldots,z_n)$ be the complete homogeneous symmetric
polynomial. Explicitly, for $k\ge0$,
\[
\mathsf h_k(z_0,\ldots,z_n)
:=\sum_{\substack{i_0,\ldots,i_n\ge0\\i_0+\cdots+i_n=k}}
 z_0^{i_0}\cdots z_n^{i_n}.
\]
For distinct nodes and an auxiliary variable $u$, comparison of the
coefficients of $u^{n+k}$ in
\[
\sum_{j=0}^n\frac{1}{(1-uz_j)\prod_{\ell\ne j}(z_j-z_\ell)}
=\frac{u^n}{\prod_{j=0}^n(1-uz_j)}
\]
gives
\[
\sum_{j=0}^n\frac{z_j^{n+k}}{\prod_{\ell\ne j}(z_j-z_\ell)}
=\mathsf h_k(z_0,\ldots,z_n).
\]
Thus the $n$th divided difference of the monomial $z^{n+k}$ is
$\mathsf h_k(z_0,\ldots,z_n)$.

For the coefficient estimate needed below, put
\[
 t_m(c):=[z^m]\mathcal D_{9,c}(z),\qquad
 d_m(\gamma):=(-1)^mq^{-9\binom m2}(q^9;q^9)_m
 t_m(\gamma q^{-9m}),\qquad (m\geq 0)
\]
and
\[
 \mathscr B(\gamma):=\sum_{j\ge0}
 \frac{\gamma^jq^{9j(j-1)}}{(q^9;q^9)_j}.
\]

The coefficient estimate in \cite[Lemma 8.4]{Xia} implies
that, for every positive integer \(s\),
\begin{equation}\label{eq:dm}
 d_m(q^s)=\mathscr B(q^s)+O(q^{9(m+1)+s}),
 \qquad d_m(q^s)\in\mathbb Z[[q]].
\end{equation}
To explain the error bound, first note that the constant
term in \(\gamma\) of
\(d_m(\gamma)-\mathscr B(\gamma)\) vanishes.
Indeed, \eqref{eq:kernel-unified} and Euler's expansion  \cite[II.2, p. 354]{GR}  
\[
(z;q)_\infty =\sum_{n=0}^\infty \frac{(-1)^n q^{n(n-1)/2}}{(q;q)_n} z^n
\]
 give
\[
t_m(0)
=[z^m](z;q^9)_\infty
=\frac{(-1)^m q^{9\binom m2}}{(q^9;q^9)_m},
\]
and hence \(d_m(0)=1=\mathscr B(0)\).
For \(j\ge1\), the coefficient estimate in the cited result gives
\[
[\gamma^j]\bigl(d_m(\gamma)-\mathscr B(\gamma)\bigr)
\in q^{9(m+1)}\mathbb Z[[q^9]].
\]
After substituting \(\gamma=q^s\), the contribution of
the \(j\)th term therefore belongs to
\[
q^{9(m+1)+sj}\mathbb Z[[q]].
\]
Since \(sj\ge s\) for \(j\ge1\), summing these contributions
gives the error term \(O(q^{9(m+1)+s})\).
This argument applies to every positive integer \(s\).

We next expand \(C_n^*\) using its definition in
\eqref{eq:Cstar}. Since
\[
\mathcal D_{9,c}(z)=\sum_{m\ge0}t_m(c)z^m,
\]
the monomial divided-difference identity established above gives
\[
\mathcal D_{9,q^{9-9n}}[\xi_{n,0},\ldots,\xi_{n,n}]
=
\sum_{k\ge0}
t_{n+k}(q^{9-9n})
\mathsf h_k(\xi_{n,0},\ldots,\xi_{n,n}).
\]
Here the terms of degree less than \(n\) have zero
\(n\)th divided difference. By the definition of \(d_m\),
\[
t_{n+k}(q^{9-9n})
=
\frac{(-1)^{n+k}q^{9\binom{n+k}{2}}}
{(q^9;q^9)_{n+k}}\,
d_{n+k}(q^{9+9k}).
\]
Substituting this expression into \eqref{eq:Cstar} and using
\[
\binom{n+k}{2}-\binom n2=nk+\binom k2
\]
yields
\begin{equation}\label{eq:Cexpand}
	C_n^*
	=
	\mathcal E_9(q^3;q^3)_n
	\sum_{k\ge0}
	\frac{(-1)^kq^{9nk+9\binom k2}
		d_{n+k}(q^{9+9k})}
	{(q^9;q^9)_{n+k}}\,
	\mathsf h_k(\xi_{n,0},\ldots,\xi_{n,n}).
\end{equation}

From
\eqref{v-4}, we see that 
both exponents of $\xi_{n,j}$ are at least \(4-6n\), and this bound is
attained by the first term of \(\xi_{n,0}\).
Every monomial in
\(\mathsf h_k(\xi_{n,0},\ldots,\xi_{n,n})\)
has total degree \(k\) in the nodes. Hence
\[
\mathsf h_k(\xi_{n,0},\ldots,\xi_{n,n})
=O\bigl(q^{k(4-6n)}\bigr).
\]
By \eqref{eq:dm},
\(d_{n+k}(q^{9+9k})\in\mathbb Z[[q]]\), and
\((q^9;q^9)_{n+k}^{-1}\) also contains no negative
powers of \(q\). Therefore the \(k\)th summand in
\eqref{eq:Cexpand} has no terms of degree less than
\[
9nk+9\binom k2+k(4-6n)
=3nk+4k+9\binom k2.
\]
The prefactor \(\mathcal E_9(q^3;q^3)_n\) contains no
negative powers and does not lower this bound.

For \(k\ge1\), the bound is at least \(3n+4\) and tends
to infinity with \(k\). Thus the sum of all terms with
\(k\ge1\) is \(O(q^{3n+4})\).
Since \(\mathsf h_0=1\), retaining the \(k=0\) term gives
\[
C_n^*
=
\mathcal E_9
\frac{(q^3;q^3)_n}{(q^9;q^9)_n}
d_n(q^9)
+O(q^{3n+4}).
\]
Also
\[
\E\frac{(q^3;q^3)_n}{(q^9;q^9)_n}
=\frac{(q^{9n+9};q^9)_\infty}{(q^{3n+3};q^3)_\infty}
=1+O(q^{3n+3}).
\]
The $k=0$ term of~\eqref{eq:Cexpand} and \eqref{eq:dm}
therefore give
 \begin{align}
C_n^*=\mathscr B(q^9)+O(q^{3n+3}). \label{h-37}
 \end{align} 
The Rogers--Ramanujan identity says
\[\mathscr B(q^9)=G(q^9).\]
It follows from    \eqref{h-37}
 and the above identity 
 	 that 
 \begin{align}
 	C_n^* &=G(q^9)+O(q^{3n+3}) \quad (n\geq 0). \label{h-35}
 \end{align}

Similarly, $A_n^*$ has the expansion
\[
A_n^*=\E(q^3;q^3)_n(1-q^{1-3n})\sum_{k\ge0}
\frac{(-1)^kq^{9nk+9\binom k2}d_{n+k}(q^{18+9k})}
 {(q^9;q^9)_{n+k}}\mathsf h_k(\eta_{n,0},\ldots,\eta_{n,n}).
\]
Now the smallest node exponent is $8-6n$. 
After multiplication
by $q^{3n}$, every $k\ge1$ term has order at least $3n+9$.
The $k=0$ term is
\[
q^{3n}(1-q^{1-3n})\{1+O(q^{3n+3})\}
\{H(q^9)+O(q^{9n+27})\},
\]
because $\mathscr B(q^{18})=H(q^9)$. Therefore, 
 \begin{align}
	q^{3n}A_n^* &=-qH(q^9)+O(q^{3n})\quad(n\ge1).\label{h-36}
\end{align}

Finally, ~\eqref{eq:Bstar}, \eqref{h-35}
 and \eqref{h-36} give
  $$B_n^*=G(q^9)-q^2H(q^9)+O(q^{3n+3}),$$
   from which with \eqref{h-35}
   and \eqref{h-36}, we deduce that  
 \eqref{eq:Zboundary} is true for $n\geq 1$.
 
 For \(n=0\), the initial values give
 \[
 \mathbf Z_0
 =\mathcal E_9(X_1,q^{-1}X_2,q^{-1}X_3)^{\mathsf T}.
 \]
 Since \(\mathcal E_9X_i\) and \(G(q^9)\) all have
 constant term \(1\), the difference between this vector
 and the stated main term belongs to
 \(\mathbb C[[q]]^3\), and hence is \(O(q^{-1})\).
 Thus the estimate also holds at \(n=0\).
\end{proof}


\begin{proposition}\label{prop:second}
	The six sums satisfy
	\begin{equation}\label{eq:second}
		\begin{aligned}
			G(q^9)F_1-qH(q^9)F_2
			&=\mathcal E_9X_1,\\
			G(q^9)F_3-q^2H(q^9)F_4
			&=\mathcal E_9X_2,\\
			G(q^9)F_5-q^4H(q^9)F_6
			&=\mathcal E_9X_3.
		\end{aligned}
	\end{equation}
\end{proposition}

\begin{proof}
By~\eqref{eq:Drec},  
\eqref{eq:UVboundary}, \eqref{eq:Zboundary}
 and Lemma \ref{lem:scaling}, the two vectors
  $-qH(q^9)\mathbf U_n+q^{-1}G(q^9)\mathbf V_n$ and $\mathbf Z_n$  
satisfy the same recurrence and differ by a vector whose least
exponent tends to infinity.
Lemma~\ref{lem:unique} gives their equality.
Indeed, their difference is $O(q^{3n-1})$, and the entries of
$\mathbf D(q^{3n})$ belong to $\mathbb C[[q]]$.
At zero, Lemma~\ref{lem:starrec} gives
$$\mathbf Z_0=\E(X_1,q^{-1}X_2,q^{-1}X_3)^{\mathsf T}.$$
Insert~\eqref{eq:UVzero} and multiply the last two components by $q$.
The result is~\eqref{eq:second}.
\end{proof}

For $1\le r\le22$, define
\[
 \mathcal J_r:=(q^r,q^{45-r};q^{45})_\infty.
\]  The following product identities are verified with
the Frye--Garvan Maple package \texttt{thetaids}~\cite{FG}.

\begin{proposition}\label{prop:products}
	For $1\le i\le6$, let $R_i$ denote the explicit
	product expression on the right-hand side of the
	corresponding identity in \eqref{eq:R1}--\eqref{eq:R6}.
Then
\begin{align}
G(q^9)R_1-qH(q^9)R_2&=\E X_1,&
qH(q)R_1+G(q)R_2&=P_1^2+qP_2P_3,\label{eq:prod12}\\
G(q^9)R_3-q^2H(q^9)R_4&=\E X_2,&
H(q)R_3+G(q)R_4&=P_2^2+P_1P_3,\label{eq:prod34}\\
G(q^9)R_5-q^4H(q^9)R_6&=\E X_3,&
H(q)R_5+q^2G(q)R_6&=P_1P_2-qP_3^2.\label{eq:prod56}
\end{align}
\end{proposition}
\begin{proof}
We first express all products using factors with base $q^{45}$.
For a list $A$,  let $\mathcal J(A):=\prod_{a\in A}\mathcal J_a$, with
multiplicity. Use the lists
\begin{align*}
\mathcal A&=(3,6,12,15,21),&
\mathcal C&=(3,6,9,12,15,18,21),\\
\mathcal X&=(1,8,10,17,19),&
\mathcal Y&=(2,7,11,16,20),\\
\mathcal Z&=(4,5,13,14,22),&
\mathcal G&=(1,4,6,9,11,14,16,19,21),\\
\mathcal H&=(2,3,7,8,12,13,17,18,22).
\end{align*}
Then
\[
\E=\mathcal J(\mathcal A)^{-1},\quad
\frac{J_{45}}{J_3}=\mathcal J(\mathcal C)^{-1},\quad
(X_1,X_2,X_3)=(\mathcal J(\mathcal X),\mathcal J(\mathcal Y),\mathcal J(\mathcal Z)),
\]
and
$G(q)=\mathcal J(\mathcal G)^{-1}$,
$H(q)=\mathcal J(\mathcal H)^{-1}$,
$G(q^9)=\mathcal J_9^{-1}$, and
$H(q^9)=\mathcal J_{18}^{-1}$.
Every term of the six differences in the proposition is now
explicitly $c_jq^{k_j}\prod_{r=1}^{22}\mathcal J_r^{e_{j,r}}$.

We now submit the six resulting generalized eta-product differences,
at level $45$, to \texttt{provemodfuncidBATCH} in the Frye--Garvan
Maple package \texttt{thetaids}~\cite{FG}.  The routine returns the
verification bounds $50$ for the three left relations and $60$ for the
three right relations in \eqref{eq:prod12}--\eqref{eq:prod56}, and
proves that all six differences vanish. Thus the six displayed product
identities follow. Since this modular verification is standard and is
not the focus of the present paper, we refer to~\cite{FG} for the
underlying criterion.
\end{proof}

We are now ready to prove Theorem~\ref{thm:main}.

\begin{proof}[Proof of Theorem~\ref{thm:main}]
At this point, we invoke the three dual evaluations proved in
\cite[Theorem 1.2]{Xia}:
\begin{equation}\label{eq:dual}
 \N(0,0)=P_1^2+qP_2P_3,\quad
 \N(-1,3)=P_2^2+P_1P_3,\quad
 \N(1,0)=P_1P_2-qP_3^2.
\end{equation}
For $1\leq i \leq 6$, let $\epsilon_i=F_i-R_i$. Equations \eqref{eq:dual}
and Propositions~\ref{prop:schur}, \ref{prop:second}, and
\ref{prop:products} give homogeneous systems for the three pairs.

For the pair $(F_1,F_2)$,
\[
\begin{pmatrix}G(q^9)&-qH(q^9)\\qH(q)&G(q)\end{pmatrix}
\binom{\epsilon_1}{\epsilon_2}=0.
\]
Its determinant is $\Delta=G(q)G(q^9)+q^2H(q)H(q^9)$,
a power series with constant
term one. Hence $\epsilon_1=\epsilon_2=0$, proving
\eqref{eq:R1} and \eqref{eq:R2}.

For the pair  $(F_3,F_4)$,
\[
\begin{pmatrix}G(q^9)&-q^2H(q^9)\\H(q)&G(q)\end{pmatrix}
\binom{\epsilon_3}{\epsilon_4}=0.
\]
The determinant is again $\Delta$, so \eqref{eq:R3} and
\eqref{eq:R4} follow.

Finally, for the pair $(F_5,F_6)$,
\[
\begin{pmatrix}G(q^9)&-q^4H(q^9)\\H(q)&q^2G(q)\end{pmatrix}
\binom{\epsilon_5}{\epsilon_6}=0.
\]
The determinant $q^2\Delta$ is nonzero. Multiplication by the
adjugate, and the absence of zero divisors in $\mathbb C[[q]]$,
give \eqref{eq:R5} and \eqref{eq:R6}.
\end{proof}

\section{The four asymmetric reflection identities}\label{sec:asymmetric}
\label{sec:airyfamilies}

To prove Theorem~\ref{thm:asym}, we  first establish identities for
$q$-Airy functions and solve a  linear system. The same
identities will also be used for the reflected sums with cube roots of unity.

For $0<|p|<1$, $X\in\mathbb C$, and $z\in\mathbb C^*$, define
\[
\Ai_p(X):=\sum_{r\ge0}\frac{p^{r(r-1)/2}X^r}{(p^2;p^2)_r} 
\]
and 
\[
\vartheta_p^+(z):=(-z,-p/z,p;p)_\infty
=\sum_{k\in\mathbb Z}p^{k(k-1)/2}z^k,
\]
where the equality follows from Jacobi's triple product identity.

By the definition of $\Ai_p(X)$,
\begin{align*}
	\Ai_p(X)-\Ai_p(p^2X)
	&=\sum_{r\ge1}
	\frac{p^{r(r-1)/2}(1-p^{2r})X^r}{(p^2;p^2)_r}\\
	&=\sum_{r\ge1}
	\frac{p^{r(r-1)/2}X^r}{(p^2;p^2)_{r-1}}
	=X\Ai_p(pX).
\end{align*}
Therefore, the function $\Ai_p$ satisfies
\begin{equation}\label{Airec}
	\Ai_p(X)-\Ai_p(p^2X)=X\Ai_p(pX).
\end{equation}
To record the normalization needed below, set
\[
\mathcal W_p(X)
:=\Ai_p(X)\Ai_p(-pX)+\Ai_p(-X)\Ai_p(pX).
\]
Applying \eqref{Airec} at $X$ and $-X$ gives
\begin{align*}
	\mathcal W_p(pX)&=\Ai_p(pX)\Ai_p(-p^2X)+\Ai_p(-pX)\Ai_p(p^2X)\\
	&=\Ai_p(pX)\{\Ai_p(-X)+X\Ai_p(-pX)\}\\
	&\quad+\Ai_p(-pX)\{\Ai_p(X)-X\Ai_p(pX)\}\\
	&=\mathcal W_p(X).
\end{align*}
Hence $\mathcal W_p(X)=\mathcal W_p(p^nX)$ for every $n\ge0$.
Since $|p|<1$, letting $n\to\infty$ and using $\Ai_p(0)=1$ yields
\begin{equation}\label{eq:AiryWronskian}
	\mathcal W_p(X)=\mathcal W_p(0)=2.
\end{equation}

For the Airy calculations, take $0<q<1$, set $\rho=q^{1/4}$,
and work in $\mathbb C((\rho))$, where $q=\rho^4$. Put
\[
 \mathsf C_\pm(\rho):=\Ai_{\rho^2}(\pm\rho),\qquad
 \mathsf D_\pm(\rho):=\Ai_{\rho^2}(\pm\rho^3),
\]
and, for $m\ge-1$, define
\[
 u_m:=(-1)^mq^{-m^2/4}\Ai_{q^{1/2}}(q^{m/2+3/4}),\qquad
 v_m:=q^{-m^2/4}\Ai_{q^{1/2}}(-q^{m/2+3/4}).
\]

\begin{lemma}
\label{lem:Airy-basis}
The sequences $u_m$ and $v_m$
satisfy \eqref{eq:hrec}. 
In addition, 
\begin{equation}\label{basis}
	u_m=\mathsf D_+(\rho)\mathsf A_m-\rho^{-1}\mathsf C_+(\rho)\mathsf B_m,
	\qquad
	v_m=\mathsf D_-(\rho)\mathsf A_m+\rho^{-1}\mathsf C_-(\rho)\mathsf B_m,
\end{equation}
and
\begin{equation}\label{ainv}
	\mathsf A_m=\frac{\mathsf C_-(\rho)u_m+\mathsf C_+(\rho)v_m}{2},
	\qquad
	\mathsf B_m=\frac{\rho(\mathsf D_+(\rho)v_m-\mathsf D_-(\rho)u_m)}{2}.
\end{equation}
\end{lemma}

\begin{proof}
  Taking $p=q^{1/2}$ and $X=q^{m/2-1/4}$ in \eqref{Airec} gives
\[
\Ai_{q^{1/2}}(q^{m/2+3/4})
=\Ai_{q^{1/2}}(q^{m/2-1/4})
-q^{m/2-1/4}\Ai_{q^{1/2}}(q^{m/2+1/4}).
\]
Multiplying the above identity
 by  $(-1)^m q^{-m^2/4}$
yields
\[
u_m=u_{m-1}+q^{1-m}u_{m-2} .
\]
Setting  $p=q^{1/2}$ and $X=-q^{m/2-1/4}$ in \eqref{Airec} gives
\[
\Ai_{q^{1/2}}(-q^{m/2+3/4})
=\Ai_{q^{1/2}}(-q^{m/2-1/4})
+q^{m/2-1/4}\Ai_{q^{1/2}}(-q^{m/2+1/4}).
\]
Multiplying the above identity
by  $ q^{-m^2/4}$
 gives
\[
v_m=v_{m-1}+q^{1-m}v_{m-2}.
\]
Their two initial values are
\[
(u_{-1},u_0)
=(-\rho^{-1}\mathsf C_+(\rho),\mathsf D_+(\rho)),
\qquad
(v_{-1},v_0)
=(\rho^{-1}\mathsf C_-(\rho),\mathsf D_-(\rho)).
\]
Recall that
\[
(\mathsf A_{-1},\mathsf A_0)=(0,1),
\qquad
(\mathsf B_{-1},\mathsf B_0)=(1,0).
\]
Thus the sequences $\mathsf A_m$ and $\mathsf B_m$ form the normalized
fundamental pair for the
second-order recurrence \eqref{eq:hrec}. In particular, every
solution $\mathsf w$ is uniquely determined by
$(\mathsf w_{-1},\mathsf w_0)$ and satisfies
\[
\mathsf w_m=\mathsf w_0\mathsf A_m+\mathsf w_{-1}\mathsf B_m.
\]
Applying this uniqueness statement to $u_m$ and $v_m$ gives
\eqref{basis}.
By \eqref{eq:AiryWronskian}, specialized at
$p=q^{1/2}$ and $X=q^{1/4}$, we have
\[
\mathsf C_+(\rho)\mathsf D_-(\rho)+\mathsf C_-(\rho)\mathsf D_+(\rho)=2.
\]
Solving the system \eqref{basis} and using the above identity gives \eqref{ainv}.
\end{proof}

\begin{lemma}\label{Cvalue}
One has
\begin{equation*}
 \mathsf C_\pm(\rho)=\frac{G(q)}{J_5}\thp_{\rho^{10}}(\pm \rho).
\end{equation*}
\end{lemma}
\begin{proof}
We use the classical finite Schur identity, as recorded in
\cite[Section 2]{UZ}:
\[
 \sum_{r\ge0}z^{r^2}\genfrac{[}{]}{0pt}{}{m-r}{r}_z
 =\sum_{j\in\mathbb Z}(-1)^jz^{j(5j+1)/2}
 \genfrac{[}{]}{0pt}{}{m}{\lfloor(m-5j)/2\rfloor}_z.
\]
We substitute $z=q^{-1}$ into the finite Schur identity and use 
\[
\genfrac{[}{]}{0pt}{}{N}{r}_{q^{-1}}
=q^{-r(N-r)}\qb{N}{r}.
\]
For $m=2M$ and $m=2M+1$, multiply by $q^{M^2}$ and
$q^{M^2+M}$, respectively.
 Writing the original summation index as \(s\) and setting \(r=M-s\), we obtain  the exponents
\[
M^2+s^2-2Ms=r^2,
\qquad
M^2+M+s^2-(2M+1)s=r^2+r.
\]
On the right-hand sides, separating $j=2k$ and $j=2k+1$
therefore gives
\begin{align*}
	\sum_{r=0}^{M}q^{r^2}\qb{M+r}{2r}
	&=\sum_{k\in\mathbb Z}
	q^{15k^2-k}\qb{2M}{M-5k} 
	 -\sum_{k\in\mathbb Z}
	q^{15k^2+19k+6}\qb{2M}{M-5k-3},\\
	\sum_{r=0}^{M}q^{r^2+r}\qb{M+r+1}{2r+1}
	&=\sum_{k\in\mathbb Z}
	q^{15k^2+4k}\qb{2M+1}{M-5k} 
	 -\sum_{k\in\mathbb Z}
	q^{15k^2+14k+3}\qb{2M+1}{M-5k-2}.
\end{align*}

For fixed $r\ge0$ and $c\in\mathbb Z$, we have,
coefficientwise as $M\to\infty$,
\[
\qb{M+r}{2r}\longrightarrow\frac{1}{(q;q)_{2r}},
\ 
\qb{M+r+1}{2r+1}\longrightarrow\frac{1}{(q;q)_{2r+1}},
\ 
\qb{2M+\varepsilon}{M+c}\longrightarrow\frac{1}{J_1}
\quad(\varepsilon=0,1).
\]
The quadratic exponents tend to infinity with the
summation indices, so only finitely many indices
contribute to any fixed coefficient of $q$.
We may therefore pass to the limit in the preceding
identities. Multiplying by $J_1$ yields
\begin{align}
	J_1\sum_{r\ge0}\frac{q^{r^2}}{(q;q)_{2r}}
	&=\sum_{k\in\mathbb Z}
	\bigl(q^{15k^2-k}-q^{15k^2+19k+6}\bigr),
	\label{evenC}\\
	J_1\sum_{r\ge0}\frac{q^{r^2+r}}{(q;q)_{2r+1}}
	&=\sum_{k\in\mathbb Z}
	\bigl(q^{15k^2+4k}-q^{15k^2+14k+3}\bigr).
	\label{oddC}
\end{align}

Since $q=\rho^4$, the definition of $\mathsf C_+(\rho)$
gives
\begin{align*}
	\mathsf C_+(\rho)
	&=\sum_{m\ge0}\frac{\rho^{m^2}}{(q;q)_m}=\sum_{r\ge0}\frac{q^{r^2}}{(q;q)_{2r}}
	+\rho\sum_{r\ge0}\frac{q^{r^2+r}}{(q;q)_{2r+1}}.
\end{align*}
Thus \eqref{evenC} and \eqref{oddC} imply
\begin{align*}
	J_1\mathsf C_+(\rho)
	&=\sum_{k\in\mathbb Z}
	\bigl(\rho^{60k^2-4k}+\rho^{60k^2+16k+1}\bigr)\\
	&\quad-\sum_{k\in\mathbb Z}
	\bigl(\rho^{60k^2+56k+13}
	+\rho^{60k^2+76k+24}\bigr)\\
	&=\sum_{n\in\mathbb Z}(-1)^n
	\bigl(\rho^{15n^2-2n}+\rho^{15n^2+8n+1}\bigr),
\end{align*}
where the last equality combines $n=2k$ and $n=2k+1$.

 The classical quintuple product identity, in the form
\[
 (p,z,p/z;p)_\infty(pz^2,p/z^2;p^2)_\infty
 =\sum_{n\in\mathbb Z}p^{n(3n-1)/2}z^{3n}(1-zp^n),
\]
with $p=\rho^{10}$ and $z=-\rho$, evaluates this as
\[
 \thp_{\rho^{10}}(\rho)(\rho^8,\rho^{12};\rho^{20})_\infty
 =\thp_{\rho^{10}}(\rho)\frac{J_1G(q)}{J_5}.
\]
Replacing $\rho$ by $-\rho$ proves the other sign.
\end{proof}

Define
\begin{align}\label{h-72}
	\tau_1:=\thp_{\rho^2}(-\rho),\quad \tau_2:=\thp_{\rho^2}(\rho),\quad
	\tau_3:=\thp_{\rho^6}(\rho),\quad \tau_4:=\thp_{\rho^6}(-\rho).
\end{align}

\begin{corollary}\label{cor:Airy-specialization}
The following identity holds among the specified values of the
$q$-Airy and theta functions:
\begin{equation}\label{CH}
 \mathsf C_+(\rho)\tau_1-\mathsf C_-(\rho)\tau_2=-2\rho J_1H(q).
\end{equation}
\end{corollary}

\begin{proof}
	Define
	\[
	\mathscr R_p(Z)
	:=
	\sum_{r\ge0}
	\frac{p^{r^2}(-Z)^r}{(p;p)_r}.
	\]
	We use Morita's connection theorem~\cite{Morita} in the normalization
	\begin{equation}\label{Morita}
		\vartheta_p^{+}(X/p)\Ai_p(-X)
		+\vartheta_p^{+}(-X/p)\Ai_p(X)
		=
		2(p^2;p^2)_\infty
		\mathscr R_{p^2}(-p^3/X^2),
	\end{equation}
	where \(X\ne0\). The factor \(2\) is consistent with
	\[
	(p,-1;p)_\infty=2(p^2;p^2)_\infty.
	\]
	
	Specialize \eqref{Morita} at
$
	p=\rho^2$, $X=\rho.$  
	By the theta shift relation,
	\[
	\vartheta_{\rho^2}^{+}(\rho^{-1})
	=\rho^{-1}\vartheta_{\rho^2}^{+}(\rho)
	=\rho^{-1}\tau_2,
	\]
	and
	\[
	\vartheta_{\rho^2}^{+}(-\rho^{-1})
	=-\rho^{-1}\vartheta_{\rho^2}^{+}(-\rho)
	=-\rho^{-1}\tau_1.
	\]
	Moreover, by the definitions of \(\mathsf C_\pm(\rho)\),
	\[
	\Ai_{\rho^2}(-\rho)=\mathsf C_-(\rho),
	\qquad
	\Ai_{\rho^2}(\rho)=\mathsf C_+(\rho).
	\]
	 	Therefore,
	\[
	\begin{aligned}
		\mathscr R_{p^2}\left(-\frac{p^3}{X^2}\right)
		&=\mathscr R_q(-q)
		=\sum_{r\ge0}
		\frac{q^{r^2}\{-(-q)\}^r}{(q;q)_r}\\
		&=\sum_{r\ge0}
		\frac{q^{r^2+r}}{(q;q)_r}
		=H(q),
	\end{aligned}
	\]
	where the last equality follows from \eqref{eq:intro-RR2}.
	
	Consequently, the specialization of \eqref{Morita} becomes
	\[
	\rho^{-1}\tau_2\mathsf C_-(\rho)
	-\rho^{-1}\tau_1\mathsf C_+(\rho)
	=
	2J_1H(q).
	\]
	Multiplying by \(-\rho\) and rearranging gives
	\[
	\mathsf C_+(\rho)\tau_1
	-\mathsf C_-(\rho)\tau_2
	=
	-2\rho J_1H(q),
	\]
	which is \eqref{CH}.
\end{proof}

\begin{lemma}\label{lem:directions}
For each integer $k$, define
\begin{align}
 \mathsf c_k:=\thp_{q^{3/2}}(q^{k-3/4}),\qquad
 \mathsf d_k:=\thp_{q^{3/2}}(-q^{k-3/4}). \label{h-71}
\end{align}
Then  
\begin{align}
 \tau_1\mathscr S_u(q^k)+\tau_2\mathscr S_v(q^k)
 &=2J_1\N(0,k),\label{firstS}\\
 \tau_1\mathscr T_u(q^k)+\tau_2\mathscr T_v(q^k)
 &=2J_1\N(1,k),\label{firstT}\\
 \mathsf c_k\mathscr S_u(q^k)+\mathsf d_k\mathscr S_v(q^k)
 &=2J_3\N(k-1,3-2k),\label{secondS}\\
 \mathsf c_k\mathscr S_u(q^{k+3})+\mathsf d_k\mathscr S_v(q^{k+3})
 &=2q^{2k}J_3\N(k+2,-3-2k),\label{shiftS}\\
 \mathsf c_k\mathscr T_u(q^k)+\mathsf d_k\mathscr T_v(q^k)
 &=-2q^{k-1}J_3\N(k,-2k).\label{shiftT}
\end{align}
\end{lemma}

\begin{proof}
		Write \[
	Z_k^+:=\mathscr S_u(q^k), \qquad  
	Z_k^-:=\mathscr S_v(q^k).
	\]
	Direct expansion gives
	\begin{equation}\label{Z}
		Z_k^\pm
		=\sum_{r,s\ge0}
		\frac{\sigma_\pm(r,s)
			q^{(r^2+6rs+3s^2+2r+4ks)/4}}
		{(q;q)_r(q^3;q^3)_s},
		\ \ 
		\sigma_+(r,s):=(-1)^s,\ \ 
		\sigma_-(r,s):=(-1)^r.
	\end{equation}
	The theta series gives, for every integer $n$,
	\[
		\thp_p(p^nz)
	 =p^{-n(n-1)/2}z^{-n}\thp_p(z).
	\]
	We first prove \eqref{firstS} and \eqref{firstT}.
	In \eqref{Morita}, take
	$
	p=q^{1/2}$, $
	X=q^{m/2+3/4}$,
 $ (m\ge-1)$. 
	Applying the theta shift formula with $n=m$ gives
	\[
	\begin{aligned}
		\thp_{q^{1/2}}(q^{m/2+1/4})
		&=q^{-m(m-1)/4}q^{-m/4}\tau_2
		=q^{-m^2/4}\tau_2,\\
		\thp_{q^{1/2}}(-q^{m/2+1/4})
		&=q^{-m(m-1)/4}(-q^{1/4})^{-m}\tau_1
		=(-1)^mq^{-m^2/4}\tau_1.
	\end{aligned}
	\]
	Consequently, \eqref{Morita} becomes
	\[
	\begin{aligned}
		&\tau_2q^{-m^2/4}
		\Ai_{q^{1/2}}(-q^{m/2+3/4})
		+\tau_1(-1)^mq^{-m^2/4}
		\Ai_{q^{1/2}}(q^{m/2+3/4})
		=2J_1\mathscr R_q(-q^{-m}).
	\end{aligned}
	\]
	By the definitions of $u_m$, $v_m$, and $\mathscr R_q$, this is
	\[
	\tau_1u_m+\tau_2v_m
	=2J_1\mathscr R_q(-q^{-m})
	=2J_1\sum_{r\ge0}\frac{q^{r^2-mr}}{(q;q)_r}.
	\]
	
	Taking $m=3s$, multiplying by
	$q^{3s^2+ks}/(q^3;q^3)_s$, and summing over $s\ge0$, we obtain
	\[
	\begin{aligned}
		\tau_1\mathscr S_u(q^k)+\tau_2\mathscr S_v(q^k)
		&=\sum_{s\ge0}
		\frac{q^{3s^2+ks}}{(q^3;q^3)_s}
		\bigl(\tau_1u_{3s}+\tau_2v_{3s}\bigr)\\
		&=2J_1\sum_{r,s\ge0}
		\frac{q^{r^2-3rs+3s^2+ks}}
		{(q;q)_r(q^3;q^3)_s}\\
		&=2J_1\N(0,k).
	\end{aligned}
	\]
	This proves \eqref{firstS}.
	
	Taking instead $m=3s-1$ gives
	\[
	\tau_1u_{3s-1}+\tau_2v_{3s-1}
	=2J_1\sum_{r\ge0}
	\frac{q^{r^2-3rs+r}}{(q;q)_r}.
	\]
	Hence
	\[
	\begin{aligned}
		\tau_1\mathscr T_u(q^k)+\tau_2\mathscr T_v(q^k)
		&=\sum_{s\ge0}
		\frac{q^{3s^2+ks}}{(q^3;q^3)_s}
		\bigl(\tau_1u_{3s-1}+\tau_2v_{3s-1}\bigr)\\
		&=2J_1\sum_{r,s\ge0}
		\frac{q^{r^2-3rs+3s^2+r+ks}}
		{(q;q)_r(q^3;q^3)_s}\\
		&=2J_1\N(1,k),
	\end{aligned}
	\]
	which proves \eqref{firstT}.  
	
	We next prove \eqref{secondS}.
 For each $r\ge0$, put
$
	X_r:=q^{3r/2+k+3/4}.
$
	The identity
	\[
	\frac{3s^2}{4}+\frac{3rs}{2}+ks
	=\frac{3s(s-1)}4+
	\left(\frac{3r}{2}+k+\frac34\right)s
	\]
	allows us to sum \eqref{Z} over $s$ first:
	\[
	\begin{aligned}
		\mathscr S_u(q^k)
		&=\sum_{r\ge0}\frac{q^{(r^2+2r)/4}}{(q;q)_r}
		\Ai_{q^{3/2}}(-X_r),\\
		\mathscr S_v(q^k)
		&=\sum_{r\ge0}\frac{(-1)^rq^{(r^2+2r)/4}}{(q;q)_r}
		\Ai_{q^{3/2}}(X_r).
	\end{aligned}
	\]
	
	Apply \eqref{Morita} with $p=q^{3/2}$ and $X=X_r$. 
	Define
	\[
	a_r:=q^{-3r^2/4+(3/2-k)r}.
	\]
	The theta shift formula gives
	\[
	\begin{aligned}
		\thp_{q^{3/2}}(X_r/p)
		&=q^{-3r(r-1)/4}q^{-(k-3/4)r}\mathsf c_k
		=a_r\mathsf c_k,\\
		\thp_{q^{3/2}}(-X_r/p)
		&=q^{-3r(r-1)/4}(-q^{k-3/4})^{-r}\mathsf d_k
		=(-1)^ra_r\mathsf d_k.
	\end{aligned}
	\]
	Thus
	\[
	\mathsf c_k\Ai_{q^{3/2}}(-X_r)
	+(-1)^r\mathsf d_k\Ai_{q^{3/2}}(X_r)
	=2J_3a_r^{-1}
	\mathscr R_{q^3}(-q^{3-3r-2k}).
	\]
	Now
	\[
	q^{(r^2+2r)/4}a_r^{-1}
	=q^{r^2+(k-1)r},
	\qquad
	\mathscr R_{q^3}(-q^{3-3r-2k})
	=\sum_{s\ge0}
	\frac{q^{3s^2+(3-3r-2k)s}}{(q^3;q^3)_s}.
	\]
	It follows that
	\[
	\begin{aligned}
		\mathsf c_k\mathscr S_u(q^k)
		+\mathsf d_k\mathscr S_v(q^k)
		&=2J_3\sum_{r\ge0}
		\frac{q^{r^2+(k-1)r}}{(q;q)_r}
		\mathscr R_{q^3}(-q^{3-3r-2k})\\
		&=2J_3\sum_{r,s\ge0}
		\frac{q^{r^2-3rs+3s^2+(k-1)r+(3-2k)s}}
		{(q;q)_r(q^3;q^3)_s}\\
		&=2J_3\N(k-1,3-2k).
	\end{aligned}
	\]
	This proves \eqref{secondS}.
	
	For \eqref{shiftS}, apply the theta shift formula with
	$p=q^{3/2}$ and $n=2$:
	\[
	\begin{aligned}
		\mathsf c_{k+3}
		&=\thp_{q^{3/2}}\bigl((q^{3/2})^2q^{k-3/4}\bigr)
		=q^{-3/2}q^{-2k+3/2}\mathsf c_k
		=q^{-2k}\mathsf c_k,\\
		\mathsf d_{k+3}
		&=\thp_{q^{3/2}}\bigl((q^{3/2})^2(-q^{k-3/4})\bigr)
		=q^{-3/2}(-q^{k-3/4})^{-2}\mathsf d_k
		=q^{-2k}\mathsf d_k.
	\end{aligned}
	\]
	Using \eqref{secondS} with $k$ replaced by $k+3$, we obtain
	\[
	\begin{aligned}
		q^{-2k}
		\bigl(\mathsf c_k\mathscr S_u(q^{k+3})
		+\mathsf d_k\mathscr S_v(q^{k+3})\bigr)
		&
		=\mathsf c_{k+3}\mathscr S_u(q^{k+3})
		+\mathsf d_{k+3}\mathscr S_v(q^{k+3})\\
		&=2J_3\N(k+2,-3-2k).
	\end{aligned}
	\]
	Multiplication by $q^{2k}$ proves \eqref{shiftS}.
	
	It remains to prove \eqref{shiftT}.
	Expanding the Airy functions at index $3s-1$ gives
	\[
	\begin{aligned}
		u_{3s-1}
		&=-(-1)^s q^{-(3s-1)^2/4}
		\sum_{r\ge0}
		\frac{q^{r(r-1)/4+(3s/2+1/4)r}}{(q;q)_r},\\
		v_{3s-1}
		&=q^{-(3s-1)^2/4}
		\sum_{r\ge0}
		\frac{(-1)^rq^{r(r-1)/4+(3s/2+1/4)r}}{(q;q)_r}.
	\end{aligned}
	\]
	The exponent simplifies to
	\[
	\begin{aligned}
		&3s^2+ks-\frac{(3s-1)^2}{4}
		+\frac{r(r-1)}4+\left(\frac{3s}{2}+\frac14\right)r
		=-\frac14+\frac{r^2+6rs+3s^2+(4k+6)s}{4}.
	\end{aligned}
	\]
	Since $\rho=q^{1/4}$, this yields
	\[
	\begin{aligned}
		\mathscr T_u(q^k)
		&=-\rho^{-1}\sum_{r,s\ge0}
		\frac{(-1)^s q^{r^2/4+3rs/2+3s^2/4+(k+3/2)s}}
		{(q;q)_r(q^3;q^3)_s},\\
		\mathscr T_v(q^k)
		&=\rho^{-1}\sum_{r,s\ge0}
		\frac{(-1)^r q^{r^2/4+3rs/2+3s^2/4+(k+3/2)s}}
		{(q;q)_r(q^3;q^3)_s}.
	\end{aligned}
	\]
	Put
$
	Y_r:=q^{3r/2+k+9/4}.
$
	Because
	\[
	\frac{3s^2}{4}+\frac{3rs}{2}
	+\left(k+\frac32\right)s
	=\frac{3s(s-1)}4+
	\left(\frac{3r}{2}+k+\frac94\right)s,
	\]
	we obtain
	\[
	\begin{aligned}
		\mathscr T_u(q^k)
		&=-\rho^{-1}\sum_{r\ge0}
		\frac{q^{r^2/4}}{(q;q)_r}\Ai_{q^{3/2}}(-Y_r),\\
		\mathscr T_v(q^k)
		&=\rho^{-1}\sum_{r\ge0}
		\frac{(-1)^rq^{r^2/4}}{(q;q)_r}\Ai_{q^{3/2}}(Y_r).
	\end{aligned}
	\]
	
	Applying  \eqref{Morita} with $p=q^{3/2}$ and $X=Y_r$ gives  
	\[
	\mathsf c_k\Ai_{q^{3/2}}(-Y_r)
	-(-1)^r\mathsf d_k\Ai_{q^{3/2}}(Y_r)
	=2J_3b_r^{-1}\mathscr R_{q^3}(-q^{-3r-2k}),
	\]
	where 	 
	\[
	b_r:=q^{-3r^2/4-kr-k+3/4}.
	\]
Therefore, 
	\[
	\begin{aligned}
		\mathsf c_k\mathscr T_u(q^k)
		+\mathsf d_k\mathscr T_v(q^k)
		&=-2\rho^{-1}J_3
		\sum_{r\ge0}\frac{q^{r^2/4}b_r^{-1}}{(q;q)_r}
		\mathscr R_{q^3}(-q^{-3r-2k})\\
		&=-2q^{k-1}J_3
		\sum_{r,s\ge0}
		\frac{q^{r^2-3rs+3s^2+kr-2ks}}
		{(q;q)_r(q^3;q^3)_s}\\
		&=-2q^{k-1}J_3\N(k,-2k),
	\end{aligned}
	\]
	which proves \eqref{shiftT}.
\end{proof}

Define
\[
 \Theta:=\sum_{m,n\in\mathbb Z}q^{m^2-mn+n^2+m-n},\qquad
 W:=\sum_{m,n\in\mathbb Z}q^{2m^2+mn+2n^2-2n}.
\]
The quadratic exponents are nonnegative on the integer lattice.
\begin{lemma}\label{binary}
	We have
\begin{equation}\label{binaryeval}
 \Theta=\frac{3J_3^3}{J_1},\qquad
 W=3J_3J_9H(q)G(q^9)-J_1J_3H(q)^2.
\end{equation}
\end{lemma}
\begin{proof}
After replacing $n$ by $-n$, the first identity becomes
\[
 \Theta=\sum_{m,n\in\mathbb Z}q^{m^2+mn+n^2+m+n},
\]
which is the specialization $x=1$ of the cubic theta-function
evaluation of Hirschhorn, Garvan, and Borwein
\cite[Eq.~(1.23)]{HGB}.

Write
\[
\vartheta_b(r):=\vartheta_{q^b}^{+}(q^r)
=\sum_{j\in\mathbb Z}
q^{\frac{b}{2}j(j-1)+rj}=(-q^r,-q^{b-r},q^b;q^b)_\infty.
\]
Two elementary identities that will be used below follow directly from
this definition:
\[
\begin{aligned}
	\vartheta_b(b-r)
	&=\sum_{j\in\mathbb Z}
	q^{\frac{b}{2}j(j+1)-rj}
	=\vartheta_b(r),\\
	\vartheta_b(r+b)
	&=\sum_{j\in\mathbb Z}
	q^{\frac{b}{2}j(j+1)+rj}
	=q^{-r}\vartheta_b(r).
\end{aligned}
\]

For every \(m,n\in\mathbb Z\), write uniquely
\[
n=4k+\varepsilon,\qquad
m=\ell-k,\qquad
\varepsilon\in\{0,1,2,3\}.
\]
A direct calculation gives
\[
\begin{aligned}
	&2m^2+mn+2n^2-2n
	=2\ell^2+\varepsilon\ell
	+30k^2+(15\varepsilon-8)k
	+2\varepsilon^2-2\varepsilon.
\end{aligned}
\]
Moreover,
\[
\sum_{\ell\in\mathbb Z}q^{2\ell^2+\varepsilon\ell}
=\vartheta_4(\varepsilon+2)
\]
and
\[
\sum_{k\in\mathbb Z}
q^{30k^2+(15\varepsilon-8)k}
=\vartheta_{60}(15\varepsilon+22).
\]
Consequently,
\begin{align}
	W
	&=\sum_{\varepsilon=0}^{3}
	q^{2\varepsilon^2-2\varepsilon}
	\vartheta_4(\varepsilon+2)
	\vartheta_{60}(15\varepsilon+22) \notag\\
	&=\vartheta_4(2)\vartheta_{60}(22)
	+\vartheta_4(3)\vartheta_{60}(37)
	+q^4\vartheta_4(4)\vartheta_{60}(52)
	+q^{12}\vartheta_4(5)\vartheta_{60}(67) \notag\\
	&=\vartheta_4(2)\vartheta_{60}(22)
	+\vartheta_4(1)\vartheta_{60}(23)
	+q^4\vartheta_4(4)\vartheta_{60}(8)
	+q^4\vartheta_4(1)\vartheta_{60}(7).
	\label{eq:W-dissection}
\end{align}
 
Note that  for \(0<r<b\),
\begin{align}
	\vartheta_b(r)
 =J_b
	\frac{(q^{2r},q^{2b-2r};q^{2b})_\infty}
	{(q^r,q^{b-r};q^b)_\infty},
	\qquad
	\vartheta_b(b)=\frac{2J_{2b}^{\,2}}{J_b}.
	\label{eq:vartheta-product}
\end{align}
 Substitution of \eqref{eq:vartheta-product} into \eqref{eq:W-dissection} reduces the second
identity in \eqref{binaryeval} to a generalized eta-product identity of
common level $720$.  We verify this identity directly with the
Frye--Garvan Maple package \texttt{thetaids}~\cite{FG}, using
\texttt{provemodfuncidBATCH}; the routine returns the required bound
$10608$ and proves the identity by its built-in valence-formula test.
Thus the Maple package proves the required eta-product equality. Since
the routine and its underlying criterion are established in~\cite{FG},
we do not reproduce the general modular function theory here.
\end{proof}

\begin{proposition}\label{prop:multiplier}
Suppose that $\mathscr U,\mathscr V,\nu,\lambda\in\mathbb C((\rho))$,
where $\rho=q^{1/4}$, satisfy
\begin{equation}\label{eq:universal-system}
\begin{pmatrix}\tau_1&\tau_2\\\tau_3&\tau_4\end{pmatrix}
\binom{\mathscr U}{\mathscr V}
=2\nu\binom{J_1}{\lambda J_3}.
\end{equation}
Then
\[
\frac{\mathsf C_-(\rho)\mathscr U+\mathsf C_+(\rho)\mathscr V}{2}
=\nu\mathcal M(\lambda),
\]
where 
\begin{equation}\label{eq:multiplier}
	\mathcal M(\lambda):=\frac{J_1}{J_3^2}
	\left(J_9G(q^9)+\frac{\lambda-1}{3}J_1H(q)\right).
\end{equation}
 In particular,
\begin{align}
\mathcal M(1)=G(q^9)/\Delta. \label{v-1}
	\end{align}
\end{proposition}

\begin{proof}
		It follows from the definitions of \(\tau_1,\tau_2,\tau_3, \tau_4\) that
		\begin{align*}
			\tau_1\tau_4-\tau_2\tau_3
			&=
			\sum_{k,l\in\mathbb Z}
			\left((-1)^{k+l}-1\right)
			\rho^{k^2+3l^2-2l}\\
			&=-2
			\sum_{\substack{k,l\in\mathbb Z\\ k+l\ {\rm odd}}}
			\rho^{k^2+3l^2-2l}.
		\end{align*}
		The change of variables
	$
		k=m+n,\  l=m-n+1
		$
		is a bijection from \(\mathbb Z^2\) onto the set of pairs
		\((k,l)\in\mathbb Z^2\) for which \(k+l\) is odd.

		Under this substitution,
		\begin{align*}
			k^2+3l^2-2l
			&=(m+n)^2+3(m-n+1)^2-2(m-n+1)\\
			&=1+4\bigl(m^2-mn+n^2+m-n\bigr).
		\end{align*}
		Since \(q=\rho^4\), we consequently obtain
		\begin{align}
			\tau_1\tau_4-\tau_2\tau_3
			&=-2\sum_{m,n\in\mathbb Z}
			\rho^{1+4(m^2-mn+n^2+m-n)} \nonumber \\
			&=-2\rho
			\sum_{m,n\in\mathbb Z}
			q^{m^2-mn+n^2+m-n} \nonumber \\
			&=-2\rho\Theta. \label{dettheta}
		\end{align}

	By Lemma~\ref{Cvalue},
		\begin{align*}
			\mathsf C_-(\rho)
			&=\frac{G(q)}{J_5}
			\sum_{k\in\mathbb Z}
			(-1)^k\rho^{5k^2-4k},\\
			\mathsf C_+(\rho)
			&=\frac{G(q)}{J_5}
			\sum_{k\in\mathbb Z}
			\rho^{5k^2-4k}.
		\end{align*}
		Combining these expansions with those of \(\tau_3\) and \(\tau_4\)
		gives
		\begin{align*}
			\mathsf C_-(\rho)\tau_4-\mathsf C_+(\rho)\tau_3
			&=
			\frac{G(q)}{J_5}
			\sum_{k,l\in\mathbb Z}
			\left((-1)^{k+l}-1\right)
			\rho^{5k^2-4k+3l^2-2l}\\
			&=-\frac{2G(q)}{J_5}
			\sum_{\substack{k,l\in\mathbb Z\\ k+l\ {\rm odd}}}
			\rho^{5k^2-4k+3l^2-2l}.
		\end{align*}
		Use again
	$
		k=m+n,\  l=m-n+1.
		$
		The exponent now becomes
		\begin{align*}
			5k^2-4k+3l^2-2l
			&=5(m+n)^2-4(m+n) +3(m-n+1)^2-2(m-n+1)\\
			&=1+4\bigl(2m^2+mn+2n^2-2n\bigr).
		\end{align*}
		It follows that
		\begin{align}
			\mathsf C_-(\rho)\tau_4-\mathsf C_+(\rho)\tau_3
			&=-\frac{2G(q)}{J_5}
			\sum_{m,n\in\mathbb Z}
			\rho^{1+4(2m^2+mn+2n^2-2n)} \nonumber \\
			&=-\frac{2\rho G(q)}{J_5}
			\sum_{m,n\in\mathbb Z}
			q^{2m^2+mn+2n^2-2n} \nonumber \\
			&=-\frac{2\rho G(q)}{J_5}W. \label{Ctheta}
		\end{align}

By \eqref{dettheta} and Lemma~\ref{binary},
\[
\tau_1\tau_4-\tau_2\tau_3=-\frac{6\rho J_3^3}{J_1}\ne0.
\]
Solving \eqref{eq:universal-system} gives
\[
\mathscr U=
\frac{2\nu(J_1\tau_4-\lambda J_3\tau_2)}{\tau_1\tau_4-\tau_2\tau_3},
\qquad
\mathscr V=
\frac{2\nu(\lambda J_3\tau_1-J_1\tau_3)}{\tau_1\tau_4-\tau_2\tau_3}.
\]
Consequently,
\begin{align*}
\frac{\mathsf C_-(\rho)\mathscr U+\mathsf C_+(\rho)\mathscr V}{2}
&=\nu\,
 \frac{J_1(\mathsf C_-(\rho)\tau_4-\mathsf C_+(\rho)\tau_3)
       +\lambda J_3(\mathsf C_+(\rho)\tau_1-\mathsf C_-(\rho)\tau_2)}
      {\tau_1\tau_4-\tau_2\tau_3}\\
&=\nu\frac{J_1}{\Theta}
 \left(\frac{G(q)W}{J_5}+\lambda J_3H(q)\right).
\end{align*}
The last equality uses \eqref{CH}, \eqref{dettheta}, and
\eqref{Ctheta}. For the final simplification,
$G(q)H(q)=J_5/J_1$ and
\eqref{binaryeval} give
\[
 \frac{G(q)W}{J_5}=\frac{3J_3J_9G(q^9)}{J_1}-J_3H(q),
 \qquad \frac{J_1}{\Theta}=\frac{J_1^2}{3J_3^3}.
\]
Substitution now yields
\begin{align*}
\frac{\mathsf C_-(\rho)\mathscr U+\mathsf C_+(\rho)\mathscr V}{2}
&=\nu\frac{J_1^2}{3J_3^3}
\left(\frac{3J_3J_9G(q^9)}{J_1}+(\lambda-1)J_3H(q)\right)\\
&=\nu\frac{J_1}{J_3^2}
\left(J_9G(q^9)+\frac{\lambda-1}{3}J_1H(q)\right)
=\nu\mathcal M(\lambda).
\end{align*}
Finally, it follows from
\cite[Chapter 8, equation (8.3.7)]{ABLostIII} that, in the present
notation,
\begin{equation}\label{Delta}
 \Delta=G(q)G(q^9)+q^2H(q)H(q^9)=\frac{J_3^2}{J_1J_9}.
\end{equation}
Setting $\lambda=1$ in \eqref{eq:multiplier} and using \eqref{Delta} proves \eqref{v-1}.
\end{proof}

\begin{lemma}\label{lem:asymmetric-series}
The four reflected sums in \eqref{eq:F7}--\eqref{eq:F10} satisfy
\begin{equation}\label{eq:asym-series}
\begin{aligned}
F_7&=\mathscr S_{\mathsf A}(q),\\
q^2F_8&=\mathscr S_{\mathsf B}(q),\\
F_9&=(1+q)\mathscr S_{\mathsf A}(q^2)
 +q^2\mathscr S_{\mathsf A}(q^5)+q^2\mathscr T_{\mathsf A}(q^2),\\
q^2F_{10}&=(1+q)\mathscr S_{\mathsf B}(q^2)
 +q^2\mathscr S_{\mathsf B}(q^5)+q^2\mathscr T_{\mathsf B}(q^2).
\end{aligned}
\end{equation}
\end{lemma}

\begin{proof}
 By \eqref{eq:F7} and \eqref{2-16},
	\[
	\begin{aligned}
		F_7
		&=\sum_{s\ge0}\frac{q^{3s^2+s}}{(q^3;q^3)_s}
		\sum_{r\ge0}q^{r^2-3sr}\qb{3s-r}{r}\\
		&=\sum_{s\ge0}\frac{q^{3s^2+s}\mathsf A_{3s}}
		{(q^3;q^3)_s}
		=\mathscr S_{\mathsf A}(q).
	\end{aligned}
	\]
	Similarly, in light of  \eqref{eq:F8} and \eqref{2-17},
	\[
	\begin{aligned}
		q^2F_8
		&=\sum_{s\ge0}\frac{q^{3s^2+s}}{(q^3;q^3)_s}
		\sum_{r\ge0}q^{r^2-3sr}\qb{3s-r-1}{r}\\
		&=\sum_{s\ge0}\frac{q^{3s^2+s}\mathsf B_{3s}}
		{(q^3;q^3)_s}
		=\mathscr S_{\mathsf B}(q).
	\end{aligned}
	\]

	For $F_9$, the two inner sums in \eqref{eq:F9} are
	$\mathsf A_{3s}$ and $\mathsf A_{3s+1}$, respectively.
  Therefore
	\[
	F_9
	=(1+q)\mathscr S_{\mathsf A}(q^2)
	+q^2\sum_{s\ge0}
	\frac{q^{3s^2+5s}\mathsf A_{3s+1}}{(q^3;q^3)_s}.
	\]
	The recurrence \eqref{eq:hrec}, with $m=3s+1$, gives
	\[
	\mathsf A_{3s+1}
	=\mathsf A_{3s}+q^{-3s}\mathsf A_{3s-1}.
	\]
	Consequently,
	\[
	\begin{aligned}
		\sum_{s\ge0}
		\frac{q^{3s^2+5s}\mathsf A_{3s+1}}{(q^3;q^3)_s}
		&=\sum_{s\ge0}
		\frac{q^{3s^2+5s}\mathsf A_{3s}}{(q^3;q^3)_s}
		+\sum_{s\ge0}
		\frac{q^{3s^2+2s}\mathsf A_{3s-1}}{(q^3;q^3)_s}\\
		&=\mathscr S_{\mathsf A}(q^5)
		+\mathscr T_{\mathsf A}(q^2).
	\end{aligned}
	\]
	This proves
	\[
	F_9=(1+q)\mathscr S_{\mathsf A}(q^2)
	+q^2\mathscr S_{\mathsf A}(q^5)
	+q^2\mathscr T_{\mathsf A}(q^2).
	\]

	Finally, multiplying \eqref{eq:F10} by $q^2$ and
	grouping the terms according to $s$, we obtain
	\[
	\begin{aligned}
		q^2F_{10}
		&=(1+q)\sum_{s\ge0}
		\frac{q^{3s^2+2s}}{(q^3;q^3)_s}
		\sum_{r\ge0}q^{r^2-3sr}\qb{3s-r-1}{r}\\
		&\quad
		+q^2\sum_{s\ge0}
		\frac{q^{3s^2+5s}}{(q^3;q^3)_s}
		\sum_{r\ge0}q^{r^2-(3s+1)r}\qb{3s-r}{r}\\
		&=(1+q)\mathscr S_{\mathsf B}(q^2)
		+q^2\sum_{s\ge0}
		\frac{q^{3s^2+5s}\mathsf B_{3s+1}}{(q^3;q^3)_s}. \qquad ({\rm by}\ \eqref{2-17})
	\end{aligned}
	\]
 Applying \eqref{eq:hrec} once more,
	\[
	\mathsf B_{3s+1}
	=\mathsf B_{3s}+q^{-3s}\mathsf B_{3s-1},
	\]
	gives
	\[
	q^2F_{10}
	=(1+q)\mathscr S_{\mathsf B}(q^2)
	+q^2\mathscr S_{\mathsf B}(q^5)
	+q^2\mathscr T_{\mathsf B}(q^2).
	\]

	Both recurrence substitutions include $s=0$:
	the prescribed values $\mathsf A_{-1}=0$ and
	$\mathsf B_{-1}=1$ give
	$\mathsf A_1=\mathsf A_0+\mathsf A_{-1}$ and
	$\mathsf B_1=\mathsf B_0+\mathsf B_{-1}$.
	Thus no initial term is omitted.
\end{proof}

Define the following four series explicitly
\begin{equation}\label{eq:asym-airy-series}
\begin{aligned}
\mathscr U_4&:=\mathscr S_u(q),\\
\mathscr V_4&:=\mathscr S_v(q),\\
\mathscr U_5&:=(1+q)\mathscr S_u(q^2)
 +q^2\mathscr S_u(q^5)+q^2\mathscr T_u(q^2),\\
\mathscr V_5&:=(1+q)\mathscr S_v(q^2)
 +q^2\mathscr S_v(q^5)+q^2\mathscr T_v(q^2).
\end{aligned}
\end{equation}
Substituting \eqref{ainv} term by term in \eqref{eq:asym-series}
gives
\begin{equation}\label{eq:asym-A-combination}
F_{2i-1}
=\frac{\mathsf C_-(\rho)\mathscr U_i
       +\mathsf C_+(\rho)\mathscr V_i}{2},
\qquad i=4,5.
\end{equation}

Define
\begin{equation*}
	N_4:=\N(0,1),\qquad N_5:=\N(0,2)+q^2\N(-2,5).
\end{equation*}

\begin{proposition}\label{connections}
For $i=4,5$, the series in \eqref{eq:asym-airy-series} satisfy
\begin{equation}\label{system}
\begin{pmatrix}\tau_1&\tau_2\\\tau_3&\tau_4\end{pmatrix}
\binom{\mathscr U_i}{\mathscr V_i}
=2N_i\binom{J_1}{J_3},
\end{equation}
The determinant of the coefficient matrix is nonzero.
\end{proposition}

\begin{proof}
For $i=4$, \eqref{firstS} and \eqref{secondS} at $k=1$ give the
two rows directly. For $i=5$, substitution of
\eqref{eq:asym-airy-series} into \eqref{firstS} and \eqref{firstT}
gives
\[
\frac{\tau_1\mathscr U_5+\tau_2\mathscr V_5}{2J_1}
=(1+q)\N(0,2)+q^2\N(0,5)+q^2\N(1,2).
\]
For the second row, use \eqref{secondS}, \eqref{shiftS}, and
\eqref{shiftT} at $k=2$, where
$(\mathsf c_2,\mathsf d_2)=(\tau_3,\tau_4)$. This gives
\[
\frac{\tau_3\mathscr U_5+\tau_4\mathscr V_5}{2J_3}
=(1+q)\N(1,-1)+q^6\N(4,-7)-q^3\N(2,-4).
\]
The contiguous relations \eqref{cont1}--\eqref{cont2} imply
\begin{align*}
(1+q)\N(1,-1)&+q^6\N(4,-7)-q^3\N(2,-4)
  =\N(1,-1)+q\N(2,-1),\\
\N(1,-1)&=q\N(0,2)+(1+q^2)\N(1,2)+q^2\N(0,5),\\
q\N(2,-1)&=\N(0,2)-\N(1,2),\\
\N(-2,5)&=q^{-1}\N(0,2)+\N(1,2)+\N(0,5).
\end{align*}
Consequently, both row expressions equal
\begin{equation}\label{eq:asym-N5-reduction}
(1+q)\N(0,2)+q^2\N(0,5)+q^2\N(1,2)
=\N(0,2)+q^2\N(-2,5)=N_5.
\end{equation}
This proves \eqref{system}. Equations  
\eqref{binaryeval} and \eqref{dettheta}  give the determinant $-6\rho+O(\rho^5)$.
\end{proof}

\begin{proof}[Proof of Theorem~\ref{thm:asym}]
The dual evaluations in \cite[Theorems 1.2--1.3]{Xia}, after
grouping the appropriate residue classes, give
\begin{equation}\label{N45}
N_4=P_4\frac{J_3^2}{J_1J_9}=P_4\Delta,\qquad
N_5=P_5\frac{J_3^2}{J_1J_9}=P_5\Delta,
\end{equation}
where \eqref{Delta} gives the second equality in each case.
Apply Proposition~\ref{prop:multiplier} to \eqref{system} with
$(\mathscr U,\mathscr V)=(\mathscr U_i,\mathscr V_i)$,
$\nu=N_i$, and $\lambda=1$. Equations  
\eqref{v-1}, \eqref{eq:asym-A-combination}, and \eqref{N45} yield
\[
F_{2i-1}=N_i\mathcal M(1)
=N_i\frac{G(q^9)}\Delta=P_iG(q^9),\qquad i=4,5.
\]

To evaluate the remaining two sums, we use the Schur
decompositions already established in
\eqref{h-8} and \eqref{h-10}--\eqref{h-11}.
Taking \(a=0\) and \(b=k\) in \eqref{h-8} gives
\begin{equation}\label{eq:general-Schur-S}
	\N(0,k)
	=
	G(q)\mathscr S_{\mathsf A}(q^k)
	+
	H(q)\mathscr S_{\mathsf B}(q^k).
\end{equation}

For \(a=1\), the same argument as in
\eqref{h-10}--\eqref{h-11}, with the additional factor

\(q^{ks}\), gives
\[
\begin{aligned}
	\N(1,k)
	&=
	H(q)
	+\sum_{s\ge1}
	\frac{q^{3s^2+ks}}{(q^3;q^3)_s}
	\left\{
	G(q)\mathsf A_{3s-1}
	+H(q)\mathsf B_{3s-1}
	\right\}  \\
	&=
	G(q)\mathscr T_{\mathsf A}(q^k)
	+
	H(q)\mathscr T_{\mathsf B}(q^k).
\end{aligned}
\]
Indeed, for \(s\ge1\), identity \eqref{eq:schur} applies with
\(m=3s-1\). For \(s=0\), the contribution to \(\N(1,k)\) is
$H(q)$, 
whereas the corresponding contribution on the right is
\[
G(q)\mathsf A_{-1}+H(q)\mathsf B_{-1}=H(q),
\]
since \(\mathsf A_{-1}=0\) and \(\mathsf B_{-1}=1\).
Thus
\begin{equation}\label{eq:general-Schur-T}
	\N(1,k)
	=
	G(q)\mathscr T_{\mathsf A}(q^k)
	+
	H(q)\mathscr T_{\mathsf B}(q^k).
\end{equation}

%
%
%
%

Combining these equalities with \eqref{eq:asym-series} and
\eqref{eq:asym-N5-reduction} gives
\[
G(q)F_{2i-1}+q^2H(q)F_{2i}=N_i=P_i\Delta,
\qquad i=4,5.
\]
Substitute the already established value of $F_{2i-1}$ and use
\eqref{Delta}. We obtain
\[q^2H(q)(F_{2i}-P_iH(q^9))=0.\]
Since $\mathbb C[[q]]$ has no zero divisors and $H(q)$ has
constant term $1$,
$$F_{2i}=P_iH(q^9).$$
These are precisely
\eqref{eq:R7}--\eqref{eq:R10} for \(0<q<1\).

It remains to remove the restriction that \(q\) be real.  Each side of
\eqref{eq:R7}--\eqref{eq:R10} is holomorphic in the unit disk
$
\mathbb D:=\{q\in\mathbb C:|q|<1\}.
$
Indeed, the defining series converge locally uniformly on \(\mathbb D\),
whereas the infinite products on the right-hand sides converge locally
uniformly there and their denominator factors do not vanish.  Since the
identities hold for every \(q\in(0,1)\), the identity theorem extends
them to the whole disk \(|q|<1\).
\end{proof}

\section{The four Hickerson identities}\label{sec:hickerson}

For $\zeta\in\{\omega,\omega^2\}$, define
\begin{equation}\label{4-19}
 \mathsf H_{1,\zeta}:=\frac{\Ec}{(\zeta^2q,\zeta q^3;q^3)_\infty},\qquad
 \mathsf H_{2,\zeta}:=\frac{\Ec}{(\zeta q^2,\zeta^2q^3;q^3)_\infty}.
\end{equation}

\begin{lemma}\label{lem:hickerson-first}
For $\zeta\in\{\omega,\omega^2\}$,
\[
 \mathsf H_{1,\zeta}=\mathcal S(1,1)-\zeta q\mathcal S(2,4).
\]
\end{lemma}
\begin{proof}

For $n\ge0$ and $0\le j\le n$, define the nodes
\begin{equation*}
 z_{n,j}^{(2)}:=\omega q^{n+2+j}+\omega^2q^{2n+2-j}.
\end{equation*}
Then define
\begin{equation*}
 Z_n^{(2)}:=\Ec(q;q)_n
 \mathcal P_{3,q^{3n+4}}[z_{n,0}^{(2)},\ldots,z_{n,n}^{(2)}].
\end{equation*}
Xia
\cite[Theorem 5.1]{Xia} proved that
\begin{equation}\label{XiaZ2}
 Z_n^{(2)}=(1+q^{n+1})\mathcal S(2n+2,3n+4)
 +q^{2n+2}\mathcal S(2n+3,3n+7).
\end{equation}
Setting $n=0$ in \eqref{XiaZ2}, we obtain
\[
Z_0^{(2)}
=(1+q)\mathcal S(2,4)+q^2\mathcal S(3,7).
\]
On the other hand, \eqref{Scont1} with $(a,b)=(1,4)$ gives
\[
q^2\mathcal S(3,7)
=\mathcal S(1,4)-\mathcal S(2,4).
\]
Consequently,
\[
Z_0^{(2)}
=\mathcal S(1,4)+q\mathcal S(2,4).
\]

We now evaluate $Z_0^{(2)}$ from its defining divided
difference. By \eqref{eq:kernel-unified},
\begin{align*}
Z_0^{(2)}
&=\mathcal E_3\,
\mathcal P_{3,q^4}(\omega q^2+\omega^2q^2)
\\
&=\frac{\mathcal E_3}
{(\omega q^2,\omega^2q^2;q^3)_\infty}
\\
	&=\frac{J_3}{J_1}
	\frac{(q^2;q^3)_\infty}{(q^6;q^9)_\infty}=P_5,
\end{align*}
 Thus
\[
P_5
=\mathcal S(1,4)+q\mathcal S(2,4)
=\frac{\mathcal E_3}
{(\omega q^2,\omega^2q^2;q^3)_\infty}.
\]

Applying \eqref{Scont1} at $(a,b)=(1,4)$ and
\eqref{Scont2} at $(a,b)=(1,1)$, respectively, yields
\begin{align}
	q^2\mathcal S(3,7)
	&=\mathcal S(1,4)-\mathcal S(2,4)
	=P_5-(1+q)\mathcal S(2,4),\label{m-1}\\
	q^4\mathcal S(4,7)
	&=\mathcal S(1,1)-\mathcal S(1,4)
	=\mathcal S(1,1)+q\mathcal S(2,4)-P_5.\label{m-2}
\end{align}
For $n=1$, \eqref{XiaZ2} reads
\begin{align}\label{m-3}
Z_1^{(2)}
=(1+q^2)\mathcal S(4,7)+q^4\mathcal S(5,10).
\end{align}
By \eqref{Scont1} at $(a,b)=(3,7)$,
\[
q^4\mathcal S(5,10)
=\mathcal S(3,7)-\mathcal S(4,7),
\]
 which, together with \eqref{m-3}, yields 
\begin{align}\label{m-4}
Z_1^{(2)}=q^2\mathcal S(4,7)+\mathcal S(3,7).
\end{align}
Substituting \eqref{m-1} and \eqref{m-2} into \eqref{m-4}, we obtain
\begin{equation}\label{H1low1}
	Z_1^{(2)}
	=\frac{\mathcal S(1,1)-\mathcal S(2,4)}{q^2}.
\end{equation}

For $n=2$, \eqref{XiaZ2} gives
\[
Z_2^{(2)}
=(1+q^3)\mathcal S(6,10)+q^6\mathcal S(7,13).
\]
Applying \eqref{Scont1} at $(a,b)=(5,10)$ gives
\[
q^6\mathcal S(7,13)
=\mathcal S(5,10)-\mathcal S(6,10),
\]
so
\[
Z_2^{(2)}=q^3\mathcal S(6,10)+\mathcal S(5,10).
\]

We now evaluate the two series on the right.
First, \eqref{Scont1} at $(a,b)=(2,4)$ gives
\[
\begin{aligned}
	q\mathcal S(3,4)
	&=q\mathcal S(2,4)-q^4\mathcal S(4,7)\\
	&=P_5-\mathcal S(1,1).
\end{aligned}
\]
Using \eqref{Scont2} at $(a,b)=(3,4)$, we then obtain
\[
\begin{aligned}
	q^9\mathcal S(6,10)
	&=q^2\bigl(\mathcal S(3,4)-\mathcal S(3,7)\bigr)\\
	&=q\bigl(P_5-\mathcal S(1,1)\bigr)
	-\bigl(P_5-(1+q)\mathcal S(2,4)\bigr)\\
	&=-q\mathcal S(1,1)
	+(1+q)\mathcal S(2,4)+(q-1)P_5.
\end{aligned}
\]
Also, the earlier relation for $\mathcal S(5,10)$ yields
\[
\begin{aligned}
	q^8\mathcal S(5,10)
	&=q^4\mathcal S(3,7)-q^4\mathcal S(4,7)\\
	&=q^2\bigl(P_5-(1+q)\mathcal S(2,4)\bigr)
	-\bigl(\mathcal S(1,1)+q\mathcal S(2,4)-P_5\bigr)\\
	&=-\mathcal S(1,1)
	-(q+q^2+q^3)\mathcal S(2,4)+(1+q^2)P_5.
\end{aligned}
\]
Consequently,
\begin{equation}\label{H1low2}
	Z_2^{(2)}
	=\frac{(1+q^3)\{P_5-\mathcal S(1,1)\}
		-q\mathcal S(2,4)}{q^8}.
\end{equation}

By \eqref{eq:divided-difference},
\begin{align*}
Z_n^{(2)}
=(q;q)_n
\sum_{j=0}^{n}
\frac{\mathcal E_3\,
	\mathcal P_{3,q^{3n+4}}(z_{n,j}^{(2)})}
{\displaystyle\prod_{\substack{0\le k\le n\\k\ne j}}
	(z_{n,j}^{(2)}-z_{n,k}^{(2)})}.
\end{align*}
For $n=1$,
\begin{align}\label{4-13}
Z_1^{(2)}
=
\frac{
	(1-\omega^2q)\mathsf H_{1,\omega}
	-(1-\omega q)\mathsf H_{1,\omega^2}}
{q^3(\omega-\omega^2)},
\end{align}

Define
\begin{align}
	e_1&:=\mathsf H_{1,\omega}
	-\mathcal S(1,1)+\omega q\mathcal S(2,4), \label{4-14}\\
	e_2&:=\mathsf H_{1,\omega^2}
	-\mathcal S(1,1)+\omega^2q\mathcal S(2,4). \label{4-15}
\end{align}
By \eqref{4-13}, \eqref{4-14}, and \eqref{4-15},
\[
Z_1^{(2)}
=
\frac{\mathcal S(1,1)-\mathcal S(2,4)}{q^2}
+
\frac{(1-\omega^2q)e_1-(1-\omega q)e_2}
{q^3(\omega-\omega^2)}.
\]
Comparison with \eqref{H1low1} yields
\[
(1-\omega^2q)e_1-(1-\omega q)e_2=0.
\]

For $n=2$,
\begin{align}
q^8Z_2^{(2)}
=(1+q^3)P_5
+\frac{
	(\omega^2q^3-\omega)\mathsf H_{1,\omega}
	+(\omega^2-\omega q^3)\mathsf H_{1,\omega^2}}
{\omega-\omega^2}.\label{4-16}
\end{align}
Substituting \eqref{4-14}
 and \eqref{4-15} into \eqref{4-16} yields
\[
\begin{aligned}
	q^8Z_2^{(2)}
	&=(1+q^3)\bigl(P_5-\mathcal S(1,1)\bigr)
	-q\mathcal S(2,4)\\
	&\quad+
	\frac{(\omega^2q^3-\omega)e_1
		+(\omega^2-\omega q^3)e_2}
	{\omega-\omega^2}.
\end{aligned}
\]
Comparison with \eqref{H1low2} gives
\[
(\omega^2q^3-\omega)e_1
+(\omega^2-\omega q^3)e_2=0.
\]

We have therefore obtained the homogeneous system
\[
\begin{pmatrix}
	1-\omega^2q & -(1-\omega q)\\
	\omega^2q^3-\omega & \omega^2-\omega q^3
\end{pmatrix}
\binom{e_1}{e_2}
=\binom00.
\]
The determinant is $-(2\omega+1)(1+q+q^3)$, which has a nonzero
constant term. Hence $e_1=e_2=0$ in $\mathbb Q(\omega)[[q]]$. Thus,
 \begin{align*}
  \mathsf H_{1,\omega}
 	&=\mathcal S(1,1)-\omega q\mathcal S(2,4), \\
 \mathsf H_{1,\omega^2}
 	&=\mathcal S(1,1)-\omega^2q\mathcal S(2,4). 
 \end{align*}
This proves the two asserted identities.
\end{proof}

\begin{lemma}\label{lem:hickerson-second}
For $\zeta\in\{\omega,\omega^2\}$,
\[
 \mathsf H_{2,\zeta}=-\zeta^2\mathcal S(0,-1)-\zeta\mathcal S(0,2).
\]
\end{lemma}
\begin{proof}
For $n\ge0$ and $0\le j\le n$, define
\begin{align*}
 z_{n,j}^{(1)}&:=\omega q^{n+1+j}+\omega^2q^{2n+1-j},\\
 Z_n^{(1)}&:=\Ec(q;q)_n
 \mathcal P_{3,q^{3n+2}}[z_{n,0}^{(1)},\ldots,z_{n,n}^{(1)}].
\end{align*}
The companion evaluation in \cite[Theorem 5.1]{Xia} gives
\begin{equation}\label{XiaZ1}
 Z_n^{(1)}=\mathcal S(2n+1,3n+2).
\end{equation}
The case $n=0$ of \eqref{XiaZ1} gives
\begin{align}\label{r-14}
Z_0^{(1)}=\mathcal S(1,2)=P_4
=\frac{\Ec}{(\omega q,\omega^2q;q^3)_\infty}.
\end{align}
Hence, the cases $n=1$ and $n=2$ of \eqref{XiaZ1} yield 
\begin{equation}\label{H2low}
	\begin{aligned}
		Z_1^{(1)}&=\frac{\mathcal S(0,-1)-\mathcal S(0,2)}{q^2},\\
		Z_2^{(1)}&=\frac{(1+q)P_4-q\mathcal S(0,-1)
			+(q-1)\mathcal S(0,2)}{q^6}.
	\end{aligned}
\end{equation}
Indeed, applying \eqref{Scont1} at $(a,b)=(1,2)$ and \eqref{Scont2} at
$(a,b)=(0,-1)$, respectively, gives
\begin{align}
	\mathcal S(1,2)- \mathcal S(2,2)=q^2 \mathcal S(3,5)
	=\mathcal S(0,-1)-\mathcal S(0,2).\label{r-11}
\end{align}
Setting $n=1$ in \eqref{XiaZ1} and using \eqref{r-11}, we obtain 
the  first equality of \eqref{H2low}.

For the second, two uses of \eqref{Scont1} give
\begin{align}\label{r-13}
 \mathcal S(5,8)=q^{-4}\{\mathcal S(3,5)-\mathcal S(4,5)\},
 \qquad \mathcal S(4,5)=q^{-3}\{\mathcal S(2,2)-\mathcal S(3,2)\}.
\end{align}
By \eqref{r-14} and \eqref{r-11},  
\begin{align}\label{r-16}
\mathcal S(2,2)=P_4+\mathcal S(0,2)-\mathcal S(0,-1).
\end{align}
Equation \eqref{Sfrec} at $(a,b)=(2,2)$ gives
\begin{align}\label{r-12}
	\mathcal S(3,2)
	&=-q\mathcal S(0,2)
	+(1+q+q^2)\mathcal S(2,2)
	\nonumber\\
	&=-q\mathcal S(0,2)
	+(1+q+q^2)
	\bigl(P_4+\mathcal S(0,2)-\mathcal S(0,-1)\bigr).
\end{align}
Setting $n=2$ in \eqref{XiaZ1} and then substituting
\eqref{r-11}, \eqref{r-16}, and \eqref{r-12} into
\eqref{r-13}, we obtain the second identity in
\eqref{H2low}.

For \(n=1\), the two nodes are
\[
z_{1,0}^{(1)}=\omega q^2+\omega^2q^3,
\qquad
z_{1,1}^{(1)}=\omega q^3+\omega^2q^2,
\]
and hence
\[
z_{1,0}^{(1)}-z_{1,1}^{(1)}
=q^2(1-q)(\omega-\omega^2).
\]
Moreover,
\[
\mathcal E_3
\mathcal P_{3,q^5}\bigl(z_{1,0}^{(1)}\bigr)
=\mathsf H_{2,\omega},
\qquad
\mathcal E_3
\mathcal P_{3,q^5}\bigl(z_{1,1}^{(1)}\bigr)
=\mathsf H_{2,\omega^2}.
\]
Therefore, the case $n=1$ gives
\begin{align}
	Z_1^{(1)}
	&=(1-q)\left\{
	\frac{\mathsf H_{2,\omega}}
	{z_{1,0}^{(1)}-z_{1,1}^{(1)}}
	+
	\frac{\mathsf H_{2,\omega^2}}
	{z_{1,1}^{(1)}-z_{1,0}^{(1)}}
	\right\} \notag\\
	&=\frac{\mathsf H_{2,\omega}
		-\mathsf H_{2,\omega^2}}
	{q^2(\omega-\omega^2)}.
	\label{H2div1}
\end{align}

For \(n=2\), the three nodes are
\[
z_{2,0}^{(1)}=\omega q^3+\omega^2q^5,\qquad
z_{2,1}^{(1)}=-q^4,\qquad
z_{2,2}^{(1)}=\omega q^5+\omega^2q^3.
\]
Their differences are
\begin{align*}
	z_{2,0}^{(1)}-z_{2,1}^{(1)}
	&=q^3(1-q)(\omega-\omega^2q),\\
	z_{2,0}^{(1)}-z_{2,2}^{(1)}
	&=q^3(1-q^2)(\omega-\omega^2),\\
	z_{2,1}^{(1)}-z_{2,2}^{(1)}
	&=q^3(1-q)(\omega q-\omega^2).
\end{align*}
The corresponding product values are
\begin{align*}
	\mathcal E_3
	\mathcal P_{3,q^8}\bigl(z_{2,0}^{(1)}\bigr)
	&=(1-\omega^2q^2)\mathsf H_{2,\omega^2},\\
	\mathcal E_3
	\mathcal P_{3,q^8}\bigl(z_{2,1}^{(1)}\bigr)
	&=(1+q+q^2)P_4,\\
	\mathcal E_3
	\mathcal P_{3,q^8}\bigl(z_{2,2}^{(1)}\bigr)
	&=(1-\omega q^2)\mathsf H_{2,\omega}.
\end{align*}
Consequently, the case $n=2$ yields
\begin{align*}
	Z_2^{(1)}
	&=(q;q)_2
	\left\{
	\frac{(1-\omega^2q^2)\mathsf H_{2,\omega^2}}
	{(z_{2,0}^{(1)}-z_{2,1}^{(1)})
		(z_{2,0}^{(1)}-z_{2,2}^{(1)})}
	\right.\\
	&\qquad\left.
	+\frac{(1+q+q^2)P_4}
	{(z_{2,1}^{(1)}-z_{2,0}^{(1)})
		(z_{2,1}^{(1)}-z_{2,2}^{(1)})}
	+\frac{(1-\omega q^2)\mathsf H_{2,\omega}}
	{(z_{2,2}^{(1)}-z_{2,0}^{(1)})
		(z_{2,2}^{(1)}-z_{2,1}^{(1)})}
	\right\}.
\end{align*}
Using \(\omega^2+\omega+1=0\), this simplifies to
\begin{equation}
	q^6Z_2^{(1)}
	=(1+q)P_4
	-\frac{q+\omega}{\omega-\omega^2}\mathsf H_{2,\omega}
	+\frac{q+\omega^2}{\omega-\omega^2}
	\mathsf H_{2,\omega^2}.
	\label{H2div2}
\end{equation}

Define
\begin{align*}
	e_1&:=\mathsf H_{2,\omega}
	+\omega^2\mathcal S(0,-1)
	+\omega\mathcal S(0,2),\\
	e_2&:=\mathsf H_{2,\omega^2}
	+\omega\mathcal S(0,-1)
	+\omega^2\mathcal S(0,2).
\end{align*}
Substitution into \eqref{H2div1} gives
\[
Z_1^{(1)}
=\frac{\mathcal S(0,-1)-\mathcal S(0,2)}{q^2}
+\frac{e_1-e_2}{q^2(\omega-\omega^2)}.
\]
Comparison with the first identity in \eqref{H2low} therefore gives
$
e_1-e_2=0.
$

Similarly, substitution into \eqref{H2div2} gives
\[
\begin{aligned}
	q^6Z_2^{(1)}
	&=(1+q)P_4-q\mathcal S(0,-1)
	+(q-1)\mathcal S(0,2)\\
	&\quad+
	\frac{-(q+\omega)e_1+(q+\omega^2)e_2}
	{\omega-\omega^2}.
\end{aligned}
\]
Comparison with the second identity in \eqref{H2low} yields
\[
-(q+\omega)e_1+(q+\omega^2)e_2=0.
\]
Thus
\[
\begin{pmatrix}
	1&-1\\
	-(q+\omega)&q+\omega^2
\end{pmatrix}
\binom{e_1}{e_2}
=\binom00.
\]
The determinant of the coefficient matrix is
 \(\omega^2-\omega\ne0\). 
Hence \(e_1=e_2=0\), and consequently
\[
\mathsf H_{2,\omega}
=-\omega^2\mathcal S(0,-1)-\omega\mathcal S(0,2),
\]
and
\[
\mathsf H_{2,\omega^2}
=-\omega\mathcal S(0,-1)-\omega^2\mathcal S(0,2).
\]
These are precisely the asserted identities. 
\end{proof}

\begin{proof}[Proof of Theorem~\ref{Hthm}]
Using \eqref{cyclo} in the definitions \eqref{4-19} gives
\begin{align*}
 \mathsf H_{1,\zeta}
 &=\frac{(q^6;q^9)_\infty(\zeta q,\zeta^2q^3;q^3)_\infty}
 {(q^2;q^3)_\infty},\\
 \mathsf H_{2,\zeta}
 &=\frac{(q^3;q^9)_\infty(\zeta^2q^2,\zeta q^3;q^3)_\infty}
 {(q;q^3)_\infty}.
\end{align*}
Lemma~\ref{lem:hickerson-first} now gives \eqref{H1}.
Multiplying the identity in Lemma~\ref{lem:hickerson-second} by
$-\zeta$ gives \eqref{H2}, since $\zeta^3=1$.
\end{proof}

\section{The four reflected identities with cube roots of unity}\label{sec:cyclotomic}
It suffices to treat $\zeta=\omega$ and then apply conjugation.

We prove Theorem~\ref{Kthm}, using the Airy identities already
established in Section~\ref{sec:asymmetric}. We first evaluate the
two required combinations of Nahm sums.

\begin{lemma}\label{duallemma}
One has
\begin{align}
 \mathcal N_{1,\omega}:=\N(1,-1)-\omega\N(0,2)
  &=(1-\omega)\Ec(\omega^2q^2,\omega q^3;q^3)_\infty,\label{dual1}\\
 \mathcal N_{2,\omega}:=\N(1,-2)+\omega^2q\N(-2,4)
  &=(1-\omega)\Ec(\omega q,\omega^2q^3;q^3)_\infty.\label{dual2}
\end{align}
\end{lemma}

\begin{proof}
For $\tau=1,2$, $n\ge0$, and $0\le j\le n$, define
\[
 \xi_{n,j}^{(\tau)}:=\omega q^{3-\tau-2n+j}
 +\omega^2q^{3-\tau-n-j}.
\]
Using these nodes, define
\begin{equation}\label{v-2}
 V_n^{(\tau)}:=\Ec(-1)^n q^{-3n(n-1)/2}(q;q)_n
 \mathcal D_{3,q^{6-2\tau-3n}}
 [\xi_{n,0}^{(\tau)},\ldots,\xi_{n,n}^{(\tau)}].
\end{equation}
The divided-difference identities proved in
\cite[Theorems 8.1 and 9.4]{Xia} give
\begin{align}
 V_n^{(1)}&=\N(n,1),\label{XiaV1}\\
 V_n^{(2)}&=\N(n,2)+q^2\N(n-2,5).\label{XiaV2}
\end{align}
We need only $n=0,1,2$. We evaluate the products at the nodes using
\[
\mathcal D_{3,\alpha\beta}(\alpha+\beta)
=(\alpha,\beta;q^3)_\infty,
\qquad
(a;q^3)_\infty=(1-a)(aq^3;q^3)_\infty.
\]
We use \eqref{eq:divided-difference} to evaluate these expressions.
The nodes used below are distinct for $0<|q|<1$, as is also clear
from the factorizations of their differences.

We first prove \eqref{dual2}.  For $\tau=1$ put
\[
 \mathsf Q_{1,0}:=\Ec(\omega q^2,\omega^2q^2;q^3)_\infty,\quad
 \mathsf Q_{1,1}:=\Ec(\omega q,\omega^2q^3;q^3)_\infty,\quad \mathsf Q_{1,2} :=\Ec(\omega^2 q,\omega q^3;q^3)_\infty.
\]
In view of \eqref{v-2} and  \eqref{XiaV1},
\[
\mathfrak N(n,1)
=
\mathcal E_3(-1)^n q^{-3n(n-1)/2}(q;q)_n
\mathcal D_{3,q^{4-3n}}
[\xi_{n,0}^{(1)},\ldots,\xi_{n,n}^{(1)}].
\]
For \(n=0\),
\begin{align} \label{dualbasis1-1}
\mathfrak N(0,1)
=\mathcal E_3
\mathcal D_{3,q^4}(\omega q^2+\omega^2q^2)
=\mathcal E_3(\omega q^2,\omega^2q^2;q^3)_\infty
=\mathsf Q_{1,0}.
\end{align}
For $n=1$,    
\[
\xi_{1,0}^{(1)}=\omega+\omega^2q,\qquad \xi_{1,1}^{(1)}=\omega q+\omega^2,
\qquad \xi_{1,0}^{(1)}-\xi_{1,1}^{(1)}=(1-q)(\omega-\omega^2).
\]
The product values are
\[
\Ec\mathcal D_{3,q}(\xi_{1,0}^{(1)})=(1-\omega)\mathsf Q_{1,2},
\qquad
\Ec\mathcal D_{3,q}(\xi_{1,1}^{(1)})=(1-\omega^2)\mathsf Q_{1,1}.
\]
Multiplying their first divided difference by $-\Ec(1-q)$ gives
\begin{align} \label{dualbasis1-2}
	\mathfrak N(1,1)
	&=-\frac{(1-\omega)\mathsf Q_{1,2}
		-(1-\omega^2)\mathsf Q_{1,1}}
	{\omega-\omega^2} \nonumber \\
	&=-\omega\mathsf Q_{1,1}
	-\omega^2\mathsf Q_{1,2}.
\end{align}
By \eqref{eq:kernel-unified}, 
\begin{align*}
\Ec\mathcal D_{3,q^{-2}}(\xi_{2,0}^{(1)})
 &=(1-\omega q^{-2})(1-\omega^2)\mathsf Q_{1,1},\\
\Ec\mathcal D_{3,q^{-2}}(\xi_{2,1}^{(1)})
 &=(1-\omega q^{-1})(1-\omega^2q^{-1})\mathsf Q_{1,0},\\
\Ec\mathcal D_{3,q^{-2}}(\xi_{2,2}^{(1)})
 &=(1-\omega)(1-\omega^2q^{-2})\mathsf Q_{1,2}.
\end{align*}
 Consequently, writing $x_j:=\xi_{2,j}^{(1)}$ for $0\le j\le2$,
\begin{align} \label{dualbasis1-3}
q\mathfrak N(2,1)
&=
\mathcal E_3q^{-2}(1-q)(1-q^2)
\sum_{j=0}^{2}
\frac{\mathcal D_{3,q^{-2}}(x_j)}
{\displaystyle\prod_{\substack{0\le i\le2,\\i\ne j}}
	(x_j-x_i)}\nonumber\\
&=(1+q)\mathsf Q_{1,0}
+(q\omega^2+\omega)\mathsf Q_{1,1}
+(q\omega+\omega^2)\mathsf Q_{1,2}.
\end{align}

To obtain \eqref{dual2}, use \eqref{cont1} at $(a,b)=(-1,1)$
and \eqref{Nfrec} at $(a,b)=(-1,1)$:
\begin{align*}
\N(1,-2)&=\N(-1,1)-\N(0,1),\\
q\N(-1,1)&=(1+2q)\N(0,1)-(1-q)\N(1,1)-q\N(2,1).
\end{align*}
Eliminating $\N(-1,1)$ gives
\[
\N(1,-2)=(1+q^{-1})\N(0,1)
 +(1-q^{-1})\N(1,1)-\N(2,1).
\]
Substituting \eqref{dualbasis1-1}--\eqref{dualbasis1-3}
 into the above identity, we arrive at 
\[
\N(1,-2)=\mathsf Q_{1,1}+\mathsf Q_{1,2}.
\]
Furthermore, \eqref{cont2} at $(a,b)=(1,-2)$ gives
$$q\N(-2,4)=\N(1,-2)-\N(1,1).$$
  Therefore, 
\begin{align*}
\N(1,-2)+\omega^2q\N(-2,4)
 &=(1+\omega^2)(\mathsf Q_{1,1}+\mathsf Q_{1,2})
   -\omega^2(-\omega\mathsf Q_{1,1}-\omega^2\mathsf Q_{1,2})\\
 &=(1-\omega)\mathsf Q_{1,1},
\end{align*}
which is \eqref{dual2}.

For $\tau=2$, define
\[
 \mathsf Q_{2,0}:=\Ec(\omega q,\omega^2q;q^3)_\infty,\quad
 \mathsf Q_{2,1}:=\Ec(\omega q^2,\omega^2q^3;q^3)_\infty,\quad \mathsf Q_{2,2}:=\Ec(\omega^2 q^2,\omega q^3;q^3)_\infty.
\]
First apply \eqref{Nup} with $(a,b)=(n-2,2)$ to obtain
\[
q^2\N(n-2,5)
=q^{-n}\{\N(n-1,2)-\N(n,2)\}-\N(n+1,2).
\]
Together with \eqref{XiaV2}, this gives
\[
V_n^{(2)}=(1-q^{-n})\N(n,2)+q^{-n}\N(n-1,2)-\N(n+1,2).
\]
The recurrence \eqref{Nfrec} at $(a,b)=(-1,2)$ and $(0,2)$ yields
\begin{align*}
\N(-1,2)&=(1+q+q^{-1})\N(0,2)
 +(1-q^{-1})\N(1,2)-\N(2,2),\\
\N(3,2)&=-q^{-1}\N(0,2)
 +(q^{-2}+q^{-1}+q^2)\N(1,2)+(1-q^{-2})\N(2,2).
\end{align*}
Substituting these expressions at $n=0,1,2$, respectively, gives
\begin{align*}
 V_0^{(2)}&=(1+q+q^{-1})\N(0,2)-q^{-1}\N(1,2)-\N(2,2),\\
 V_1^{(2)}&=q^{-1}\N(0,2)+(1-q^{-1})\N(1,2)-\N(2,2),\\
 V_2^{(2)}&=q^{-1}\N(0,2)-(q^2+q^{-1})\N(1,2).
\end{align*}
Solving the linear system gives
\begin{align}
\N(0,2)
	&=\frac{(q^2+q^{-1})
		\bigl(V_0^{(2)}-V_1^{(2)}\bigr)-V_2^{(2)}}
	{1+q^2+q^3},\label{h-41}\\[4pt]
	\N(1,2)
	&=\frac{q^{-1}\bigl(V_0^{(2)}-V_1^{(2)}\bigr)
		-(1+q)V_2^{(2)}}
	{1+q^2+q^3}, \label{h-42}\\[4pt]
	\N(2,2)
	&
	=\frac{(1+q^2)V_0^{(2)}
		-(1+q+q^2+q^3+q^4)V_1^{(2)}
		-q^2V_2^{(2)}}
	{q(1+q^2+q^3)}. \label{h-43}
\end{align}

We next evaluate their product sides. At $n=0$,
\begin{align}
V_0^{(2)}=\Ec\mathcal D_{3,q^2}(\omega q+\omega^2q)
=\mathsf Q_{2,0}. \label{h-44}
\end{align}
At $n=1$, 
\begin{align}
V_1^{(2)}
 &=-\frac{(q-\omega)(1-\omega^2)}{\omega-\omega^2}\mathsf Q_{2,1}
   +\frac{(q-\omega^2)(1-\omega)}{\omega-\omega^2}\mathsf Q_{2,2}\nonumber \\
 &=(q\omega-\omega^2)\mathsf Q_{2,1}
   +(q\omega^2-\omega)\mathsf Q_{2,2}.\label{h-45}
\end{align}
At $n=2$,  
\begin{align}\label{h-46}
V_2^{(2)}=q^{-1}\{(1+q^3)\mathsf Q_{2,0}
 -(q^3+\omega+1)\mathsf Q_{2,1}-(q^3-\omega)\mathsf Q_{2,2}\}.
\end{align}
Substituting 
 \eqref{h-44}--\eqref{h-46}
  into \eqref{h-41}--\eqref{h-43} yields 
\begin{align}
 \N(0,2)&=-\omega \mathsf Q_{2,1}-\omega^2\mathsf Q_{2,2},\notag\\
 \N(1,2)&=-\mathsf Q_{2,0}+\mathsf Q_{2,1}+\mathsf Q_{2,2},\label{dualbasis2}\\
 q\N(2,2)&=(1-q)\mathsf Q_{2,0}+\{\omega^2-\omega q(1+q)\}\mathsf Q_{2,1}
       +\{\omega-\omega^2q(1+q)\}\mathsf Q_{2,2}.\notag
\end{align}
Lastly, \eqref{cont1} at $(a,b)=(-1,2)$ gives
$$\N(1,-1)=\N(-1,2)-\N(0,2).$$
 Using the expression for
$\N(-1,2)$ above, we obtain
\[
 \N(1,-1)=(q+q^{-1})\N(0,2)+(1-q^{-1})\N(1,2)-\N(2,2).
\]
Substitution of \eqref{dualbasis2} gives
\[
\N(1,-1)=-\omega^2\mathsf Q_{2,1}-\omega\mathsf Q_{2,2}.
\]
Consequently,
\begin{align*}
\N(1,-1)-\omega\N(0,2)
 &=-\omega^2\mathsf Q_{2,1}-\omega\mathsf Q_{2,2}
   +\omega^2\mathsf Q_{2,1}+\omega^3\mathsf Q_{2,2}\\
 &=(1-\omega)\mathsf Q_{2,2},
\end{align*}
which proves \eqref{dual1}. Coefficientwise conjugation, fixing $q$
and interchanging $\omega$ and $\omega^2$, gives both conjugate identities.
\end{proof}

\begin{lemma}\label{lem:cyclotomic-series}
The two reflected sums have the representations
\begin{align}
\mathcal K_{1,\omega}
 &=\mathscr T_{\mathsf A}(q^{-1})
   -\omega\mathscr S_{\mathsf A}(q^2),\label{L1}\\
\mathcal K_{2,\omega}
 &=\mathscr T_{\mathsf A}(q^{-2})
   +\omega^2\bigl\{\mathscr S_{\mathsf A}(q)
   +q\mathscr S_{\mathsf A}(q^4)
   +q\mathscr T_{\mathsf A}(q)\bigr\}.\label{L2}
\end{align}
\end{lemma}

\begin{proof}
By \eqref{Tdef}, we obtain, for $b=-1,-2$,
	\begin{align}
		\mathcal T_{1,b}^{(-1)}
		&=\sum_{s\ge1}\frac{q^{3s^2+bs}}{(q^3;q^3)_s}
		\sum_{r\ge0}q^{r^2-(3s-1)r}
		\genfrac{[}{]}{0pt}{}{3s-r-1}{r}_q \nonumber \\
		&=\sum_{s\ge1}
		\frac{q^{3s^2+bs}\mathsf A_{3s-1}}{(q^3;q^3)_s}
		=\mathscr T_{\mathsf A}(q^b). \label{h-53}
	\end{align}
	Here the $s=0$ term in $\mathcal T_{1,b}^{(-1)}$
	vanishes by the convention for Gaussian coefficients,
	and the corresponding term in $\mathscr T_{\mathsf A}$
	vanishes because $\mathsf A_{-1}=0$.
	Similarly,
	\begin{align}
		\mathcal T_{0,2}^{(0)}
		&=\sum_{s\ge0}\frac{q^{3s^2+2s}}{(q^3;q^3)_s}
		\sum_{r\ge0}q^{r^2-3sr}
		\genfrac{[}{]}{0pt}{}{3s-r}{r}_q \nonumber \\
		&=\sum_{s\ge0}
		\frac{q^{3s^2+2s}\mathsf A_{3s}}{(q^3;q^3)_s}
		=\mathscr S_{\mathsf A}(q^2). \label{h-54}
	\end{align}
	Identity \eqref{L1} follows from \eqref{h-52}, \eqref{h-53}
	 and \eqref{h-54}.

By \eqref{Tdef},
	\[
	\begin{aligned}
		q\mathcal T_{-2,4}^{(2)}
		&=q\sum_{s\ge0}\frac{q^{3s^2+4s}}{(q^3;q^3)_s}
		\sum_{r\ge0}q^{r^2-(3s+2)r}
		\genfrac{[}{]}{0pt}{}{3s-r+2}{r}_q\\
		&=q\sum_{s\ge0}
		\frac{q^{3s^2+4s}\mathsf A_{3s+2}}{(q^3;q^3)_s}.
	\end{aligned}
	\]
	Applying the recurrence of $\mathsf A_m$ 
	at $m=3s+2$ and $m=3s+1$, respectively, yields
	\[
	\begin{aligned}
		\mathsf A_{3s+2}
		&=\mathsf A_{3s+1}+q^{-3s-1}\mathsf A_{3s}\\
		&=(1+q^{-3s-1})\mathsf A_{3s}
		+q^{-3s}\mathsf A_{3s-1}.
	\end{aligned}
	\]
	This also holds at $s=0$, using
	$\mathsf A_{-1}=0$ and $\mathsf A_0=\mathsf A_1=1$.
	Therefore,
	\begin{align}
		q\mathcal T_{-2,4}^{(2)}
		&=\sum_{s\ge0}
		\frac{q^{3s^2+4s+1}\mathsf A_{3s}}{(q^3;q^3)_s}
		+\sum_{s\ge0}
		\frac{q^{3s^2+s}\mathsf A_{3s}}{(q^3;q^3)_s}+\sum_{s\ge0}
		\frac{q^{3s^2+s+1}\mathsf A_{3s-1}}{(q^3;q^3)_s} \nonumber \\
		&=q\mathscr S_{\mathsf A}(q^4)
		+\mathscr S_{\mathsf A}(q)
		+q\mathscr T_{\mathsf A}(q).\label{h-55}
	\end{align}
	Setting $b=-2$ in \eqref{h-53} yields 
	\begin{align}
		\mathcal T_{1,-2}^{(-1)}
		=\mathscr T_{\mathsf A}(q^{-2}). \label{h-56}
	\end{align} 
	Identity \eqref{L2} follows from \eqref{h-52}, \eqref{h-55}
and \eqref{h-56}.  
\end{proof}

Define
\begin{equation}\label{eq:cyclo-airy-series}
\begin{aligned}
\mathscr U_{1,\omega}
 &:=\mathscr T_u(q^{-1})-\omega\mathscr S_u(q^2),\\
\mathscr V_{1,\omega}
 &:=\mathscr T_v(q^{-1})-\omega\mathscr S_v(q^2),\\
\mathscr U_{2,\omega}
 &:=\mathscr T_u(q^{-2})
  +\omega^2\{\mathscr S_u(q)+q\mathscr S_u(q^4)
                         +q\mathscr T_u(q)\},\\
\mathscr V_{2,\omega}
 &:=\mathscr T_v(q^{-2})
  +\omega^2\{\mathscr S_v(q)+q\mathscr S_v(q^4)
                         +q\mathscr T_v(q)\}.
\end{aligned}
\end{equation}
Termwise substitution of \eqref{ainv} into \eqref{L1}--\eqref{L2}
gives
\begin{equation}\label{eq:cyclo-A-combination}
\mathcal K_{i,\omega}
=\frac{\mathsf C_-(\rho)\mathscr U_{i,\omega}
       +\mathsf C_+(\rho)\mathscr V_{i,\omega}}2,
\qquad i=1,2.
\end{equation}

\begin{proposition}\label{twistedconnection}
The series in \eqref{eq:cyclo-airy-series} satisfy
\begin{equation}\label{eq:cyclo-system}
\begin{pmatrix}\tau_1&\tau_2\\\tau_3&\tau_4\end{pmatrix}
\binom{\mathscr U_{i,\omega}}{\mathscr V_{i,\omega}}
=2\mathcal N_{i,\omega}\binom{J_1}{\omega^2J_3},
\qquad i=1,2,
\end{equation}
where $\mathcal N_{1,\omega},\mathcal N_{2,\omega}$ are defined in
\eqref{dual1}--\eqref{dual2}.
\end{proposition}

\begin{proof} We prove the two rows of the asserted matrix identity separately.
	
	For the first row with $i=1$, the definitions of
	$\mathscr U_{1,\omega}$ and $\mathscr V_{1,\omega}$,
	together with \eqref{firstS} and \eqref{firstT}, give
	\begin{align}
		 \tau_1\mathscr U_{1,\omega}
		+\tau_2\mathscr V_{1,\omega}&=
		\tau_1\mathscr T_u(q^{-1})
		+\tau_2\mathscr T_v(q^{-1})
		-\omega\bigl\{
		\tau_1\mathscr S_u(q^2)
		+\tau_2\mathscr S_v(q^2)
		\bigr\} \nonumber \\
		&\quad=2J_1\bigl\{
		\N(1,-1)-\omega\N(0,2)
		\bigr\} \nonumber \\
		&\quad=2J_1\mathcal N_{1,\omega}. \label{h-61}
	\end{align}
	For $i=2$, the same two identities yield
	\begin{align*}
		\frac{\tau_1\mathscr U_{2,\omega}
			+\tau_2\mathscr V_{2,\omega}}{2J_1}
		&=\N(1,-2) +\omega^2\bigl\{
		\N(0,1)+q\N(0,4)+q\N(1,1)
		\bigr\}. 
	\end{align*}
	Applying \eqref{cont1} at $(a,b)=(-2,4)$ and $(-1,4)$,
	respectively, gives
	\[
	\begin{aligned}
		\N(-2,4)-\N(-1,4)&=q^{-1}\N(0,1),\quad 
		\N(-1,4)-\N(0,4) =\N(1,1).
	\end{aligned}
	\]
	Multiplying their sum by $q$, we obtain
	\[
	q\N(-2,4)
	=\N(0,1)+q\N(0,4)+q\N(1,1).
	\]
	Consequently,
	\begin{align}\label{h-62}
		\frac{\tau_1\mathscr U_{2,\omega}
			+\tau_2\mathscr V_{2,\omega}}{2J_1}
		&=\N(1,-2)+\omega^2q\N(-2,4)=\mathcal N_{2,\omega}.
	\end{align}
Identities \eqref{h-61}
 and \eqref{h-62} prove  the first row for both values of $i$.
	
	For the second row, it follows  from \eqref{h-72}
	 and \eqref{h-71} that
	\[
	(\mathsf c_1,\mathsf d_1)
	=(\mathsf c_2,\mathsf d_2)
	=(\tau_3,\tau_4).
	\]
	The theta shift relations
	\[
	\mathsf c_{k+3}=q^{-2k}\mathsf c_k,
	\qquad
	\mathsf d_{k+3}=q^{-2k}\mathsf d_k
	\]
	therefore imply
	\[
	\begin{aligned}
		(\mathsf c_{-1},\mathsf d_{-1})
		&=q^{-2}(\tau_3,\tau_4),\qquad 
		(\mathsf c_{-2},\mathsf d_{-2})
		 =q^{-4}(\tau_3,\tau_4).
	\end{aligned}
	\]
	Using \eqref{shiftT} at $k=-1$ and $k=-2$, respectively,
	we find
	\[
	\begin{aligned}
		\tau_3\mathscr T_u(q^{-1})
		+\tau_4\mathscr T_v(q^{-1})
		&=q^2\bigl\{
		\mathsf c_{-1}\mathscr T_u(q^{-1})
		+\mathsf d_{-1}\mathscr T_v(q^{-1})
		\bigr\}\\
		&=-2J_3\N(-1,2),\\[3pt]
		\tau_3\mathscr T_u(q^{-2})
		+\tau_4\mathscr T_v(q^{-2})
		&=q^4\bigl\{
		\mathsf c_{-2}\mathscr T_u(q^{-2})
		+\mathsf d_{-2}\mathscr T_v(q^{-2})
		\bigr\}\\
		&=-2qJ_3\N(-2,4).
	\end{aligned}
	\]
	
	For $i=1$, \eqref{secondS} at $k=2$ also gives
	\[
	\tau_3\mathscr S_u(q^2)
	+\tau_4\mathscr S_v(q^2)
	=2J_3\N(1,-1).
	\]
	Hence
	\[
	\frac{\tau_3\mathscr U_{1,\omega}
		+\tau_4\mathscr V_{1,\omega}}{2J_3}
	=-\N(-1,2)-\omega\N(1,-1).
	\]
	By \eqref{cont1} at $(a,b)=(-1,2)$,
	\[
	\N(-1,2)=\N(0,2)+\N(1,-1).
	\]
	Using $1+\omega+\omega^2=0$, $\omega^3=1$ and the above identity, we obtain
	\[
	\begin{aligned}
		-\N(-1,2)-\omega\N(1,-1)
		&=-\N(0,2)-(1+\omega)\N(1,-1)=\omega^2\mathcal N_{1,\omega}.
	\end{aligned}
	\]
	Thus
	\[
	\tau_3\mathscr U_{1,\omega}
	+\tau_4\mathscr V_{1,\omega}
	=2\omega^2J_3\mathcal N_{1,\omega}.
	\]
	
	For $i=2$, applying \eqref{secondS}, \eqref{shiftS},
	and \eqref{shiftT} at $k=1$ gives, respectively,
	\[
	\begin{aligned}
		\tau_3\mathscr S_u(q)+\tau_4\mathscr S_v(q)
		&=2J_3\N(0,1),\\
		\tau_3\mathscr S_u(q^4)+\tau_4\mathscr S_v(q^4)
		&=2q^2J_3\N(3,-5),\\
		\tau_3\mathscr T_u(q)+\tau_4\mathscr T_v(q)
		&=-2J_3\N(1,-2).
	\end{aligned}
	\]
	Combining these evaluations with the formula for
	$\tau_3\mathscr T_u(q^{-2})
	+\tau_4\mathscr T_v(q^{-2})$ yields
	\[
	\begin{aligned}
		\frac{\tau_3\mathscr U_{2,\omega}
			+\tau_4\mathscr V_{2,\omega}}{2J_3}
		&=-q\N(-2,4)+\omega^2\bigl\{
		\N(0,1)+q^3\N(3,-5)-q\N(1,-2)
		\bigr\}.
	\end{aligned}
	\]
	To simplify the expression in braces, apply \eqref{cont1}
	at $(a,b)=(0,1)$ and $(1,-2)$:
	\[
	\begin{aligned}
		\N(0,1)-\N(1,1)&=q\N(2,-2),\\
		\N(1,-2)-\N(2,-2)&=q^2\N(3,-5).
	\end{aligned}
	\]
	It follows that
	\[
	\begin{aligned}
		&\N(0,1)+q^3\N(3,-5)-q\N(1,-2)
	=\N(0,1)-q\N(2,-2)
		=\N(1,1).
	\end{aligned}
	\]
	Moreover, \eqref{cont2} at $(a,b)=(1,-2)$ gives
	\[
	\N(1,-2)-\N(1,1)=q\N(-2,4).
	\]
	Therefore,
	\[
	\begin{aligned}
		\frac{\tau_3\mathscr U_{2,\omega}
			+\tau_4\mathscr V_{2,\omega}}{2J_3}
		&=-q\N(-2,4)+\omega^2\N(1,1)\\
		&=-q\N(-2,4)
		+\omega^2\bigl\{\N(1,-2)-q\N(-2,4)\bigr\}\\
		&=\omega^2\mathcal N_{2,\omega}.
	\end{aligned}
	\]
	This proves the second row for $i=2$, and hence completes
	the proof of the matrix identity.
\end{proof}

\begin{proof}[Proof of Theorem~\ref{Kthm}]
Apply Proposition~\ref{prop:multiplier} to \eqref{eq:cyclo-system}
with
\[
 (\mathscr U,\mathscr V)
 =(\mathscr U_{i,\omega},\mathscr V_{i,\omega}),\qquad
 \nu=\mathcal N_{i,\omega},\qquad \lambda=\omega^2.
\]
Together with \eqref{eq:cyclo-A-combination}, this gives
\begin{equation}\label{universal}
\begin{aligned}
\mathcal K_{i,\omega}
 &=\mathcal N_{i,\omega}\mathcal M(\omega^2)\\
 &=\mathcal N_{i,\omega}\frac{J_1}{J_3^2}
 \left(J_9G(q^9)+\frac{\omega^2-1}{3}J_1H(q)\right),
 \qquad i=1,2.
\end{aligned}
\end{equation}
Since $\mathcal E_3=J_3/J_1$ and
\[
 (1-\omega)(\omega^2-1)
 =\omega^2+\omega-2=-3,
\]
substitution of \eqref{dual1}--\eqref{dual2} into
\eqref{universal} yields
\begin{align}
\mathcal K_{1,\omega}
 &=\frac{(\omega^2q^2,\omega q^3;q^3)_\infty}{J_3}
 \left((1-\omega)J_9G(q^9)-J_1H(q)\right),\label{Kfactor1}\\
\mathcal K_{2,\omega}
 &=\frac{(\omega q,\omega^2q^3;q^3)_\infty}{J_3}
 \left((1-\omega)J_9G(q^9)-J_1H(q)\right).\label{Kfactor2}
\end{align}

We next evaluate their common factor. Write
\[
 j(z;p):=(z,p/z,p;p)_\infty
 =\sum_{n\in\mathbb Z}(-1)^np^{n(n-1)/2}z^n.
\]
The product formulas for $G$ and $H$ give
\[
\begin{aligned}
 J_9G(q^9)
 &=\frac{(q^9;q^9)_\infty}{(q^9,q^{36};q^{45})_\infty}
   =(q^{18},q^{27},q^{45};q^{45})_\infty
   =j(q^{18};q^{45}),\\
 J_1H(q)
 &=\frac{(q;q)_\infty}{(q^2,q^3;q^5)_\infty}
   =(q,q^4,q^5;q^5)_\infty
   =j(q;q^5).
\end{aligned}
\]
In the bilateral expansion of $j(q;q^5)$, the terms with
$n=3m$ are
\[
 \sum_{m\in\mathbb Z}(-1)^{3m}q^{(45m^2-9m)/2}
 =\sum_{m\in\mathbb Z}(-1)^m
   (q^{45})^{m(m-1)/2}q^{18m}
 =j(q^{18};q^{45}).
\]
Consequently,
\begin{align*}
 (1-\omega)J_9G(q^9)-J_1H(q)
 &=\sum_{n\in\mathbb Z}(-1)^nq^{(5n^2-3n)/2}
   \bigl((1-\omega)\mathbf s(n) -1\bigr)\nonumber\\
   &= - \sum_{n\in\mathbb Z}(-1)^nq^{(5n^2-3n)/2}
  \omega^{2n^2+1}
\end{align*}
Here $\mathbf s(n)$ is defined by   
\[
\mathbf s(n):=
\begin{cases}
	1,& {\rm if}\ 3\mid n,\\
	0,& {\rm if}\  3\nmid n,
\end{cases}
\]
and  for every integer $n$,
\[
 (1-\omega)\mathbf s(n) -1
 =- \omega^{2n^2+1}.
\]
Also,
\[
 2n^2-\left(\frac{n(n-1)}2+2n\right)
 =\frac{3n(n-1)}2\in3\mathbb Z.
\]
It follows that
\begin{equation}\label{Bjacobi}
\begin{aligned}
 (1-\omega)J_9G(q^9)-J_1H(q)
 &=-\omega\sum_{n\in\mathbb Z}(-1)^n
   (\omega q^5)^{n(n-1)/2}(\omega^2q)^n\\
 &=-\omega\,j(\omega^2q;\omega q^5).
\end{aligned}
\end{equation}

Note that 
\begin{align}\label{th-81}
 j(\omega^2q;\omega q^5)
 &=(q^6,q^9,q^{15};q^{15})_\infty (\omega^2q,\omega^2q^4,\omega^2q^{10},
              \omega q^5,\omega q^{11},\omega q^{14};q^{15})_\infty.
\end{align}
Furthermore,
\begin{equation}\label{normalcancel}
\begin{aligned}
 \frac{(q^6,q^9,q^{15};q^{15})_\infty}{J_3}
 & =\frac1{(q^3,q^{12},q^{18},q^{27},q^{33},q^{42};q^{45})_\infty}.
\end{aligned}
\end{equation}
Substituting this and \eqref{Bjacobi}--\eqref{normalcancel}
into \eqref{Kfactor1}, and cancelling the factors with exponents
$33$ and $42$ at base $q^{45}$, yields
\[
\mathcal K_{1,\omega}
 =-\omega\,
 \frac{(q^{15};q^{45})_\infty(\omega q^3;q^3)_\infty
 (\omega^2q,\omega^2q^2,\omega^2q^4,\omega^2q^8,\omega^2q^{10};q^{15})_\infty}
 {(q^5,q^{11},q^{14};q^{15})_\infty
  (q^3,q^{12},q^{18},q^{27};q^{45})_\infty}.
\]
This proves \eqref{K1} for $\zeta=\omega$.

Substitution \eqref{Bjacobi}--\eqref{normalcancel} into \eqref{Kfactor2} and cancellation of the factors
with exponents $3$ and $12$ at base $q^{45}$ give
\[
\mathcal K_{2,\omega}
 =-\omega\,
 \frac{(q^{30};q^{45})_\infty(\omega^2q^3;q^3)_\infty
 (\omega q^5,\omega q^7,\omega q^{11},\omega q^{13},\omega q^{14};q^{15})_\infty}
 {(q,q^4,q^{10};q^{15})_\infty
  (q^{18},q^{27},q^{33},q^{42};q^{45})_\infty}.
\]

This proves \eqref{K2} for \(\zeta=\omega\).  Applying the
coefficient-field conjugation
\(\omega\mapsto\omega^2\), with \(q\) fixed, proves the other two
identities for \(0<q<1\).

We finally extend the identities to complex \(q\).  For each fixed
\(\zeta\in\{\omega,\omega^2\}\), both sides of \eqref{K1} and
\eqref{K2} are holomorphic functions of \(q\) on \(|q|<1\).  The
defining series converge locally uniformly in this disk, while the
infinite products converge locally uniformly and their denominator
factors do not vanish there.  Although the auxiliary variable
\(\rho=q^{1/4}\) was used in the Airy-function calculation, the final
expressions contain only integral powers of \(q\) and hence are
independent of the choice of a fourth root.  Since the identities hold
on the real interval \(0<q<1\), the identity theorem shows that they
hold for every \(|q|<1\).
\end{proof}

\section{Concluding remarks}\label{sec:conclusion}

We have proved all eighteen conjectural identities using a unified
method that combines parameterized Nahm-sum evaluations, Schur-type
finite sums, divided differences, \(q\)-Airy functions, and
theta-function identities. For the six symmetric reflections, a Schur
decomposition gives one relation between each pair of sums, while
finite rational identities and residue calculations give a second
relation. The corresponding dual Nahm-sum evaluations
from~\cite{Xia} then determine the reflected sums through nonsingular
two-by-two linear systems.

For the asymmetric family, we express the Schur sequences in a
\(q\)-Airy basis and use Airy--theta connection formulas to derive a
common multiplier identity. The same basis and multiplier apply to the
cyclotomic family after specialization at a primitive cube root of
unity. The Hickerson and cyclotomic Nahm-sum evaluations follow from
low-order divided differences of   the parameterized families established in \cite{Xia}. The
remaining product simplifications are obtained from theta-function
dissections and, where necessary, generalized eta-product identities
verified with the Frye--Garvan package~\cite{FG}. Thus, the apparently
different identities are governed by the same underlying recurrences
and Nahm-sum structure.
 
 Several questions remain open. First, it would be useful to construct
 finite polynomial refinements of the eighteen identities and to
 determine whether their reflected sums admit natural partition
 interpretations or bijective proofs. Second, alternative derivations
 using Bailey pairs, constant-term identities, or direct
 recurrence-and-summation methods may reveal transformations valid for
 larger parameter families. In particular, the common difference
 equations \eqref{eq:ST1} and \eqref{eq:ST2} suggest a possible extension
 of the methods in~\cite{Wang12}. Finally, one may investigate
 $q$-reflections of identities of other moduli and of higher-rank Nahm
 sums, with the aim of distinguishing those whose reflected limits are
 modular products from those leading to partial theta, false theta, or
 mock modular functions.

\section*{Statements and Declarations}

\noindent{\bf Acknowledgments.}
This work was supported by the National Natural Science Foundation
of China (Grant No.~12371334) and the Qinglan Project of Jiangsu Province. 

\medskip
\noindent{\bf Competing Interests.}
The author declares no competing interests.

\medskip
\noindent\textbf{Data Availability}. 
Data sharing is not applicable to this article as no datasets were generated
or analyzed during the current study.

\medskip
\noindent\textbf{Declaration of AI usage.}
During the development of this work, the author used ChatGPT
(OpenAI) as a research-assistance tool to explore possible proof
strategies and to assist with language editing and proofreading.
All mathematical arguments presented in the paper were independently
verified, refined, and written by the author, who takes full
responsibility for the final content of the publication.

\end{document}